\documentclass{article}

    \PassOptionsToPackage{numbers, compress}{natbib}

\usepackage{graphicx}
 \usepackage[preprint]{neurips_2026}

\usepackage{wrapfig}
\usepackage[utf8]{inputenc} 
\usepackage[T1]{fontenc}    

\usepackage{url}            
\usepackage{booktabs}       
\usepackage{amsfonts}       
\usepackage{nicefrac}       
\usepackage{microtype}      
\usepackage{xcolor}         
\usepackage{amsmath,amssymb,amsfonts,amsthm,mathtools,bm}
\usepackage{array}
\usepackage{enumitem}
\usepackage{subcaption}

\usepackage{hyperref}

\newtheorem{theorem}{Theorem}[section]
\newtheorem{lemma}[theorem]{Lemma}
\newtheorem{proposition}[theorem]{Proposition}
\newtheorem{corollary}[theorem]{Corollary}
\newtheorem{definition}[theorem]{Definition}
\newtheorem{assumption}[theorem]{Assumption}
\newtheorem{remark}[theorem]{Remark}

\title{Continuous-Time Analysis for Minimax and Bilevel Problems}

\author{%
  Hyunwoo Lee \\
  KIAS\\
  \texttt{hyunwoolee@kias.re.kr} \\
  \And
  Jeongyeol Kwon \\
  Meta \\
  \texttt{sortingkwon@gmail.com} \\
  \And
  Dohyun Kwon\thanks{Corresponding author} \\
  Yonsei University / KIAS \\
  \texttt{dohyunkwon@yonsei.ac.kr} \\
}

\begin{document}

\maketitle

\begin{abstract}
We study single-loop gradient-flow dynamics for nested optimization, where the outer variable evolves while auxiliary variables track the inner solution map.
While existing analyses typically rely on problem- and condition-specific Lyapunov constructions, we propose, to our knowledge, the first unified Lyapunov template for continuous-time analysis that covers minimax, bilevel via a lifted penalty formulation, and min--min--max. 
Our proof is modular, built from reusable lemmas that yield a unified characterization of time-scale separation. This characterization bridges regimes from strong convexity/concavity to mere convexity through an error-bound condition, and produces explicit closed-form thresholds that avoid the coupled ratio conditions common in discrete-time analyses. We further compare the penalty dynamics with the ideal hyper-gradient flow, derive a finite-time tracking bound, and discuss an Euler one-step analogue; hypercleaning diagnostics show that the predicted relative time-scale regions remain visible under stable forward-Euler discretization.

\end{abstract}

\section{Introduction}
Modern large-scale machine learning increasingly relies on nested optimization as a core design principle. Canonical examples include hyperparameter optimization, meta-learning, ~\cite{maclaurin2015gradient,domke2012generic,pedregosa2016hyperparameter,franceschi2018bilevel,shaban2019truncated}, adversarial robustness \cite{goodfellow2020generative}, game theory \cite{stackelberg1952theory} and reinforcement learning \cite{konda1999actor, sutton2018reinforcement}. In this work, we focus on the continuous-time dynamics of bilevel optimization~\cite{colson2007overview}, which arise in many important applications with inherently hierarchical two-level optimization structures, formally described as the following:
\begin{equation}\label{eq:problem2}
\tag{\textbf{P2}}
\min_{x\in X}\quad  F(x)\ :=\ f\bigl(x,y^{*}(x)\bigr) 
\quad \text{ s.t. }\quad  y^{*}(x)\ \in\ \arg\min_{y\in Y}\ g(x,y),
\end{equation}
where $f,g:\mathbb{R}^{d_x}\times\mathbb{R}^{d_y}\to\mathbb{R}$ are continuously differentiable. The minimax   is obtained by setting
\begin{equation}\label{eq:P1}
\tag{\textbf{P1}}
g(x,y)\ :=\ -f(x,y).
\end{equation}
Then $y^{*}(x)\in\arg\max_{y\in Y}f(x,y)$, and hence $F(x) = \max_{y\in Y}f(x,y).$ If $g(x,\cdot)$ is strongly convex for every $x\in X$, then the inner minimizer $y^{*}(x)$ is unique and hence well-defined for all $x$ in \ref{eq:problem2}. 

Evaluating $\nabla F(x)$ requires repeatedly solving the inner problem and often involves computing an implicit hyper-gradient involving second-order information. 
A widely used alternative is the single-loop principle~\cite{kwon2023fully, kwon2023penalty}: evolve the outer variable while auxiliary variables track the inner optimizer on-the-fly. In continuous time, this leads naturally to coupled gradient flows~\cite{su2016differential, sharrock2023two} with explicit time-scale parameters controlling how fast the inner dynamics track their instantaneous targets relative to the $x$-dynamics.
Recent work has started to clarify stability and convergence of such coupled dynamics via Lyapunov and time-scale separation viewpoints~\cite{chen2022single, hong2023two, doan2022convergence, lin2025two}.
However, prior analyses are largely formulation-specific and discrete-time, leading to bespoke Lyapunov/absorption arguments and intertwined step-size ratio conditions whose implied time-scale thresholds are often algebraically intricate. As a result, prior case-by-case analyses often obscure the reusable structure and the role of penalty-induced terms in the time-scale absorption step. In contrast, our continuous-time normal form separates deviation-to-residual conversion from time-scale absorption, yielding closed-form threshold conditions and clarifying how the lifted penalty affects the absorption mechanism.

Our main focus is bilevel optimization, but the standard penalty surrogate has an inherently saddle-type structure.
Let $g^*(x):=\min_{z} g(x,z)$ and define
\[
\mathcal L_\lambda(x,y)=f(x,y)+\lambda\bigl(g(x,y)-g^*(x)\bigr),\
\mathcal L_\lambda^*(x)=\min_y \mathcal L_\lambda(x,y).
\]

To avoid evaluating $g^*(x)$ explicitly, we introduce a tracker $z$ and lift the penalty as
\[
E_\lambda(x,y,z):=f(x,y)+\lambda\bigl(g(x,y)-g(x,z)\bigr),
\]
so that $\min_z g(x,z)=g^*(x)$ is recovered by minimizing over $z$.
This yields the lifted min--min--max formulation
$\min_x \min_y \max_z \; E_\lambda(x,y,z)$,
which places bilevel penalty dynamics in the same structural family as minimax and makes the tracking mechanism explicit.
It also cleanly exposes how penalty-induced stiffness enters the Lyapunov absorption step and how the corresponding threshold conditions should be interpreted. Appendix~\ref{sec:minminmax} studies the genuine min--min--max problem
\begin{equation}\tag{\textbf{P3}}\label{eq:P3}
\min_x\min_y\max_z \; E(x,y,z),
\end{equation}
where $E(x,y,z):=f(x,y)-g(x,z)$, and no penalty parameter is involved.
This serves as a calibration case: the same Lyapunov template and time-scale absorption rule apply verbatim, yielding thresholds that are intrinsic to the saddle-type coupling.

\begin{wraptable}{r}{0.4\textwidth}
\centering
\small
\setlength{\tabcolsep}{3.5pt}
\resizebox{\linewidth}{!}{%
\begin{tabular}{@{}lccc@{}}
\toprule
Problem & $\Phi(x)$ & $H_y$ & $H_z$\\
\midrule
Minimax   
& $F$                    
& $F-f$                         
& $0$\\
Bilevel   
& $\mathcal L_\lambda^*$ 
& $\mathcal L_\lambda-\mathcal L_\lambda^*$ 
& $g-g^*$\\
Min-min-max 
& $\mathcal L$           
& $f-f^*$                   
& $g-g^*$\\
\bottomrule
\end{tabular}
}
\caption{Instantiation of the unified Lyapunov blueprint. Here $F=\max_y f(x,y)$, $f=f(x,y)$, $g=g(x,z)$, and $f^*,g^*$ denote inner optima.}
\label{tab:instantiation}
\end{wraptable}
We develop a unified continuous-time analysis for minimax, bilevel, and min--min--max problems. All three are handled by the same value-level Lyapunov template
\[
W:=\Phi(x)+\alpha H_y(\cdot)+\beta H_z(\cdot),
\]
where $\Phi$ is the target outer value and $H_y,H_z\ge 0$ measure tracking gaps.
Differentiating $W$ yields outer descent plus tracking-induced cross terms. Our
proof separates their control into two reusable modules: a deviation-to-residual
conversion, supplied by strong convexity/concavity or an error-bound interface,
and a closed-form time-scale absorption rule. Thus, once
$(\Phi,H_y,H_z)$ are instantiated, the required thresholds follow from the same normal form.

While unified analyses for nested optimization exist
~\cite{chen2021alset,tarzanagh2022fednest}, and Lyapunov analyses are known for coupled saddle dynamics~\cite{feijer2010pd,cherukuri2017saddle}, we are not aware of a value-function-level continuous-time blueprint that yields plug-in time-scale thresholds across minimax envelopes and lifted-penalty bilevel dynamics.

The same modular view clarifies weaker inner geometry. Under strong
convexity/concavity, the conversion constant scales like $1/\mu$; under a local
EB/PL/QG-type condition, it is replaced by an error-bound constant $\kappa$.
Consequently, the thresholds deteriorate through the same $C^2$ dependence,
making explicit how the required time-scale separation changes when strong
convexity is relaxed. This geometry--threshold relation is often less explicit
inside algorithm-specific step-size conditions in discrete-time two-time-scale
analyses~\cite{chen2022single,hong2023two}. We further corroborate this geometric interpretation on Wine and Digits
hypercleaning tasks, a standard bilevel benchmark~\cite{franceschi2018bilevel,shaban2019truncated}.
Under stable forward-Euler discretization, reducing the lower-level
regularization weakens the empirical inner curvature, increases the
deviation-to-residual conversion constant, and shifts the predicted thresholds upward. The phase diagrams show the same above-threshold region as the favorable time-scale regime. This supports our central message: the EB interface exposes how weaker inner geometry changes the conversion constant and shifts the time-scale thresholds, a relationship that is often hidden inside
algorithm-specific step-size conditions in discrete-time two-timescale analyses~\cite{hong2023two}. This connects EB/PL/QG-type
regularity~\cite{liao2024error} directly to the Lyapunov absorption thresholds.

Finally, we compare the penalty-based three-variable dynamics with the ideal bilevel gradient flow
$\dot w=-\nabla F(w)$ and establish a finite-time trajectory tracking bound.
Our result controls $\sup_{t\in[0,T]}\|x(t)-w(t)\|$ and cleanly separates (i) the approximation bias induced by the penalty parameter $\lambda$
from (ii) the residual tracking errors governed by the time-scale separations $\delta>\delta_0(\lambda)$ and $\eta>\eta_0$.
This provides an explicit quantitative link between single-loop tracking dynamics and the hyper-gradient flow they approximate. 

\paragraph{Related Work}
Minimax and bilevel optimization have largely separate literatures: minimax work often studies saddle-point/game dynamics, while bilevel methods typically use implicit differentiation, unrolling, or single-loop approximations \cite{cherukuri2017saddle,ghadimi2018approximation,yang2021provably,hong2020two,khanduri2021near}. Across these settings, the two-time-scale principle provides a common viewpoint from stochastic approximation and dynamical systems \cite{borkar1997twotimescale,feijer2010pd,cherukuri2017saddle}. Existing discrete-time analyses often encode stability through algorithm-specific step-size conditions; our contribution is to isolate a value-function-level continuous-time Lyapunov template that yields plug-in thresholds and makes their dependence on inner geometry explicit.

\section{Preliminaries}
We state several assumptions on \ref{eq:P1} and \ref{eq:problem2} to specify the problem class of interest.
In both settings, the outer objective has the composite form $F(x)=f(x,y^*(x))$ defined through an inner optimizer $y^*(x)$.
We focus on well-conditioned regimes where $y^*(x)$ is uniquely defined and sufficiently regular, so that single-loop tracking dynamics admit clean Lyapunov bounds.
\begin{assumption}\label{ass:1}
We consider \ref{eq:P1} with a continuously differentiable function $f$. There exist constants $\mu>0$ and $L_{xy}\ge 0$ such that the following hold.
\begin{enumerate}
  \item For every $x$, $f(x,\cdot)$ is $\mu$-strongly concave. 
  \item $\nabla_x f(x,y)$ is $L_{xy}$-Lipschitz in $y$, uniformly over $x$.
  \item $F_{\inf}:=\inf_{x\in X} F(x) > -\infty$.
\end{enumerate}
\end{assumption}

\begin{assumption}\label{ass:2}
We consider \ref{eq:problem2} with continuously differentiable functions $f$ and $g$.
There exist constants $\mu>0$, $\rho\ge 0$, and $L_{fx},L_{gx}\ge 0$ such that:
\begin{enumerate}
  \item For every $x$, $g(x,\cdot)$ is $\mu$-strongly convex.
  \item $\nabla_x f(x,y)$ is $L_{fx}$-Lipschitz in $y$, uniformly over $x$.
  \item $\nabla_x g(x,y)$ is $L_{gx}$-Lipschitz in $y$, uniformly over $x$.
  \item $\mathcal{L}_{\lambda,\inf}^{*}:=\inf_{x\in\mathbb{R}^{d_x}} \mathcal L_\lambda^{*}(x) > -\infty.$
  \item For every $x$, $f(x,\cdot)$ is $\rho$--weakly convex.
\end{enumerate}
\end{assumption}

The goal is to find a stationary solution defined as the following:
\begin{definition}[$\varepsilon$-stationarity]\label{def:eps-stationary}
Let $F:\mathbb{R}^{d_x}\to\mathbb{R}$ be differentiable and let $\varepsilon>0$.
We say that a point $x\in\mathbb{R}^{d_x}$ is $\varepsilon$-stationary for $F$
if $\|\nabla F(x)\|^{2}\ \le\ \varepsilon$.
\end{definition}

\paragraph{Notation.}
We use $\|\cdot\|$ and $\langle\cdot,\cdot\rangle$ for the Euclidean norm and inner product.
For a differentiable function $h(x,y)$, $\nabla_x h$ and $\nabla_y h$ denote partial gradients.
\section{Unified Continuous Time Framework for Nested Optimization}\label{sec:unified}
This paper investigates minimax~\ref{eq:P1} and bilevel~\ref{eq:problem2} nested problem through a unified Lyapunov perspective. In addition, we consider analysis for the auxiliary min--min--max problem in Appendix~\ref{sec:minminmax}.
All three share a common obstacle: the outer descent direction depends on an inner optimizer that varies with $x$, which makes the outer-gradient dynamics hard to compute directly.
Therefore, we focus on single-loop coupled dynamics that track inner solutions along the trajectory.
In continuous time, sufficient time-scale separation makes tracking errors dissipative, so the cross terms can be absorbed in a Lyapunov inequality and descent of the relevant outer value function follows.
Building on this, we introduce a common Lyapunov construction that augments the outer value term with nonnegative tracking gaps, and we state two reusable lemmas that serve as shared analysis modules across these settings.

Our ideal outer-gradient flow is given by
\begin{equation}
\label{eq:bi-gf}
    \dot x=-\nabla F(x).
\end{equation}
In \ref{eq:P1} and \ref{eq:problem2}, the outer objective takes the composite form
$F(x)=f\bigl(x,y^*(x)\bigr)$, where $y^*(x)$ is defined implicitly as the optimizer of an inner problem.
As a result, evaluating $\nabla F(x)$ typically requires nontrivial inner computation along the trajectory.
In particular, even for \ref{eq:P1} one has $\nabla F(x)=\nabla_x f\bigl(x,y^*(x)\bigr)$, which depends on the
current inner maximizer $y^*(x)$.
For \ref{eq:problem2}, the dependence is more explicit: $\nabla F(x)$ is an implicit hyper-gradient involving
second-order information,
\begin{equation}\label{eq:hypergrad}
\begin{aligned}
\nabla F(x)
&=\nabla_x f\bigl(x,y^{*}(x)\bigr)-\nabla^2_{xy}g\bigl(x,y^{*}(x)\bigr)\times\bigl(\nabla^2_{yy}g\bigl(x,y^{*}(x)\bigr)\bigr)^{-1}\nabla_y f\bigl(x,y^{*}(x)\bigr).
\end{aligned}
\end{equation} 
Consequently, computing $\nabla F(x)$ in the bilevel setting entails solving the lower-level problem to obtain $y^*(x)$ and solving a linear system involving $\nabla^2_{yy}g$.
To avoid repeatedly resolving inner optimizers and explicitly forming such hyper-gradients, we adopt single-loop first-order coupled dynamics as in \cite{kwon2023fully}. Accordingly, we introduce the coupled flows below together with their time-scale parameters.

\subsection{Coupled Flows and Time Scale Parameters}\label{subsec:unified-flows}

We specify the above single-loop principle by evolving the outer variable together with inner variables, and, when needed, additional tracking variables.
The resulting continuous time systems are parameterized by explicit time-scale factors that control how fast the inner dynamics track their instantaneous targets relative to the $x$-dynamics.

For \ref{eq:P1}, we consider the two-variable gradient flow
\begin{equation}\label{eq:flow-minimax-unified}
\dot x(t) = -\nabla_x f\bigl(x(t),y(t)\bigr),\
\dot y(t) = \gamma\,\nabla_y f\bigl(x(t),y(t)\bigr).
\end{equation}
where $\gamma>0$ scales the ascent speed of $y(t)$.
    When $\gamma$ is sufficiently large, $y(t)$ tracks the maximizer of $f(x(t),\cdot)$ fast enough for the Lyapunov analysis, to yield an ergodic stationarity bound. 

For \ref{eq:problem2}, fix $\lambda>0$ and define the lower-level function $g^*(x):=\min_y g(x,y)$ and the penalized objective
\begin{equation}\label{eq:bilevel-value-f}
\begin{aligned}
\mathcal L_\lambda(x,y)
&:= f(x,y)+\lambda\bigl(g(x,y)-g^*(x)\bigr),\
\mathcal L_\lambda^*(x):= \min_y \mathcal L_\lambda(x,y).
\end{aligned}
\end{equation}
To avoid explicit evaluation of $g^*(x)$ along the trajectory, we introduce a tracker $z(t)$ via the following reformulation:
\begin{equation}\label{eq:Elambda-unified}
E_\lambda(x,y,z):=f(x,y)+\lambda\bigl(g(x,y)-g(x,z)\bigr).
\end{equation}
In this reformulation, maximizing over $z$ replaces $g^*(x)=\min_z g(x,z)$ by an explicit tracker, so that $z(t)$ is driven toward $\arg\min_z g(x(t),z)$ and $y(t)$ descends on $f(x(t),y)+\lambda g(x(t),y)$.
We study the three-variable gradient flow with time-scale parameters $\delta>0$ and $\eta>0$,
\begin{equation}\label{eq:flow-bilevel-unified}
\begin{aligned}
\dot x = -\nabla_x E_\lambda(x,y,z),\
\dot y = -\delta\,\nabla_y E_\lambda(x,y,z),\
\dot z = \eta\,\nabla_z E_\lambda(x,y,z).
\end{aligned}
\end{equation}


\subsection{A Unified Lyapunov Toolkit}\label{subsec:unified-lyap}
We present a common Lyapunov construction for minimax and bilevel problems, and later instantiate the same construction for the auxiliary min--min--max problem in Appendix~\ref{sec:minminmax}. In all cases, the Lyapunov function has the same form: an outer value term plus nonnegative tracking gaps. Its dissipation reduces to controlling cross terms caused by imperfect tracking, which we handle using two shared ingredients: a deviation bound and a time-scale threshold rule that absorbs these terms under sufficient separation.

These two lemmas are the key tools behind the ergodic $O(1/T)$ stationarity guarantees proved in Section~\ref{sec:main-results} for all three problems \ref{eq:P1}--\ref{eq:P31}.







\paragraph{A Unified Lyapunov Construction.}
Across all three settings, our Lyapunov function combines a target value function and nonnegative gap terms~(tracking errors):
\begin{equation}\label{eq:W-template}
W \;:=\; \Phi(x)\;+\;\alpha\,H_{y}(\cdot)\;+\;\beta\,H_{z}(\cdot),
\end{equation}
where $\Phi(x)$ denotes the target outer objective, and $H_{y}\ge 0$ and $H_{z}\ge 0$ are gap terms that quantify the tracking errors of the inner variables.
Along the coupled dynamics, differentiating $W$ yields a descent term for $\Phi$ plus cross terms involving tracking errors.
Choosing the time-scale parameters sufficiently large allows these cross terms to be absorbed, leading to $\dot W\le -c\,\|\nabla \Phi(x(t))\|^2$ and hence an ergodic $O(1/T)$ stationarity bound.

For \ref{eq:P1},
let $F(x):=\max_{y\in Y}f(x,y)$ and define
\[
\Phi(x):=F(x),\quad
H_y:=F(x)-f(x,y),\quad
H_z:= 0.
\]
Accordingly,
\[
W(x,y)\;=\;F(x)+\alpha\,H_y(x,y).
\]
For $\gamma$ above an explicit threshold, $\dot W$ controls $\|\nabla F(x(t))\|^2$, yielding the ergodic stationarity guarantee.

For \ref{eq:problem2},
recall $\mathcal L_\lambda:=f(x,y)+\lambda\bigl(g(x,y)-g^*(x)\bigr)$ and $\mathcal L_\lambda^*(x):=\min_y\mathcal L_\lambda(x,y)$.
Define $\Phi(x):=\mathcal L_\lambda^*(x),$
\[
H_y:=\mathcal L_\lambda(x,y)-\mathcal L_\lambda^*(x), \quad
H_z:=g(x,z)-g^*(x).
\]
Accordingly,
\[
W(x,y,z)\;=\;\mathcal L_\lambda^*(x)+\alpha\,H_y(x,y)+\beta\lambda\,H_z(x,z).
\]
If $\delta\ge\delta_0(\lambda)$ and $\eta\ge\eta_0$, then $\dot W$ yields an ergodic bound on $\|\nabla_x\mathcal L_\lambda^*(x(t))\|^2$. 
In the appendix~\ref{sec:minminmax}, we additionally consider 
\begin{equation}\label{eq:P31}
\tag{\textbf{P3}}
\min_x\min_y\max_z\,\bigl(f(x,y)-g(x,z)\bigr).
\end{equation}

\paragraph{Common Lemma}
Under the unified Lyapunov construction $W=\Phi(x)+\alpha H_y+\beta H_z$, differentiating $W$
along the coupled flow yields the desired descent term in $\|\nabla \Phi(x(t))\|^2$ together with cross terms induced by imperfect inner tracking. 
These cross terms are expressed through $x$--gradient deviations, which cannot be absorbed without relating them to quantities damped by the inner dynamics. We therefore isolate a reusable two-step proof skeleton that applies uniformly to \ref{eq:P1}--\ref{eq:P31}. First, Lemma~\ref{lem:deviation} converts each deviation $\|D_\cdot\|$ into a first-order residual
$\|r_\cdot\|$ of the corresponding inner problem, which serves as a stationarity measure for \ref{eq:P1}--\ref{eq:P31}.
Second, Lemma~\ref{lem:timescale} shows that once the inner time-scale parameters exceed explicit thresholds,
the residual dissipation dominates these cross terms, yielding $\dot W\le -c\,\|\nabla\Phi(x(t))\|^2$.
Integrating this inequality gives the common $O(1/T)$ ergodic stationarity bound for $\Phi$.

After instantiating $(\Phi,H_y,H_z)$ as in Table~\ref{tab:instantiation}, the remaining work in each problem
is to compute the constants appearing in Lemmas~\ref{lem:deviation}--\ref{lem:timescale} via
Lipschitz/strong-convexity parameters and a Young-type bound. 

\begin{lemma}\label{lem:deviation}
Fix $x\in\mathbb R^{d_x}$ and let $\psi(x,\cdot):\mathbb R^{d_u}\to\mathbb R$ be $\mu$--strongly convex
with unique minimizer $u^*(x)$.
Assume $\nabla_x\psi(x,u)$ is $L_{xu}$--Lipschitz in $u$ uniformly over $x$.
Then for all $u\in\mathbb R^{d_u}$, 
\[
\|\nabla_x\psi(x,u)-\nabla_x\psi(x,u^*(x))\|
\ \le\ \frac{L_{xu}}{\mu}\,\|\nabla_u\psi(x,u)\|.
\]
\end{lemma}
In \ref{eq:P1} and \ref{eq:P31}, Lemma~\ref{lem:deviation} is applied directly to the inner objective:
take $\psi=-f$ for the maximization tracker in \ref{eq:P1}, and take $\psi=f$ (for $y$) and $\psi=g$ (for $z$) in \ref{eq:P31}. In \ref{eq:problem2}, we bound the mismatch
$\|\nabla_x\mathcal L_\lambda(x,y)-\nabla_x\mathcal L_\lambda(x,y_\lambda^*(x))\|$.
Since $\mathcal L_\lambda(x,y)=\phi_\lambda(x,y)-\lambda g^*(x)$ with $\phi_\lambda=f+\lambda g$ and $g^*$ independent of $y$,
Lemma~\ref{lem:deviation} applies to $\phi_\lambda$ under strong convexity in $y$ given by Assumption \ref{ass:2}.

Lemma~\ref{lem:deviation} remains valid if strong convexity of $\psi(x,\cdot)$ and the $u$--Lipschitz property of
$\nabla_x\psi(x,u)$ hold only on a Lyapunov sublevel set $\{W\le W_0\}$.
In our proofs, once the time scales satisfy the absorption conditions, we obtain $\dot W(t)\le 0$ a.e., hence
$W(t)\le W(0)=:W_0$ for all $t\ge 0$.
Therefore the trajectory stays in $\{W\le W_0\}$ and the local assumptions are self-consistent.
\begin{lemma}\label{lem:timescale}
Let $W(t)=\Phi(x(t))+\alpha H_y(t)+\beta H_z(t)$ along a coupled flow, with $\alpha,\beta>0$.
Assume that for a.e.\ $t\ge 0$,
\begin{equation}\label{eq:timescale-premise}
\begin{aligned}
\dot W(t)\ \le\ -c_0\|\nabla \Phi(x(t))\|^2
\ +\ K_y\|D_y(t)\|^2\ +\ K_z\|D_z(t)\|^2 
-\ \alpha\,\tau_y\,\|r_y(t)\|^2\ -\ \beta\,\tau_z\,\|r_z(t)\|^2.
\end{aligned}
\end{equation}
for some constants $c_0>0$, $K_y,K_z\ge 0$ and residuals $r_y(t),r_z(t)$.
Suppose that Lemma~\ref{lem:deviation} holds:
\[
\|D_y(t)\|\le C_y\|r_y(t)\|,
\qquad
\|D_z(t)\|\le C_z\|r_z(t)\|.
\]
If $\tau_y\ge K_yC_y^2/\alpha$ and $\tau_z\ge K_zC_z^2/\beta$, then for a.e.\ $t\ge 0$,
$\dot W(t)\ \le\ -c_0\|\nabla \Phi(x(t))\|^2$.

\end{lemma}

In the next section, we verify \eqref{eq:timescale-premise} for each of \ref{eq:P1}--\ref{eq:P31} by explicitly differentiating the corresponding Lyapunov function $W$
and bounding the cross terms via Young's inequality, which yields concrete constants $(c_0,K_y,K_z)$ and effective time-scales $(\tau_y,\tau_z)$.
Combining these problem-specific normal-form bounds with Lemma~\ref{lem:timescale} gives the stated ergodic $O(1/T)$ stationarity guarantees.

Strong convexity/concavity is a sufficient condition that delivers (i) deviation-to-residual conversion and (ii) smooth value envelopes.
Under mere convexity/concavity, the unified proof template breaks at the same structural point in \ref{eq:P1}--\ref{eq:P31}:
cross terms cannot be absorbed without controlling $x$-gradient deviations by inner residuals.
Nevertheless, the Lyapunov framework remains a unifying perspective: any mechanism that restores a conversion of the form (error bounds / sharpness conditions, regularization/smoothing, etc.) enables the same
time-scale separation principle and yields ergodic stationarity guarantees beyond the strongly convex/concave regime.

\section{Main Results}\label{sec:main-results}
This section instantiates the unified Lyapunov toolkit developed in Section~\ref{sec:unified}.
All results below follow the same proof template driven by Lemmas~\ref{lem:deviation}--\ref{lem:timescale}:
\begin{enumerate}[label=\textbf{S\arabic*.}, ref=\textbf{S\arabic*}]
\item\label{step:S1}
Choose $W=\Phi+\alpha H_y+\beta H_z$ (Table~\ref{tab:instantiation}) and differentiate it along the coupled flow,
using Young's inequality to obtain~\eqref{eq:timescale-premise}.
This requires a smooth envelope for $\Phi$, i.e., $\Phi$ is differentiable and
$\nabla \Phi(x)=\nabla_x \psi\bigl(x,u^*(x)\bigr)$ for the relevant inner optimizer $u^*(x)$.

\item\label{step:S2}
Use Lemma~\ref{lem:deviation}
to convert the resulting deviations $(D_y,D_z)$ into the residuals $(r_y,r_z)$
that are directly dissipated by the inner dynamics.

\item\label{step:S3}
Choose the time scales large enough so that Lemma~\ref{lem:timescale} absorbs the cross terms, yielding
$\dot W \le -c_0\|\nabla \Phi(x(t))\|^2$.
Integrating this inequality over $[0,T]$ then gives an ergodic $O(1/T)$ stationarity bound.
\end{enumerate}
We collect a dictionary of the
problem-wise instantiations of $(D_y,D_z,r_y,r_z,\tau_y,\tau_z,c_0,K_y,K_z)$ in Appendix~\ref{app:diction},
and defer algebraic derivations to the corresponding appendix proofs.

\subsection{Minimax}\label{subsec:main-minimax}
We consider \ref{eq:P1} with value function $F(x):=\max_{y\in Y} f(x,y)$ and the coupled flow~\eqref{eq:flow-minimax-unified}.
Under Assumption~\ref{ass:1}, $f(x,\cdot)$ is $\mu$-strongly concave, hence the maximizer $y^*$ is unique and
Danskin's theorem yields the smooth envelope identity $\nabla F(x)=\nabla_x f\bigl(x,y^*(x)\bigr)$.

We instantiate the unified Lyapunov template with
$W_\alpha(x,y)=F(x)+\alpha\bigl(F(x)-f(x,y)\bigr)$, where the gap term tracks the inner maximization error.

\begin{proposition}\label{prop:minimax-ergodic}
Suppose Assumption~\ref{ass:1} holds, and let $t\sim\mathrm{Unif}[0,T]$.
Set $\alpha=\tfrac12$ in $W_\alpha(x,y):=F(x)+\alpha\bigl(F(x)-f(x,y)\bigr)$, and denote $W_{1/2}$ accordingly.
Assume that $\gamma\ \ge\ \frac{5}{4}\Bigl(\frac{L_{xy}}{\mu}\Bigr)^{2}$.
Define $c_1:=\tfrac12$.
Then $\mathbb{E}_{t}\!\left[\bigl\|\nabla F\bigl(x(t)\bigr)\bigr\|^{2}\right]
\ \le\ \frac{W_{1/2}(x(0),y(0)) - F_{\inf}}{c_1\,T}$.
\end{proposition}

\noindent
Proposition~\ref{prop:minimax-ergodic} immediately implies that drawing $t\sim\mathrm{Unif}[0,T]$
returns an $\varepsilon$-stationary point of $F$ in expectation once $T=\Theta(\varepsilon^{-1})$. 
\begin{corollary}\label{cor:minimax-tracking}
Suppose Assumption~\ref{ass:1} holds and let $t\sim\mathrm{Unif}[0,T]$.
If $\gamma>\frac{5}{4}\bigl(\frac{L_{xy}}{\mu}\bigr)^2$, then
$\mathbb E_t[\|y(t)-y^*(x(t))\|^{2}]
\le
\frac{W_{1/2}(x(0),y(0))-F_{\inf}}{\mu^{2}c_y(\gamma)T}$,
where $c_y(\gamma):=\frac{\gamma}{2}-\frac{5}{8}\bigl(\frac{L_{xy}}{\mu}\bigr)^2$.
\end{corollary}
When $\gamma$ is strictly above the threshold, the Lyapunov dissipation retains a quantitative control of the inner residual
$\|\nabla_y f(x(t),y(t))\|^2$.
Strong concavity of $f(x,\cdot)$ converts this residual into a bound on the tracking error $\|y(t)-y^*(x(t))\|$,
yielding an $O(1/T)$ ergodic tracking guarantee stated in Corollary~\ref{cor:minimax-tracking}.

 
\subsection{Bilevel}\label{subsec:main-bilevel}
We now instantiate the unified Lyapunov toolkit for the bilevel problem \ref{eq:problem2}.
Fix $\lambda>\rho/\mu$ and recall the formulation \eqref{eq:Elambda-unified} and the induced penalized value function
$\Phi(x):=\mathcal L_\lambda^*(x)$ in \eqref{eq:bilevel-value-f}.
As in Table~\ref{tab:instantiation}, we use the Lyapunov construction
\[
\begin{aligned}
W(x,y,z)\ :=\ &\mathcal L_\lambda^*(x)\ +\ \alpha\bigl(\mathcal L_\lambda(x,y)-\mathcal L_\lambda^*(x)\bigr)+\ \beta\lambda\bigl(g(x,z)-g^*(x)\bigr),
\end{aligned}
\]
so that the gap terms quantify tracking errors for the $y$-update (toward $y_\lambda^*(x)$)
and the $z$-tracker (toward $g^*(x)$). 

Differentiating $W$ along the coupled flow~\eqref{eq:flow-bilevel-unified} yields cross terms involving the
$x$--gradient mismatches $(D_y,D_z)$.
To carry out the Lyapunov absorption step (Lemma~\ref{lem:timescale}), we bound these mismatches by the inner residuals
$(r_y,r_z)$ dissipated by the auxiliary dynamics; see Table~\ref{tab:normal-form-dictionary_long} for the bilevel identification of $(D_y,D_z,r_y,r_z)$.
In particular, the analysis requires a deviation-to-residual conversion of the form
\begin{equation}\label{eq:DyDz-conv-main}
\|D_y\|\ \le\ C_y(\lambda)\,\|r_y\|,
\qquad
\|D_z\|\ \le\ C_z\,\|r_z\|.
\end{equation}
This conversion is the key condition needed to absorb the cross terms via Lemma~\ref{lem:timescale}. A sufficient condition for \eqref{eq:DyDz-conv-main} is the strong convexity of
$\phi_\lambda(x,y):=f(x,y)+\lambda g(x,y)
\quad\text{in $y$},$ together with the cross-Lipschitz bounds in Assumptions~\ref{ass:2}.2--\ref{ass:2}.3.
Indeed, since $\mathcal L_\lambda(x,y)=\phi_\lambda(x,y)-\lambda g^*(x)$, Lemma~\ref{lem:deviation} applies to $\phi_\lambda$ and yields \eqref{eq:DyDz-conv-main} with $C_y(\lambda)=\frac{L_{fx}+\lambda L_{gx}}{\mu_\lambda},\ C_z=\frac{L_{gx}}{\mu}$.
Here $\mu_\lambda$ denotes the strong-convexity modulus of
$\phi_\lambda(x,\cdot)=f(x,\cdot)+\lambda g(x,\cdot)$. Moreover, for each fixed $x$, $\phi_\lambda(x,\cdot)$ is strongly convex if $g(x,\cdot)$ is $\mu$--strongly convex and
the map $y\mapsto f(x,y)$ is $\rho$--weakly convex, provided that $\lambda>\rho/\mu$.
The proof is deferred to Proposition~\ref{prop:bilevel-deviation}.
The Lyapunov analysis only needs \eqref{eq:DyDz-conv-main}; $\rho$-weak convexity on $f(x,\cdot)$~(Assumption~\ref{ass:2}) is merely a convenient sufficient condition to ensure it.

With \eqref{eq:DyDz-conv-main} in hand, the Lyapunov derivative matches the normal form \eqref{eq:timescale-premise}
with effective time scales $(\tau_y,\tau_z)=(\delta,\eta\lambda^2)$ and Young-type constants $(c_0,K_y,K_z)$
(see Appendix~\ref{app:diction} for the dictionary).
The Young parameter $\varepsilon_1$ can be chosen arbitrarily in $(0,2)$; for readability we set $\varepsilon_1=1$
in the main text (the general expressions are given in the proof). Fix any $\alpha,\beta\in(0,1)$ and any $\theta>0$, and define $c_1:=(1-\alpha)/4,$
\begin{align}
\delta_0(\lambda)
\ &:=\ \frac{C_y(\lambda)^2}{\alpha}
\left(\frac{1-\alpha}{2}+\frac{\theta|\alpha-\beta|}{2}\right), \
\eta_0
\ :=\ \frac{C_z^2}{\beta}
\left(\beta+\frac{(1-\beta)^2}{1-\alpha}+\frac{|\alpha-\beta|}{2\theta}\right).
\label{eq:eta0-bilevel}
\end{align}
Here $\theta>0$ is only a Young-balancing parameter.  

As a useful symmetric specialization, if $\alpha=\beta=\tfrac12$, then $\delta_0(\lambda)=\tfrac12 C_y(\lambda)^2, \eta_0=2C_z^2$.
The following theorem applies Lemma~\ref{lem:timescale} with these thresholds to obtain the fixed-$\lambda$ $O(1/T)$ stationarity bound for $\Phi(x)=\mathcal L_\lambda^*(x)$.


\begin{theorem}\label{thm:fixed-lambda-simple}
Assume Assumption~\ref{ass:2} holds. Fix $\lambda>\rho/\mu$, $\alpha,\beta\in(0,1)$, and $\theta>0$.
Consider any solution $(x(t),y(t),z(t))$ of \eqref{eq:flow-bilevel-unified}.
If the time-scale parameters satisfy $\delta \ \ge\ \delta_0(\lambda)$ and $\eta \ \ge\ \eta_0$, where $\delta_0(\lambda)$ and $\eta_0$ are defined in \eqref{eq:eta0-bilevel},
then for $t\sim\mathrm{Unif}[0,T]$,
$
\mathbb E_t\Big[\bigl\|\nabla_x \mathcal L_\lambda^{*}(x(t))\bigr\|^{2}\Big]
\ \le\ \frac{W(x(0),y(0),z(0))-\mathcal L_{\lambda,\inf}^{*}}{c_1\,T},
$
where $c_1= (1-\alpha)/4$ .
\end{theorem}
Theorem~\ref{thm:fixed-lambda-simple} implies that if $t\sim\mathrm{Unif}[0,T]$, then
$\mathbb E_t\!\left[\|\nabla_x \mathcal L_\lambda^*(x(t))\|^2\right]\le \varepsilon$
once $T=\Theta(\varepsilon^{-1})$. 
Under strict separation $\delta>\delta_0(\lambda)$ and $\eta>\eta_0$, the same Lyapunov inequality also controls the residual terms, which implies quantitative tracking of $y(t)$ and $z(t)$.

\begin{corollary}\label{cor:bilevel-tracking}
Assume the conditions of Theorem~\ref{thm:fixed-lambda-simple}.
Define the positive margins $c_y(\delta,\lambda):=\alpha\bigl(\delta-\delta_0(\lambda)\bigr),\quad
c_z(\eta,\lambda):=\beta\lambda^2\bigl(\eta-\eta_0\bigr)$.
If, for every $x$, the function
\(
y\mapsto \phi_\lambda(x,y):=f(x,y)+\lambda g(x,y)
\)
is $\mu_\lambda$-strongly convex, then for $t\sim\mathrm{Unif}[0,T]$, $\mathbb E_t\bigl[\|y(t)-y_\lambda^*(x(t))\|^2\bigr]
\le
\frac{W(x(0),y(0),z(0))-\mathcal L_{\lambda,\inf}^{*}}
{\mu_\lambda^2\,c_y(\delta,\lambda)\,T}$.
Moreover, since $g(x,\cdot)$ is $\mu$-strongly convex, $\mathbb E_t\bigl[\|z(t)-y^*(x(t))\|^2\bigr]
\le
\frac{W(x(0),y(0),z(0))-\mathcal L_{\lambda,\inf}^{*}}
{\mu^2\,c_z(\eta,\lambda)\,T}$,
where $y^*(x)\in\arg\min_y g(x,y)$.
\end{corollary}
Theorem~\ref{thm:fixed-lambda-simple} guarantees stationarity of the penalized value function $\mathcal L_\lambda^*$.
To obtain a guarantee for the true bilevel hyper-objective $F(x)=f(x,y^*(x))$, we must control the surrogate bias
$\|\nabla_x\mathcal L_\lambda^*(x)-\nabla F(x)\|$. To quantify the discrepancy between $\nabla_x\mathcal L_\lambda^{*}(x)$ and $\nabla F(x)$, we impose the following
regularity conditions.
\begin{assumption}\label{ass:3}
Suppose that $f$ and $g$ are twice continuously differentiable and that:
\begin{enumerate}[leftmargin=*]
\item $\displaystyle \sup_{x,y}\ \|\nabla_{xy}^{2}g(x,y)\|
\le M_{gxy}.$
\item $\displaystyle \|\nabla_y f(x,y_1)-\nabla_y f(x,y_2)\|
\le L_{fy}\|y_1-y_2\|.$
\item $\displaystyle \|\nabla_{xy}^{2}g(x,y_1)-\nabla_{xy}^{2}g(x,y_2)\|
\le L_{gxy}\|y_1-y_2\|$, 
$\displaystyle \|\nabla_{yy}^{2}g(x,y_1)-\nabla_{yy}^{2}g(x,y_2)\|
\le L_{gyy}\|y_1-y_2\|.$
\item $\|\nabla F(x)-\nabla F(x')\|\le L_F\|x-x'\|.$
\end{enumerate}
\end{assumption}

\begin{lemma}\label{lem:bridge}
Assume the conditions of Theorem~\ref{thm:fixed-lambda-simple}, and
Assumption~\ref{ass:3}. Let
$\bar\lambda\ :=\ \frac{2L_{fy}}{\mu}$. Then for any $\lambda\ge \bar\lambda$ and every $x\in\mathbb R^{d_x}$, $\bigl\|\nabla_x \mathcal L_\lambda^{*}(x)\;-\;\nabla F(x)\bigr\|
\ \le\ \frac{C_{\mathrm{br}}(x)}{\lambda}$,
where 
\begin{equation}\label{eq:Cbr-x}
\begin{aligned}
C_{\mathrm{br}}(x)\!
:=
\frac{2}{\mu}\Bigl(L_{fx}+\frac{M_{gxy}L_{fy}}{\mu}\Bigr)\,\|\nabla_y f(x,y^{*}(x))\| 
+
\frac{2}{\mu^2}\Bigl(L_{gxy}+\frac{M_{gxy}L_{gyy}}{\mu}\Bigr)\,\|\nabla_y f(x,y^{*}(x))\|^{2}.
\end{aligned}
\end{equation}
\end{lemma}
Lemma~\ref{lem:bridge} shows that the bias decays as $O(1/\lambda)$; choosing $\lambda$ balances this bias with the $O(1/T)$ optimization error from Theorem~\ref{thm:fixed-lambda-simple}.
We combine the fixed-$\lambda$ ergodic bound from Theorem~\ref{thm:fixed-lambda-simple}
with the bridge Lemma~\ref{lem:bridge} to obtain an ergodic stationarity guarantee for the true bilevel objective
$F(x)$. 
\begin{assumption}\label{ass:4}
There exists a constant $G_{*}>0$ such that, along the region of interest,
\begin{equation}\label{eq:Gstar-assump}
\|\nabla_y f(x,y^{*}(x))\|\ \le\ G_{*}.
\end{equation}
\end{assumption}
\begin{corollary}\label{cor:bilevel-eps}
Assume the conditions of Theorem~\ref{thm:fixed-lambda-simple}, Lemma~\ref{lem:bridge}, and Assumption~\ref{ass:4}. Let $C_{\mathrm{br}}>0$ be any constant such that $C_{\mathrm{br}}(x)\ \le\ C_{\mathrm{br}}
\;(\text{for all $x$ in the region of interest})$,
where $C_{\mathrm{br}}(x)$ is defined in \eqref{eq:Cbr-x}.
Fix any $\varepsilon>0$ and choose $\lambda\ :=\ \max\Bigl\{\bar\lambda,\ \frac{2C_{\mathrm{br}}}{\sqrt\varepsilon}\Bigr\}$, where $\bar\lambda$ is defined in Lemma~\ref{lem:bridge}.
Choose $\delta \ \ge\ \delta_0(\lambda), \  \eta \ \ge\ \eta_0$.
Let $t\sim\mathrm{Unif}[0,T]$. If $T\ \ge\ \frac{4\bigl(W(x(0),y(0),z(0))-\mathcal L_{\lambda,\inf}^{*}\bigr)}
{c_1\,\varepsilon},$
then $\mathbb E_t\Big[\|\nabla F(x(t))\|^{2}\Big]\ \le\ \varepsilon$.
\end{corollary}


\subsection{Unified Construction for Convex Inner Problems}\label{sec:mere}
Our unified Lyapunov analysis for \ref{eq:P1}--\ref{eq:P31} relies on two structural ingredients:
(i) a deviation-to-residual conversion that links $x$--gradient mismatches to inner residuals,
and (ii) a smooth envelope property for the outer value function.
Strong convexity/concavity is a convenient sufficient condition that simultaneously guarantees both. This section makes precise, step-by-step, where the $O(1/T)$ ergodic stationarity template breaks when the inner problems are merely convex/concave, and how to repair the argument using modular replacements (error bounds, PL/sharpness, or regularization) while keeping the same Lyapunov construction. 

We first focus on the convex/concave regime where the lower-level optimizer map is single-valued,
so the template breaks only at the two modules \ref{step:S1}--\ref{step:S2}.
When the lower-level solution is set-valued, an additional difficulty arises (nonsmooth envelopes and selection dependence);
we defer this extension to Appendix,
where we show that under suitable sensitivity regularity one can still proceed by
modifying only \ref{step:S1}--\ref{step:S2} while keeping the same Lyapunov construction.

If the inner objectives are mere convex/concave, then:
\begin{itemize}[leftmargin=*]
\item \textbf{Failure at \ref{step:S1}} 
In general, $\arg\min$/$\arg\max$ sets are multi-valued, and value functions such as
$\max_y f(x,y)$ or $\min_y g(x,y)$ need not be differentiable.
Thus the identity 
\begin{equation}\label{eq:envelope}
\nabla\Phi(x)=\nabla_x\psi(x,u^*(x)) 
\end{equation}
used to produce the descent term may break.

\item \textbf{Failure at \ref{step:S2}}
Lemma~\ref{lem:deviation} requires $\mu$--strong convexity of $\psi(x,\cdot)$ to obtain
\[
\|D\|
\;=\;
\|\nabla_x\psi(x,u)-\nabla_x\psi(x,u^*)\|
\;\lesssim\;
\|u-u^*\|,\;
\|u-u^*\|
\;\lesssim\;
\|\nabla_u\psi(x,u)\|.
\]
Without strong convexity/concavity of the inner objective, Lemma~\ref{lem:deviation} is not guaranteed to hold.
\end{itemize}
When strong convexity/concavity is dropped, the proof template breaks exactly at these two modules, and the repairs can be described in a problem-agnostic way.


\paragraph{Repairing \ref{step:S1}.}
When inner optimizers are nonunique, $\Phi$ can be nonsmooth and \eqref{eq:envelope} may fail.
A standard remedy is to use nonsmooth stationarity
$\operatorname{dist}(0,\partial\Phi(x))$ with a nonsmooth chain rule, or to restore differentiability via smoothing/regularization.
We do not pursue these nonsmooth-envelope variants here, as they would require a separate nonsmooth analysis
(e.g., subgradient chain rules and differential inclusions), and instead focus on the modular repair of \ref{step:S2} below.

\begin{assumption}[Error Bound]\label{assump:inner-EB}
Let $\psi:\mathbb R^{d_x}\times\mathbb R^{d_u}\to\mathbb R$ be differentiable and convex in $u$.
Let $\mathcal R\subseteq\mathbb R^{d_x}\times\mathbb R^{d_u}$ be a region of interest and
$U^*(x):=\arg\min_u \psi(x,u)$ be nonempty for all $(x,u)\in\mathcal R$.
We assume that $\psi$ satisfies a (local) gradient error bound on $\mathcal R$: there exists $\kappa\in(0,\infty)$ such that
$\operatorname{dist}\bigl(u,U^*(x)\bigr)\ \le\ \kappa\,\|\nabla_u \psi(x,u)\|
\quad(\forall\, (x,u)\in\mathcal R)$.
\end{assumption}
\paragraph{Repairing \ref{step:S2}.} Assumption~\ref{assump:inner-EB} reinstates the missing metric control needed for Lyapunov absorption.
In smooth convex settings, it is implied (up to constants) by standard sharpness conditions on the relevant sublevel set,
such as PL/gradient-dominance and quadratic growth; see Appendix~\ref{app:eb-interface}.

\begin{lemma}\label{lem:deviation-EB}
Assume $\nabla_x\psi(x,u)$ is $L_{xu}$--Lipschitz in $u$ on $\mathcal R$. If Assumption~\ref{assump:inner-EB} holds on $\mathcal R$, then for all $(x,u)\in\mathcal R$, $\inf_{u^*\in U^*(x)}
\bigl\|\nabla_x\psi(x,u)-\nabla_x\psi(x,u^*)\bigr\|
\ \le\ L_{xu}\,\kappa\,\|\nabla_u\psi(x,u)\|$.
\end{lemma}
Once the deviation-to-residual conversion~(\ref{eq:DyDz-conv-main}) holds (with constants $C_y,C_z$),
Lemma~\ref{lem:timescale} applies without modification.
In particular, the absorption conditions retain the same algebraic form
$\tau_y \ge K_y C_y^2/\alpha$ and $\tau_z \ge K_z C_z^2/\beta$.
Compared to the strongly convex case (where $\kappa=1/\mu$), Assumption~\ref{assump:inner-EB} replaces $1/\mu$ by $\kappa$,
so the required time-scale separation deteriorates through the same $C^2$ dependence in the thresholds.

\begin{figure*}[t]
\centering
\begin{subfigure}[t]{0.49\textwidth}
    \centering
    \includegraphics[width=\linewidth]{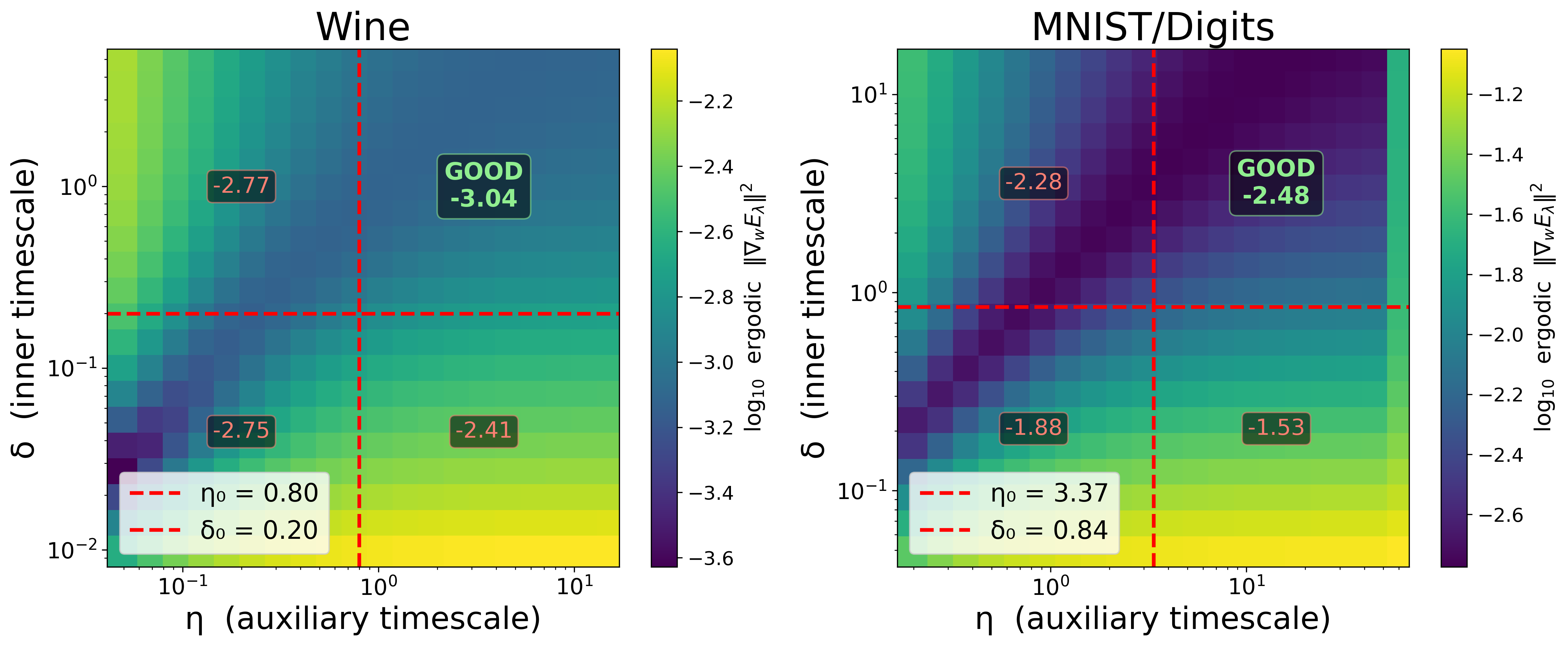}
    \caption{Regularized setting with $\rho_{\rm reg}=0.05$.}
    \label{fig:euler-heatmap-regularized}
\end{subfigure}
\hfill
\begin{subfigure}[t]{0.49\textwidth}
    \centering
    \includegraphics[width=\linewidth]{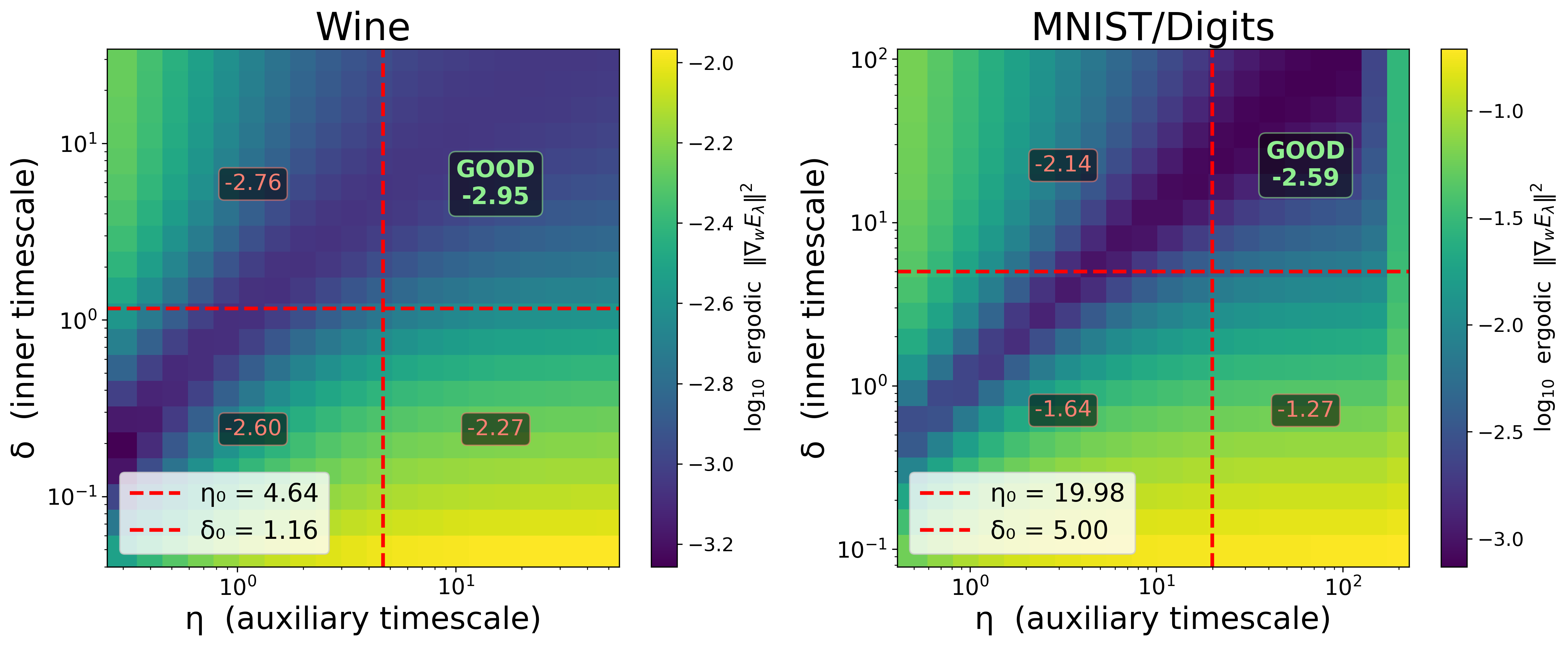}
    \caption{Weakened inner geometry with $\rho_{\rm reg}=0.02$.}
    \label{fig:euler-heatmap-weakgeom}
\end{subfigure}
\caption{
Forward-Euler diagnostics for lifted-penalty hypercleaning.
Left: $\rho_{\rm reg}=0.05$; right: weakened geometry with
$\rho_{\rm reg}=0.02$. Here $\rho_{\rm reg}$ denotes the lower-level $\ell_2$ regularization coefficient in the objective. Dashed lines mark
$\delta_0$ and $\eta_0$.
Darker values indicate smaller
$\log_{10}\|\nabla_w E_\lambda\|^2$.
}
\label{fig:euler-heatmap}
\vspace{-0.5cm}
\end{figure*}

\subsection{Discrete-time interpretation and diagnostics.}
This continuous-time viewpoint also provides a useful reference point for stable Euler discretizations. The S1--S3 normal form also gives a direct Euler one-step interpretation:
for $\dot u=G(u)$ and $u_{k+1}=u_k+hG(u_k)$,
\[
W(u_{k+1})-W(u_k)
\le h\dot W(u_k)+O(h^2\|G(u_k)\|^2).
\]
Hence discretization preserves the continuous-time dissipation at first order and adds only a controllable remainder; see Appendix~\ref{app:euler-discretization}. Under weaker geometry, the EB/PL/QG interface changes the conversion constant $C$, shifting the thresholds through the same $C^2$ rule. Figure~\ref{fig:euler-heatmap} illustrates this on Wine and Digits hypercleaning~\cite{franceschi2018bilevel,shaban2019truncated}: when $\rho_{\rm reg}$ is reduced from $0.05$ to $0.02$, the predicted threshold increase is about $6.25$, the observed shifts are $5.8$--$6.0$, and the above-threshold region remains favorable. Details are in Appendix~\ref{app:euler-hypercleaning}.

\section{Conclusion}
We introduced a value-function-level continuous-time Lyapunov template for minimax, lifted-penalty bilevel, and min--min--max dynamics. The S1--S3 decomposition separates deviation-to-residual conversion from time-scale absorption, yielding closed-form thresholds and making the dependence on inner geometry explicit through EB constants. The finite-time tracking bound connects the penalty dynamics to the ideal hyper-gradient flow, while the Euler analogue and hypercleaning diagnostics indicate that the resulting relative time-scale structure remains visible under stable discretization. Overall, the framework provides a modular route for analyzing nested optimization dynamics across problem classes and weaker inner geometries.

\section*{Acknowledgements}
HL is partially supported by the National Research Foundation of Korea (NRF) grant funded by the Korea government (MSIT) (No. RS-2024-00408003) and by a KIAS Individual Grant (AP103101) via the Center for AI and Natural Sciences at the Korea Institute for Advanced Study. DK is partially supported by the National Research Foundation of Korea (NRF) grant funded by the Korea government (MSIT) (No. RS-2023-00252516, No. RS-2024-00408003, and No. RS-2026-25488663), the POSCO Science Fellowship of POSCO TJ Park Foundation, and the Korea Institute for Advanced Study.

\bibliography{references}

\clearpage
\clearpage
\appendix
\section{Auxiliary Lemmas}
\subsection{Instantiation Dictionary for the Unified Lemmas}\label{app:diction}
We list the concrete objects $(D_y,D_z,r_y,r_z,\tau_y,\tau_z)$ in
Lemmas~\ref{lem:deviation}--\ref{lem:timescale} for \ref{eq:P1}--\ref{eq:P31}. In each case, $D_y,D_z$ are the $x$-gradient
mismatches induced by imperfect tracking, $r_y,r_z$ are the inner residuals driven by the auxiliary dynamics, and $\tau_y,\tau_z$ are the effective time-scale coefficients multiplying $\|r_y\|^2,\|r_z\|^2$ in $\dot W$.

\begin{table}[h]
\centering
\small
\setlength{\tabcolsep}{4.0pt}
\begin{tabular}{@{}lccc@{}}
\toprule
 & (P1) Minimax & (P2) Bilevel & (P3) Min--min--max \\
\midrule
$r_y$ &
$\nabla_y f(x,y)$ &
$\nabla_y \mathcal L_\lambda(x,y)$ &
$\nabla_y f(x,y)$ \\

$r_z$ &
-- &
$\nabla_y g(x,z)$ &
$\nabla_y g(x,z)$ \\

$D_y$ &
$\nabla_x f(x,y)-\nabla_x f\bigl(x,y^*(x)\bigr)$ &
$\nabla_x \mathcal L_\lambda(x,y)-\nabla_x \mathcal L_\lambda\bigl(x,y_\lambda^*(x)\bigr)$ &
$\nabla_x f(x,y)-\nabla_x f\bigl(x,y^*(x)\bigr)$ \\

$D_z$ &
-- &
$\nabla_x g(x,z)-\nabla g^*(x)$ &
$\nabla_x g(x,z)-\nabla_x g\bigl(x,z^*(x)\bigr)$ \\

$\tau_y$ &
$\gamma$ &
$\delta$ &
$\delta$ \\

$\tau_z$ &
-- &
$\eta\,\lambda^{2}$ &
$\eta$ \\
\bottomrule
\end{tabular}
\caption{Problem-wise definitions of $(r_y,r_z)$, $(D_y,D_z)$, and $(\tau_y,\tau_z)$ in Lemmas~\ref{lem:deviation}--\ref{lem:timescale}.}
\label{tab:normal-form-dictionary_long}
\end{table}

We next summarize, for each problem, the Lyapunov derivative $\dot W$ in the normal form~\eqref{eq:timescale-premise},
together with the corresponding Young-type constants $c_0$, $K_y$, and $K_z$.

\paragraph{(P1) Minimax.}
From Appendix~\ref{subsec:proof-minmax}~(see \eqref{eq:W-star-minimax}), one can write
\[
\dot W \ \le\ -c_0\|\nabla\Phi(x)\|^2 + K_y\|D_y\|^2 - \alpha\,\tau_y\|r_y\|^2,
\qquad(\beta=0,\ D_z=r_z=0),
\]
with $\Phi(x)=F(x)$, $\tau_y=\gamma$, and
\[
c_0 = 1-\frac{(1-\alpha)\varepsilon}{2},
\qquad
K_y = \alpha+\frac{1-\alpha}{2\varepsilon}.
\]
(Setting $\alpha=\tfrac12$ and $\varepsilon=2$ yields $c_0=\tfrac12$ and matches Proposition~4.1.)

\paragraph{(P2) Bilevel.}
From Appendix~\ref{subsec:proof-bilevel}~(see \eqref{eq:Wdot-mid-bilevel}), we obtain
\[
\dot W \ \le\ -c_0\|\nabla\Phi(x)\|^2 + K_y\|D_y\|^2 + K_z\|D_z\|^2
-\alpha\,\tau_y\|r_y\|^2-\beta\,\tau_z\|r_z\|^2,
\]
with $\Phi(x)=\mathcal L_\lambda^*(x)$, $(\tau_y,\tau_z)=(\delta,\eta\lambda^2)$, and
\[
c_0=c_1 := (1-\alpha)\Bigl(1-\frac{\varepsilon_1}{2}\Bigr)-\lambda(1-\beta)\frac{\varepsilon_2}{2},
\]
\[
K_y := \frac{1-\alpha}{2\varepsilon_1}+\frac{\theta|\alpha-\beta|}{2},
\qquad
K_z := \frac{\lambda(1-\beta)}{2\varepsilon_2}+\beta\lambda^{2}
+\frac{\lambda^2|\alpha-\beta|}{2\theta},
\]
In particular, choosing $\varepsilon_2$ as in \eqref{eq:eps2-choice} makes $c_0=c_1>0$ independent of $\lambda$.

\paragraph{(P3) Min--min--max.}
From Appendix~\ref{sec:minminmax}~(see \eqref{eq:Wdot-mid} therein), we obtain the same normal form
\[
\dot W
\ \le\
-\,c_0\|\nabla \Phi(x)\|^2
\ +\ K_y\|D_y\|^2
\ +\ K_z\|D_z\|^2
\ -\ \alpha\,\tau_y\,\|r_y\|^2
\ -\ \beta\,\tau_z\,\|r_z\|^2,
\]
with $\Phi(x)=\mathcal L(x)$, $(\tau_y,\tau_z)=(\delta,\eta)$, and the Young-type constants
\[
c_0 \;:=\; 1-\frac{(1+\alpha)\varepsilon_1}{2}-\frac{(1-\beta)\varepsilon_2}{2},
\qquad
K_y \;:=\; \frac{1+\alpha}{2\varepsilon_1}+\frac{|\alpha-\beta|}{2},
\qquad
K_z \;:=\; \frac{1-\beta}{2\varepsilon_2}+\frac{|\alpha-\beta|}{2}+\beta.
\]

\clearpage
\subsection{Auxiliary Lemmas}\label{subsec:aux-minmax}

\begin{definition}\label{def:SC}
A continuously differentiable function $h:\mathbb{R}^n\to\mathbb{R}$ is called $\mu$--strongly convex
if there exists $\mu>0$ such that for all $u,v\in\mathbb{R}^n$,
\[
h(v)\ \ge\ h(u)\ +\ \langle \nabla h(u),\,v-u\rangle\ +\ \frac{\mu}{2}\|v-u\|^{2}.
\]
\end{definition}

\begin{definition}\label{def:weak-convex}
A continuously differentiable function $h:\mathbb{R}^n\to\mathbb{R}$ is called $\rho$--weakly convex
if there exists $\rho\ge 0$ such that for all $u,v\in\mathbb{R}^n$,
\[
h(v)\ \ge\ h(u)\ +\ \langle \nabla h(u),\,v-u\rangle\ -\ \frac{\rho}{2}\|v-u\|^{2}.
\]
\end{definition}

\begin{lemma}\label{lem:coercive}
If $h$ is $\mu$--strongly convex and continuously differentiable, then for all $u,v$,
\[
\mu\,\|u-v\|\ \le\ \|\nabla h(u)-\nabla h(v)\|.
\]
\end{lemma}

\begin{proof}
A standard property of $\mu$--strongly convex functions is the strong monotonicity of the gradient:
\[
\langle \nabla h(u)-\nabla h(v),\,u-v\rangle\ \ge\ \mu\|u-v\|^{2}.
\]
Applying Cauchy--Schwarz to the left-hand side gives
$\|\nabla h(u)-\nabla h(v)\|\,\|u-v\|\ge \mu\|u-v\|^2$, and dividing by $\|u-v\|$ yields the claim.
\end{proof}

\begin{lemma}\label{lem:one-sided}
Assume that for every $y\in Y$, the mapping $x\mapsto f(x,y)$ is $\mu_x$--strongly convex on $X$,
and for every $x\in X$, the mapping $y\mapsto f(x,y)$ is $\mu_y$--strongly concave on $Y$.
Then for any $x,x'\in X$ and any $y,y'\in Y$,
\begin{align}
\big\langle \nabla_x f(x,y)-\nabla_x f(x',y),\,x-x'\big\rangle
&\ge \mu_x\,\|x-x'\|^2,
\label{eq:scx}\\
\big\langle \nabla_y f(x,y)-\nabla_y f(x,y'),\,y-y'\big\rangle
&\le -\mu_y\,\|y-y'\|^2.
\label{eq:sccy}
\end{align}
\end{lemma}

\begin{proof}
These are standard consequences of strong convexity of $f(\cdot,y)$ and strong concavity of $f(x,\cdot)$,
respectively, via strong monotonicity of $\nabla_x f(\cdot,y)$ and strong antimonotonicity of $\nabla_y f(x,\cdot)$.
\end{proof}

\begin{lemma}\label{lem:support}
Assume that for every $y\in Y$, the mapping $x\mapsto f(x,y)$ is convex on $X$,
and for every $x\in X$, the mapping $y\mapsto f(x,y)$ is concave on $Y$.
Then for any $x,x'\in X$ and any $y,y'\in Y$,
\begin{align}
f(x,y) &\ge f(x',y)+\big\langle \nabla_x f(x',y),\,x-x'\big\rangle,
\label{eq:support-x}\\
f(x,y) &\le f(x,y')+\big\langle \nabla_y f(x,y'),\,y-y'\big\rangle.
\label{eq:support-y}
\end{align}
\end{lemma}

\begin{proof}
The first inequality follows from convexity of $f(\cdot,y)$ in $x$, and the second follows from concavity
of $f(x,\cdot)$ in $y$ (equivalently, convexity of $-f(x,\cdot)$), both via the first-order characterization.
\end{proof}

\begin{lemma}\label{lem:minimax-inner-residual}
Suppose Assumptions~\ref{ass:1} holds and let $t\sim\mathrm{Unif}[0,T]$.
If $\gamma>\frac{5}{4}\bigl(\frac{L_{xy}}{\mu}\bigr)^2$, then
\[
\mathbb E_t\!\left[\bigl\|\nabla_y f\bigl(x(t),y(t)\bigr)\bigr\|^{2}\right]
\ \le\
\frac{W_{1/2}(x(0),y(0)) - F_{\inf}}{c_y(\gamma)\,T}.
\]
\end{lemma}
\begin{proof}
Integrating \eqref{eq:minimax-dissipation} over $[0,T]$ and using $W_{1/2}(x,y)\ge F(x)\ge F_{\inf}$ give
\[
\frac12\int_0^T \|\nabla F(x(s))\|^2\,ds
\ +\ 
c_y(\gamma)\int_0^T \|\nabla_y f(x(s),y(s))\|^2\,ds
\ \le\
W_{1/2}(x(0),y(0)) - F_{\inf}.
\]
Dropping the first nonnegative term and dividing by $T$ yields
\[
\mathbb E_t\!\left[\|\nabla_y f(x(t),y(t))\|^2\right]
=
\frac{1}{T}\int_0^T \|\nabla_y f(x(s),y(s))\|^2\,ds
\ \le\
\frac{W_{1/2}(x(0),y(0)) - F_{\inf}}{c_y(\gamma)\,T}.
\]
\end{proof}

\subsection{Common Lemmas}
We state two common lemmas that drive the unified analysis: a deviation-to-residual conversion and a time-scale absorption principle.

\subsubsection{Proof of Lemma~\ref{lem:deviation}}
\begin{proof}
Strong convexity implies strong monotonicity of $\nabla_u\psi(x,\cdot)$:
$\langle \nabla_u\psi(x,u)-\nabla_u\psi(x,u^*),\,u-u^*\rangle\ge \mu\|u-u^*\|^2$.
Using $\nabla_u\psi(x,u^*)=0$ and Cauchy--Schwarz yields
$\mu\|u-u^*\|\le \|\nabla_u\psi(x,u)\|$.
By Lipschitzness of $\nabla_x\psi$ in $u$,
$\|\nabla_x\psi(x,u)-\nabla_x\psi(x,u^*)\|\le L_{xu}\|u-u^*\|$.
Combine the last two inequalities.
\end{proof}

\subsubsection{Proof of lemma~\ref{lem:timescale}}

\begin{proof}
Using $\|D_y\|^2\le C_y^2\|r_y\|^2$ and $\|D_z\|^2\le C_z^2\|r_z\|^2$ in \eqref{eq:timescale-premise} gives
\[
\begin{aligned}
\dot W(t)\le\ &-c_0\|\nabla\Phi(x(t))\|^2
+\bigl(K_yC_y^2-\alpha\tau_y\bigr)\|r_y(t)\|^2 \\
&+\bigl(K_zC_z^2-\beta\tau_z\bigr)\|r_z(t)\|^2.
\end{aligned}
\]
Under $\tau_y\ge K_yC_y^2/\alpha$ and $\tau_z\ge K_zC_z^2/\beta$, the last two coefficients are nonpositive, so we drop them.
\end{proof}





\section{Minimax Problem}
\subsection{Proof of Proposition~\ref{prop:minimax-ergodic}}\label{subsec:proof-minmax}

\begin{proof}
\textbf{[S1: smooth envelope.]}
Under $\mu$--strong concavity of $f(x,\cdot)$, for each $x$ the maximizer $y^*(x)$ is unique and
Danskin's theorem gives the smooth-envelope identity
\[
\Phi(x)\ :=\ \max_{y\in Y} f(x,y)\ =\ f\bigl(x,y^*(x)\bigr),
\qquad
\nabla\Phi(x)\ =\ \nabla_x f\bigl(x,y^*(x)\bigr).
\]

\textbf{Lyapunov choice (unified template).}
Define the nonnegative gap $H_y(x,y):=\Phi(x)-f(x,y)\ge 0$ and
\begin{equation}\label{eq:W-def-minimax}
W_\alpha(x,y)\ :=\ \Phi(x)\ +\ \alpha\,H_y(x,y)\qquad(\alpha>0).
\end{equation}

Along the coupled flow~\eqref{eq:flow-minimax-unified}, let
\[
A\ :=\ \nabla\Phi(x)\ =\ \nabla_x f\bigl(x,y^*(x)\bigr),\qquad
B\ :=\ \nabla_x f(x,y),\qquad
r_y\ :=\ \nabla_y f(x,y).
\]
Define the $x$--gradient deviation (mismatch)
\[
D_y\ :=\ B-A\ =\ \nabla_x f(x,y)-\nabla_x f\bigl(x,y^*(x)\bigr).
\]
Then $\dot x=-B$ and $\dot y=\gamma r_y$.

We have
\[
\frac{d}{dt}\Phi(x(t))=\langle \nabla\Phi(x(t)),\dot x(t)\rangle
=-\langle A,B\rangle.
\]
Using $\nabla_y f(x,y^*(x))=0$ and the chain rule,
\[
\frac{d}{dt}f\bigl(x(t),y^*(x(t))\bigr)=\langle A,\dot x\rangle,
\qquad
\frac{d}{dt}f\bigl(x(t),y(t)\bigr)=\langle B,\dot x\rangle+\langle r_y,\dot y\rangle.
\]
Hence $\dot H_y=\langle A-B,\dot x\rangle-\langle r_y,\dot y\rangle$ and therefore
\[
\dot H_y
=\langle -D_y,-B\rangle-\gamma\|r_y\|^2
=-\langle D_y,B\rangle-\gamma\|r_y\|^2.
\]
From~\eqref{eq:W-def-minimax},
\[
\dot W_\alpha=\dot\Phi+\alpha\dot H_y
=-(1+\alpha)\langle A,B\rangle+\alpha\|B\|^2-\alpha\gamma\|r_y\|^2.
\]
Substituting $B=A+D_y$ and expanding gives
\begin{equation}\label{eq:W-expand-minimax}
\dot W_\alpha
=-\|A\|^2+(\alpha-1)\langle A,D_y\rangle+\alpha\|D_y\|^2-\alpha\gamma\|r_y\|^2.
\end{equation}

\textbf{[S2: deviation-to-residual conversion.]}
By the cross-Lipschitz assumption $\|\nabla_x f(x,y)-\nabla_x f(x,y')\|\le L_{xy}\|y-y'\|$ and
$\mu$--strong concavity of $f(x,\cdot)$ (equivalently $\mu$--strong convexity of $-f(x,\cdot)$),
Lemma~\ref{lem:deviation} (applied to $\psi(x,y)=-f(x,y)$) yields
\begin{equation}\label{eq:Dy-ry-minimax}
\|D_y\|\ \le\ \frac{L_{xy}}{\mu}\,\|r_y\|.
\end{equation}

\textbf{[S3: absorption via time-scale separation.]}
Since $(1-\alpha)>0$ when $\alpha\in(0,1)$, Young's inequality gives for any $\varepsilon>0$,
\[
(\alpha-1)\langle A,D_y\rangle
\le
(1-\alpha)\Bigl(\frac{\varepsilon}{2}\|A\|^2+\frac{1}{2\varepsilon}\|D_y\|^2\Bigr).
\]
Inserting this into~\eqref{eq:W-expand-minimax} yields
\begin{equation}\label{eq:W-star-minimax}
\dot W_\alpha
\le
-\Bigl(1-\frac{(1-\alpha)\varepsilon}{2}\Bigr)\|A\|^2
+\Bigl(\alpha+\frac{1-\alpha}{2\varepsilon}\Bigr)\|D_y\|^2
-\alpha\gamma\|r_y\|^2.
\end{equation}
Using~\eqref{eq:Dy-ry-minimax}, we obtain
\[
\dot W_\alpha
\le
-\Bigl(1-\frac{(1-\alpha)\varepsilon}{2}\Bigr)\|A\|^2
+\Biggl(-\alpha\gamma+\Bigl(\alpha+\frac{1-\alpha}{2\varepsilon}\Bigr)\Bigl(\frac{L_{xy}}{\mu}\Bigr)^2\Biggr)\|r_y\|^2.
\]

Set $\alpha=\tfrac12$ and $\varepsilon=2$. Then
\[
\dot W_{1/2}(t)
\le
-\frac12\|A\|^2
-\Bigl(\frac{\gamma}{2}-\frac{5}{8}\Bigl(\frac{L_{xy}}{\mu}\Bigr)^2\Bigr)\|r_y\|^2.
\]
Thus, if $\gamma\ge \frac{5}{4}\bigl(\frac{L_{xy}}{\mu}\bigr)^2$, we have $\dot W_{1/2}(t)\le -\frac12\|A\|^2
= -\frac12\|\nabla\Phi(x(t))\|^2$.

Since $H_y\ge 0$, we have $W_{1/2}(x,y)\ge \Phi(x)\ge \Phi_{\inf}:=\inf_x\Phi(x)$.
Integrating over $[0,T]$ and letting $t\sim\mathrm{Unif}[0,T]$ yields
\[
\mathbb E_t\big[\|\nabla\Phi(x(t))\|^2\big]
\le
\frac{W_{1/2}(x(0),y(0))-\Phi_{\inf}}{(1/2)\,T},
\]
which is the desired $O(1/T)$ ergodic stationarity bound.
\end{proof}

\medskip

\subsection{Proof of Corollary~\ref{cor:minimax-tracking}}\label{subsec:proof-corr}
\begin{proof}
The proof of Proposition~\ref{prop:minimax-ergodic} in fact yields a joint dissipation bound that controls
not only the outer stationarity measure $\|\nabla F(x(t))\|$ but also the inner residual
$\|\nabla_y f(x(t),y(t))\|$.
Specifically, define the Lyapunov function with $\alpha=\tfrac12$:
\[
W_{1/2}(x,y)\ :=\ F(x)\ +\ \tfrac12\bigl(F(x)-f(x,y)\bigr).
\]
Starting from the general Lyapunov derivative estimate in the proof of
Proposition~\ref{prop:minimax-ergodic}, we obtain
\[
\dot W_\alpha
\le
-\Bigl(1-\frac{(1-\alpha)\varepsilon}{2}\Bigr)\|A\|^2
+\Biggl(-\alpha\gamma+\Bigl(\alpha+\frac{1-\alpha}{2\varepsilon}\Bigr)\Bigl(\frac{L_{xy}}{\mu}\Bigr)^2\Biggr)\|C\|^2.
\]
Set $\alpha=\tfrac12$ and $\varepsilon=2$. Then
\[
1-\frac{(1-\alpha)\varepsilon}{2}=\frac12=:c_1,
\qquad
-\alpha\gamma+\Bigl(\alpha+\frac{1-\alpha}{2\varepsilon}\Bigr)\Bigl(\frac{L_{xy}}{\mu}\Bigr)^2
=-\frac{\gamma}{2}+\frac{5}{8}\Bigl(\frac{L_{xy}}{\mu}\Bigr)^2=:-c_2.
\]
Consequently, if $\gamma\ge \frac{5}{4}\left(\frac{L_{xy}}{\mu}\right)^2$ (so that $c_2\ge 0$), we have
\[
\dot W_{1/2}(t)\ \le\ -c_1\|A\|^2-c_2\|C\|^2,
\]
where $A=\nabla F(x(t))$ and $C=\nabla_y f(x(t),y(t))$.
Equivalently, this can be written as
\begin{equation}\label{eq:minimax-dissipation}
\frac{d}{dt}W_{1/2}\bigl(x(t),y(t)\bigr)
\ \le\
-\frac12\,\bigl\|\nabla F(x(t))\bigr\|^{2}
\ -\
c_y(\gamma)\,\bigl\|\nabla_y f(x(t),y(t))\bigr\|^{2},
\end{equation}
with
\begin{equation}\label{eq:cy-def}
c_y(\gamma)\ :=\ \frac{\gamma}{2}\ -\ \frac{5}{8}\Bigl(\frac{L_{xy}}{\mu}\Bigr)^{2}.
\end{equation}
In particular, the strict condition $\gamma>\tfrac54(L_{xy}/\mu)^2$ ensures $c_y(\gamma)>0$ and yields quantitative tracking.
Fix any $x$. Since $f(x,\cdot)$ is $\mu$-strongly concave, $y^*(x)$ is unique and satisfies
$\nabla_y f(x,y^*(x))=0$. 
By $\mu$-strong concavity of $f(x,\cdot)$, we have
\[
\|y-y^*(x)\|\ \le\ \frac{1}{\mu}\,\|\nabla_y f(x,y)\|.
\]
Applying this pointwise at $(x,y)=(x(t),y(t))$ and squaring yields
\[
\|y(t)-y^*(x(t))\|^2 \ \le\ \frac{1}{\mu^2}\,\|\nabla_y f(x(t),y(t))\|^2.
\]
Taking $\mathbb E_t[\cdot]$ and invoking Lemma~\ref{lem:minimax-inner-residual} completes the proof.
\end{proof}

\section{Bilevel}

\subsection{Proof of Theorem~\ref{thm:fixed-lambda-simple}}\label{subsec:proof-bilevel}
\begin{proof}

\textbf{[S1: smooth envelope.]}
\noindent By Assumption~\ref{ass:2}, for each $x$ the minimizer $y^{*}(x)\in\arg\min_{y}g(x,y)$ is unique and satisfies
\[
\nabla_y g\bigl(x,y^{*}(x)\bigr)=\mathbf{0}.
\]
Since $g^{*}(x)=g\bigl(x,y^{*}(x)\bigr)$, 
\[
\nabla g^{*}(x)
=\nabla_x g\bigl(x,y^{*}(x)\bigr)
+\bigl(Dy^{*}(x)\bigr)^{\top}\nabla_y g\bigl(x,y^{*}(x)\bigr)
=\nabla_x g\bigl(x,y^{*}(x)\bigr).
\]

\medskip
\noindent By Assumption~\ref{ass:2}, for each $x$ 
\[
y_\lambda^{*}(x)\in\arg\min_{y}\mathcal L_\lambda(x,y)
\]
exists and is unique. Hence it satisfies the condition
\[
\nabla_y \mathcal L_\lambda\bigl(x,y_\lambda^{*}(x)\bigr)=\mathbf{0}.
\]
By definition of the value function,
\[
\mathcal L_\lambda^{*}(x)=\min_{y}\mathcal L_\lambda(x,y)=\mathcal L_\lambda\bigl(x,y_\lambda^{*}(x)\bigr).
\]
\[
\nabla \mathcal L_\lambda^{*}(x)
=\nabla_x \mathcal L_\lambda\bigl(x,y_\lambda^{*}(x)\bigr)
+\bigl(Dy_\lambda^{*}(x)\bigr)^{\top}\nabla_y \mathcal L_\lambda\bigl(x,y_\lambda^{*}(x)\bigr)
=\nabla_x \mathcal L_\lambda\bigl(x,y_\lambda^{*}(x)\bigr).
\]

\medskip
\noindent We now match the unified notation used in Section~\ref{sec:unified}.
Set the outer value $\Phi(x):=\mathcal L_\lambda^{*}(x)$ and define the gap terms
\[
H_y(x,y)\ :=\ \mathcal L_\lambda(x,y)-\mathcal L_\lambda^{*}(x)\ \ge 0,
\qquad
H_z(x,z)\ :=\ g(x,z)-g^{*}(x)\ \ge 0.
\]
Then the Lyapunov function can be written equivalently as
\[
W(x,y,z)\ =\ \Phi(x)+\alpha H_y(x,y)+\beta\lambda H_z(x,z)
          \ =\ (1-\alpha)\mathcal L_\lambda^{*}(x)+\alpha \mathcal L_\lambda(x,y)+\beta\lambda H_z(x,z).
\]

\medskip
\noindent Let
\[
A:=\nabla_x\mathcal L_\lambda^{*}(x)=\nabla \Phi(x),\qquad
r_y:=\nabla_y \mathcal L_\lambda(x,y)=\nabla_y f(x,y)+\lambda\nabla_y g(x,y),\qquad
r_z:=\nabla_y g(x,z).
\]
(So $r_y$ and $r_z$ correspond to $(r_y,r_z)$ in Table~\ref{tab:normal-form-dictionary_long}.)

Define the $x$--gradient deviations (mismatches) in unified $x$-direction notation:
\[
D_y:=\nabla_x \mathcal L_\lambda(x,y)-\nabla_x \mathcal L_\lambda\bigl(x,y_\lambda^{*}(x)\bigr),
\qquad
D_z:=\nabla_x g(x,z)-\nabla g^{*}(x).
\]
(These are exactly the previous $D_y$ and $D_z$, but we rename them to emphasize they are \emph{$x$-gradient} deviations.)

\textbf{[S2: deviation-to-residual conversion.]}
\smallskip
\noindent For all $(x,y,z)$ with $W(x,y,z)\le W_0$, there exists $C_y>0$ such that 
\begin{equation}\label{eq:Dy-bound}
\|D_y\|
=\Bigl\|\nabla_x \mathcal L_\lambda(x,y)-\nabla_x \mathcal L_\lambda\bigl(x,y_\lambda^{*}(x)\bigr)\Bigr\|
\ \le\ C_y\,\|\nabla_y \mathcal L_\lambda(x,y)\|
= C_y\,\|r_y\|.
\end{equation}

\smallskip
\noindent Since 
\[
\nabla g^{*}(x)=\nabla_x g\bigl(x,y^{*}(x)\bigr),
\]
\[
D_z
=\nabla_x g(x,z)-\nabla g^{*}(x)
=\nabla_x g(x,z)-\nabla_x g\bigl(x,y^{*}(x)\bigr).
\]
By Assumption~\ref{ass:2}.3,
\[
\|D_z\|
=\bigl\|\nabla_x g(x,z)-\nabla_x g\bigl(x,y^{*}(x)\bigr)\bigr\|
\ \le\ L_{gx}\,\|z-y^{*}(x)\|.
\]
Finally, by the strong-convexity,
\[
\|z-y^{*}(x)\|\ \le\ \frac{1}{\mu}\,\|\nabla_y g(x,z)\|
=\frac{1}{\mu}\|r_z\|.
\]
Combining the last two displays yields
\begin{equation}\label{eq:Dz-bound}
\|D_z\|\ \le\ \frac{L_{gx}}{\mu}\,\|r_z\|.
\end{equation}

\textbf{[S3: absorption via time-scale separation.]}
\noindent Recall
\[
E_{\lambda}(x,y,z)=f(x,y)+\lambda\bigl(g(x,y)-g(x,z)\bigr).
\]
Differentiating with respect to each variable yields
\[
\nabla_x E_\lambda(x,y,z)=\nabla_x f(x,y)+\lambda\bigl(\nabla_x g(x,y)-\nabla_x g(x,z)\bigr),
\]
\[
\nabla_y E_\lambda(x,y,z)=\nabla_y f(x,y)+\lambda\nabla_y g(x,y),
\qquad
\nabla_z E_\lambda(x,y,z)=-\lambda\nabla_y g(x,z).
\]
Using $r_y=\nabla_y f(x,y)+\lambda\nabla_y g(x,y)$ and $r_z=\nabla_y g(x,z)$, we obtain
\[
\nabla_y E_\lambda(x,y,z)=r_y,
\qquad
\nabla_z E_\lambda(x,y,z)=-\lambda r_z.
\]

\medskip
\noindent Since $D_z:=\nabla_x g(x,z)-\nabla g^{*}(x)$, we have
\[
\nabla_x g(x,z)=\nabla g^{*}(x)+D_z.
\]
Substituting this into $\nabla_x E_\lambda$ gives
\begin{align*}
\nabla_x E_\lambda(x,y,z)
&=\nabla_x f(x,y)+\lambda\Bigl(\nabla_x g(x,y)-\nabla_x g(x,z)\Bigr)\\
&=\nabla_x f(x,y)+\lambda\Bigl(\nabla_x g(x,y)-\nabla g^{*}(x)-D_z\Bigr)\\
&=\Bigl(\nabla_x f(x,y)+\lambda(\nabla_x g(x,y)-\nabla g^{*}(x))\Bigr)
-\lambda D_z\\
&=\nabla_x \mathcal L_\lambda(x,y)-\lambda D_z.
\end{align*}
Hence 
\[
\dot x=-\nabla_x E_\lambda(x,y,z)
=-\nabla_x \mathcal L_\lambda(x,y)+\lambda D_z.
\]
Therefore,
\begin{equation}\label{eq:xdot}
\dot x=-A-D_y+\lambda D_z.
\end{equation}

\medskip
\noindent Since 
\[
\dot{\mathcal L}_\lambda^{*}(x(t))=\big\langle\nabla_x\mathcal L_\lambda^{*}(x),\dot x\big\rangle
=\langle A,\dot x\rangle.
\]
Substituting \eqref{eq:xdot},
\begin{align}
\dot{\mathcal L}_\lambda^{*}
&=\langle A,-A-D_y+\lambda D_z\rangle
\nonumber\\
&=-\langle A,A\rangle-\langle A,D_y\rangle+\lambda\langle A,D_z\rangle
\nonumber\\
&=-\|A\|^{2}-\langle A,D_y\rangle+\lambda\langle A,D_z\rangle.
\label{eq:Lstar-dot}
\end{align}

\begin{equation}\label{eq:Lxy-chain}
\frac{d}{dt}\mathcal L_\lambda(x,y)
=\big\langle\nabla_x \mathcal L_\lambda(x,y),\dot x\big\rangle
+\big\langle\nabla_y \mathcal L_\lambda(x,y),\dot y\big\rangle
=\langle A+D_y,\dot x\rangle+\langle r_y,\dot y\rangle.
\end{equation}
Substituting \eqref{eq:xdot} into \eqref{eq:Lxy-chain} gives
\begin{align}
\frac{d}{dt}\mathcal L_\lambda(x,y)
&=\langle A+D_y,-A-D_y+\lambda D_z\rangle-\delta\|r_y\|^{2}
\nonumber\\
&=-\langle A+D_y,A+D_y\rangle+\lambda\langle A+D_y,D_z\rangle-\delta\|r_y\|^{2}
\nonumber\\
&=-\|A+D_y\|^{2}+\lambda\langle A+D_y,D_z\rangle-\delta\|r_y\|^{2}.
\label{eq:Lxy-dot}
\end{align}

\medskip
\noindent For $H_z(x,z)=g(x,z)-g^{*}(x)$, we have
\[
\nabla_x H_z(x,z)=\nabla_x g(x,z)-\nabla g^{*}(x)=D_z,
\qquad
\nabla_z H_z(x,z)=\nabla_y g(x,z)=r_z,
\]
so the chain rule yields
\begin{equation}\label{eq:Hdot-chain}
\dot H_z(x,z)=\langle D_z,\dot x\rangle+\langle r_z,\dot z\rangle.
\end{equation}
\[
\langle D_z,\dot x\rangle
=\langle D_z,-A-D_y+\lambda D_z\rangle
=-\langle D_z,A\rangle-\langle D_z,D_y\rangle+\lambda\|D_z\|^{2},
\]
\[
\langle r_z,\dot z\rangle=\langle r_z,-\eta\lambda r_z\rangle=-\eta\lambda\|r_z\|^{2}.
\]
Substituting into \eqref{eq:Hdot-chain} gives
\begin{equation}\label{eq:Hdot-z}
\dot H_z(x,z)
=-\langle D_z,A\rangle-\langle D_z,D_y\rangle+\lambda\|D_z\|^{2}-\eta\lambda\|r_z\|^{2}.
\end{equation}

\medskip
\noindent Recall that
\[
W(x,y,z)=(1-\alpha)\mathcal L_\lambda^{*}(x)+\alpha \mathcal L_\lambda(x,y)+\beta\lambda H_z(x,z).
\]
\noindent Now we have
\begin{align}
\dot W
&=(1-\alpha)\dot{\mathcal L}_\lambda^{*}+\alpha\frac{d}{dt}\mathcal L_\lambda(x,y)+\beta\lambda \dot H_z(x,z)
\nonumber\\
&=(1-\alpha)\Bigl(-\|A\|^{2}-\langle A,D_y\rangle+\lambda\langle A,D_z\rangle\Bigr)
+\alpha\Bigl(-\|A+D_y\|^{2}+\lambda\langle A+D_y,D_z\rangle-\delta\|r_y\|^{2}\Bigr)
\nonumber\\
&\quad+\beta\lambda\Bigl(-\langle D_z,A\rangle-\langle D_z,D_y\rangle+\lambda\|D_z\|^{2}-\eta\lambda\|r_z\|^{2}\Bigr).
\label{eq:Wdot-expand}
\end{align}
Collecting like terms in \eqref{eq:Wdot-expand} gives
\begin{align}
\dot W
&= -(1-\alpha)\|A\|^{2}
   -(1-\alpha)\langle A,D_y\rangle
   -\alpha\|A+D_y\|^{2}
\nonumber\\
&\quad +\lambda(1-\beta)\langle A,D_z\rangle
       +\lambda(\alpha-\beta)\langle D_y,D_z\rangle
       -\alpha\delta\|r_y\|^{2}
       +\beta\lambda^{2}\|D_z\|^{2}
       -\beta\eta\lambda^{2}\|r_z\|^{2}.
\label{eq:Wdot-raw}
\end{align}

\noindent Fix $\varepsilon_1,\varepsilon_2>0$ and $\theta>0$. By Cauchy--Schwarz and Young's inequality,
\[
-(1-\alpha)\langle A,D_y\rangle
\le (1-\alpha)\Bigl(\frac{\varepsilon_1}{2}\|A\|^{2}+\frac{1}{2\varepsilon_1}\|D_y\|^{2}\Bigr),
\]
\[
\lambda(1-\beta)\langle A,D_z\rangle
\le \lambda(1-\beta)\Bigl(\frac{\varepsilon_2}{2}\|A\|^{2}+\frac{1}{2\varepsilon_2}\|D_z\|^{2}\Bigr),
\]
and
\[
\lambda(\alpha-\beta)\langle D_y,D_z\rangle
\le
\lambda|\alpha-\beta|\,\|D_y\|\,\|D_z\|
\le
\frac{\theta|\alpha-\beta|}{2}\|D_y\|^2
+
\frac{\lambda^2|\alpha-\beta|}{2\theta}\|D_z\|^2.
\]
Moreover, since $-\alpha\|A+D_y\|^2\le 0$, we drop it from the upper bound. 
\begin{equation}\label{eq:Wdot-mid-bilevel}
\dot W
\le -c_1\|A\|^{2}
     +K_y\|D_y\|^{2}
     +K_z\|D_z\|^{2}
     -\alpha\delta\|r_y\|^{2}
     -\beta\eta\lambda^{2}\|r_z\|^{2},
\end{equation}
where
\begin{equation}\label{eq:c1-def}
c_1:= (1-\alpha)\Bigl(1-\frac{\varepsilon_1}{2}\Bigr)-\lambda(1-\beta)\frac{\varepsilon_2}{2},
\end{equation}
\begin{equation}\label{eq:KyKz-def}
K_y:=\frac{1-\alpha}{2\varepsilon_1}
+\frac{\theta|\alpha-\beta|}{2},
\qquad
K_z:=\frac{\lambda(1-\beta)}{2\varepsilon_2}
+\beta\lambda^{2}
+\frac{\lambda^2|\alpha-\beta|}{2\theta}.
\end{equation}

\smallskip
\noindent Choice of $\varepsilon_2$ to keep $c_1$ positive for large $\lambda$.
Fix any $\varepsilon_1\in(0,2)$ and choose
\begin{equation}\label{eq:eps2-choice}
\varepsilon_2:=\frac{(1-\alpha)\bigl(1-\varepsilon_1/2\bigr)}{\lambda(1-\beta)}.
\end{equation}
Then \eqref{eq:c1-def} becomes
\begin{equation}\label{eq:c1-pos}
c_1=\frac{1-\alpha}{2}\Bigl(1-\frac{\varepsilon_1}{2}\Bigr)\ >\ 0,
\end{equation}
and in particular $c_1$ is independent of $\lambda$ under \eqref{eq:eps2-choice}.
Moreover, with \eqref{eq:eps2-choice},
\begin{equation}\label{eq:Kz-expanded}
K_z
=\lambda^2\Biggl(
\beta
+\frac{(1-\beta)^2}{2(1-\alpha)\bigl(1-\varepsilon_1/2\bigr)}
+\frac{|\alpha-\beta|}{2\theta}
\Biggr).
\end{equation}

\medskip
\noindent By \eqref{eq:Dy-bound}, for all $(x,y,z)$ in the sublevel set $\{W\le W_0\}$,
\[
\|D_y\|\le C_y\,\|r_y\|,
\]
where we emphasize that $C_y=C_y(\lambda)$ may depend on $\lambda$.
Next, by \eqref{eq:Dz-bound}, define
\[
C_z:=\frac{L_{gx}}{\mu},
\qquad\text{so that}\quad
\|D_z\|^2\le C_z^2\|r_z\|^2.
\]
Substituting these bounds into \eqref{eq:Wdot-mid-bilevel} yields
\begin{equation}\label{eq:Wdot-mid-absorb}
\dot W
\le -c_1\|A\|^{2}
   +\bigl(K_y C_y(\lambda)^{2}-\alpha\delta\bigr)\|r_y\|^{2}
   +\bigl(K_z C_z^{2}-\beta\eta\lambda^{2}\bigr)\|r_z\|^{2}.
\end{equation}

\smallskip
\noindent
Using \eqref{eq:KyKz-def} and \eqref{eq:Kz-expanded}, define
\begin{equation}\label{eq:delta0-explicit}
\delta_0(\lambda)
:=\frac{K_y\,C_y(\lambda)^2}{\alpha}
=
\frac{C_y(\lambda)^2}{\alpha}
\left(
\frac{1-\alpha}{2\varepsilon_1}
+
\frac{\theta|\alpha-\beta|}{2}
\right),
\end{equation}
and
\begin{align}
\eta_0
&:=\frac{K_z\,C_z^2}{\beta\lambda^2} \nonumber\\
&=
\frac{C_z^2}{\beta}
\left(
\beta
+
\frac{(1-\beta)^2}{2(1-\alpha)\bigl(1-\varepsilon_1/2\bigr)}
+
\frac{|\alpha-\beta|}{2\theta}
\right).
\label{eq:eta0-explicit}
\end{align}

Then for any $\delta\ge\delta_0(\lambda)$ and $\eta\ge\eta_0$, define
\[
c_2:=\alpha\delta-K_y C_y(\lambda)^{2}>0,
\qquad
c_3:=\beta\eta\lambda^{2}-K_z C_z^{2}>0.
\]
With these choices,
\[
K_y C_y(\lambda)^{2}-\alpha\delta=-c_2,
\qquad
K_z C_z^{2}-\beta\eta\lambda^{2}=-c_3,
\]
and therefore \eqref{eq:Wdot-mid-absorb} implies
\begin{equation}\label{eq:Wdot-final}
\dot W
\le -c_1\|A\|^{2}-c_2\|r_y\|^{2}-c_3\|r_z\|^{2}.
\end{equation}

Integrating \eqref{eq:Wdot-final} over $[0,T]$ yields
\[
W(x(T),y(T),z(T)) - W(x(0),y(0),z(0))
\ \le\
-\,c_1\int_{0}^{T}\|\nabla_x\mathcal L_\lambda^{*}(x(t))\|^{2}\,dt.
\]
Dropping the nonpositive left-hand side term and using
$W(x(T),y(T),z(T))\ge \mathcal L_\lambda^{*}(x(T))\ge \mathcal L_{\lambda,\inf}^{*}$, we obtain
\[
c_1\int_{0}^{T}\|\nabla_x\mathcal L_\lambda^{*}(x(t))\|^{2}\,dt
\ \le\
W(x(0),y(0),z(0))-\mathcal L_{\lambda,\inf}^{*}.
\]
Letting $t\sim\mathrm{Unif}[0,T]$ gives
\[
\mathbb E_t\!\left[\|\nabla_x\mathcal L_\lambda^{*}(x(t))\|^{2}\right]
=
\frac{1}{T}\int_{0}^{T}\|\nabla_x\mathcal L_\lambda^{*}(x(t))\|^{2}\,dt
\ \le\
\frac{W(x(0),y(0),z(0))-\mathcal L_{\lambda,\inf}^{*}}{c_1\,T},
\]
which proves the theorem.
\end{proof}

\medskip 

\subsection{Proof of Lemma~\ref{lem:bridge}}\label{lem-bridgeeel}
\begin{proof}
Fix $x$ and abbreviate
$y^*:=y^*(x)$, $y_\lambda^*:=y_\lambda^*(x)$, and $\Delta:=y_\lambda^*-y^*$. Since $y^*$ is the unique minimizer of $g(x,\cdot)$, Danskin's theorem yields
$\nabla_x g^*(x)=\nabla_x g(x,y^*)$.
Moreover, by the envelope theorem applied to $\mathcal L_\lambda^{*}(x)=\min_y \mathcal L_\lambda(x,y)$,
\begin{equation}\label{eq:env-L}
\nabla_x \mathcal L_\lambda^{*}(x)
=
\nabla_x f(x,y_\lambda^*)
+\lambda\nabla_x g(x,y_\lambda^*)
-\lambda\nabla_x g(x,y^*).
\end{equation}
On the other hand, differentiating $F(x)=f(x,y^*(x))$ via the implicit-function formula
(see \eqref{eq:hypergrad}) gives
\begin{equation}\label{eq:env-F}
\nabla F(x)
=
\nabla_x f(x,y^*)
-
\nabla_{xy}^{2}g(x,y^*)\,\bigl(\nabla_{yy}^{2}g(x,y^*)\bigr)^{-1}\,\nabla_y f(x,y^*).
\end{equation}

The first-order optimality conditions for $y^*$ and $y_\lambda^*$ are
\[
\nabla_y g(x,y^*)=0,
\qquad
\nabla_y f(x,y_\lambda^*)+\lambda\nabla_y g(x,y_\lambda^*)=0.
\]
By strong convexity of $g(x,\cdot)$,
\[
\|\Delta\|
=\|y_\lambda^*-y^*\|
\ \le\ \frac{1}{\mu}\,\|\nabla_y g(x,y_\lambda^*)\|
=\frac{1}{\mu\lambda}\,\|\nabla_y f(x,y_\lambda^*)\|.
\]
\[
\|\nabla_y f(x,y_\lambda^*)\|
\le \|\nabla_y f(x,y^*)\| + L_{fy}\|\Delta\|.
\]
Combining the last two displays yields
\[
\|\Delta\|
\le
\frac{1}{\mu\lambda}\|\nabla_y f(x,y^*)\|
+\frac{L_{fy}}{\mu\lambda}\|\Delta\|.
\]
For any $\lambda\ge \bar\lambda=2L_{fy}/\mu$, we have $1-\frac{L_{fy}}{\mu\lambda}\ge \frac12$, and hence
\begin{equation}\label{eq:Delta-bound}
\|\Delta\|\ \le\ \frac{2}{\mu\lambda}\,\|\nabla_y f(x,y^*)\|.
\end{equation}
Moreover, the same estimates imply
\begin{equation}\label{eq:fy-at-ylam}
\|\nabla_y f(x,y_\lambda^*)\|
\le \|\nabla_y f(x,y^*)\| + L_{fy}\|\Delta\|
\le 2\,\|\nabla_y f(x,y^*)\|.
\end{equation}

We now rewrite the term
$\lambda\bigl(\nabla_x g(x,y_\lambda^{*})-\nabla_x g(x,y^{*})\bigr)$
in a form that matches the hyper-gradient structure
$\nabla_{xy}^{2}g(\nabla_{yy}^{2}g)^{-1}\nabla_y f$ appearing in \eqref{eq:env-F}.

\medskip
\noindent\emph{(i) Relating $\Delta$ to an inverse Hessian times $\nabla_y f$.}\\
By the fundamental theorem of calculus applied to the map $y\mapsto \nabla_y g(x,y)$ along the segment
$y^{*}\to y_\lambda^{*}=y^{*}+\Delta$, we have
\[
\nabla_y g(x,y_\lambda^{*})-\nabla_y g(x,y^{*})
=\int_{0}^{1} \nabla_{yy}^{2}g\bigl(x,y^{*}+s\Delta\bigr)\,\Delta\,ds.
\]
Since $y^{*}$ is the minimizer of $g(x,\cdot)$, $\nabla_y g(x,y^{*})=0$, and thus the left-hand side equals
$\nabla_y g(x,y_\lambda^{*})$.
Define the averaged Hessian
\[
\bar B\ :=\ \int_{0}^{1} \nabla_{yy}^{2}g\bigl(x,y^{*}+s\Delta\bigr)\,ds,
\qquad\text{so that}\qquad
\nabla_y g(x,y_\lambda^{*})=\bar B\,\Delta.
\]
By $\mu$-strong convexity of $g(x,\cdot)$, we have $\nabla_{yy}^{2}g(x,\cdot)\succeq \mu I$, hence
$\bar B\succeq \mu I$ and $\|\bar B^{-1}\|\le 1/\mu$.

On the other hand, the optimality condition defining $y_\lambda^{*}$ is
$\nabla_y f(x,y_\lambda^{*})+\lambda\nabla_y g(x,y_\lambda^{*})=0$, i.e.,
\[
\nabla_y g(x,y_\lambda^{*})=-\frac{1}{\lambda}\nabla_y f(x,y_\lambda^{*}).
\]
Combining this with $\nabla_y g(x,y_\lambda^{*})=\bar B\Delta$ yields the key identity
\[
\lambda\Delta\ =\ -\,\bar B^{-1}\nabla_y f(x,y_\lambda^{*}),
\]
which already has the desired ``inverse-Hessian times $\nabla_y f$'' form.

\medskip
\noindent\emph{(ii) Converting the $x$-gradient difference to a mixed-Hessian integral.}\\
Next, apply the fundamental theorem of calculus to the map $y\mapsto \nabla_x g(x,y)$:
\[
\nabla_x g(x,y_\lambda^{*})-\nabla_x g(x,y^{*})
=
\int_{0}^{1} \nabla_{xy}^{2}g\bigl(x,y^{*}+s\Delta\bigr)\,\Delta\,ds.
\]
Multiplying by $\lambda$ and substituting $\lambda\Delta=-\bar B^{-1}\nabla_y f(x,y_\lambda^{*})$ gives
\begin{equation}\label{eq:gx-diff}
\lambda\bigl(\nabla_x g(x,y_\lambda^{*})-\nabla_x g(x,y^{*})\bigr)
=
-\int_{0}^{1} \nabla_{xy}^{2}g\bigl(x,y^{*}+s\Delta\bigr)\,\bar B^{-1}\,\nabla_y f(x,y_\lambda^{*})\,ds.
\end{equation}
This representation makes explicit that the penalty term is an averaged version of
$\nabla_{xy}^{2}g(\nabla_{yy}^{2}g)^{-1}\nabla_y f$, which is precisely the second term in the hyper-gradient
\eqref{eq:env-F}.

Combining \eqref{eq:env-L}, \eqref{eq:env-F}, and \eqref{eq:gx-diff}, we can write
\begin{equation}\label{eq:bridge-decomp}
\begin{aligned}
\nabla_x \mathcal L_\lambda^{*}(x)-\nabla F(x)
&=
\underbrace{\bigl(\nabla_x f(x,y_\lambda^{*})-\nabla_x f(x,y^{*})\bigr)}_{(\mathrm{I})}\\
&\quad+\underbrace{\Biggl(
\nabla_{xy}^{2}g(x,y^{*})\bigl(\nabla_{yy}^{2}g(x,y^{*})\bigr)^{-1}\nabla_y f(x,y^{*})
-\int_{0}^{1}\nabla_{xy}^{2}g(x,y^{*}+s\Delta)\,\bar B^{-1}\,\nabla_y f(x,y_\lambda^{*})\,ds
\Biggr)}_{(\mathrm{II})}.
\end{aligned}
\end{equation}

For term~($\mathrm{I}$), Assumptions~\ref{ass:2}.2 gives $\|(\mathrm{I})\|\le L_{fx}\|\Delta\|$.

For term~($\mathrm{II}$), add and subtract
$\nabla_{xy}^{2}g(x,y^*)\bar B^{-1}\nabla_y f(x,y_\lambda^*)$ and use the triangle inequality:
\begin{align*}
\|(\mathrm{II})\|
&\le
\underbrace{\left\|
\nabla_{xy}^{2}g(x,y^*)\bigl(\nabla_{yy}^{2}g(x,y^*)\bigr)^{-1}
\bigl(\nabla_y f(x,y^*)-\nabla_y f(x,y_\lambda^*)\bigr)
\right\|}_{(\mathrm{II}\text{-}a)}\\
&\quad+
\underbrace{\left\|
\nabla_{xy}^{2}g(x,y^*)\Bigl(\bigl(\nabla_{yy}^{2}g(x,y^*)\bigr)^{-1}-\bar B^{-1}\Bigr)\nabla_y f(x,y_\lambda^*)
\right\|}_{(\mathrm{II}\text{-}b)}\\
&\quad+
\underbrace{\left\|
\int_{0}^{1}\bigl(\nabla_{xy}^{2}g(x,y^*)-\nabla_{xy}^{2}g(x,y^*+s\Delta)\bigr)\bar B^{-1}\nabla_y f(x,y_\lambda^*)\,ds
\right\|}_{(\mathrm{II}\text{-}c)}.
\end{align*}

($\mathrm{II}$-a) Using $\|\nabla_{xy}^{2}g(x,y^*)\|\le M_{gxy}$, $\|(\nabla_{yy}^{2}g(x,y^*))^{-1}\|\le 1/\mu$,
\[
\|(\mathrm{II}\text{-}a)\|
\le
\frac{M_{gxy}}{\mu}\,L_{fy}\,\|\Delta\|.
\]

($\mathrm{II}$-b)  First, $\|\nabla_{xy}^{2}g(x,y^*)\|\le M_{gxy}$ and \eqref{eq:fy-at-ylam} gives
$\|\nabla_y f(x,y_\lambda^*)\|\le 2\|\nabla_y f(x,y^*)\|$.
Next, since $\bar B\succeq \mu I$, we have $\|\bar B^{-1}\|\le 1/\mu$, and hence
\[
\Bigl\|(\nabla_{yy}^{2}g(x,y^*))^{-1}-\bar B^{-1}\Bigr\|
\le
\|(\nabla_{yy}^{2}g(x,y^*))^{-1}\|\,\|\nabla_{yy}^{2}g(x,y^*)-\bar B\|\,\|\bar B^{-1}\|
\le
\frac{1}{\mu^2}\,\|\nabla_{yy}^{2}g(x,y^*)-\bar B\|.
\]
\[
\|\nabla_{yy}^{2}g(x,y^*)-\bar B\|
=
\left\|\int_{0}^{1}\bigl(\nabla_{yy}^{2}g(x,y^*)-\nabla_{yy}^{2}g(x,y^*+s\Delta)\bigr)\,ds\right\|
\le
\int_{0}^{1} L_{gyy}\,s\|\Delta\|\,ds
=
\frac{L_{gyy}}{2}\|\Delta\|.
\]
Therefore,
\[
\|(\mathrm{II}\text{-}b)\|
\le
M_{gxy}\cdot \frac{L_{gyy}}{2\mu^2}\|\Delta\|\cdot 2\|\nabla_y f(x,y^*)\|
=
\frac{M_{gxy}L_{gyy}}{\mu^2}\,\|\nabla_y f(x,y^*)\|\,\|\Delta\|.
\]

($\mathrm{II}$-c) $\|\bar B^{-1}\|\le 1/\mu$, and \eqref{eq:fy-at-ylam},
\[
\|(\mathrm{II}\text{-}c)\|
\le
\int_{0}^{1} L_{gxy}\,s\|\Delta\|\,ds\cdot \frac{1}{\mu}\cdot 2\|\nabla_y f(x,y^*)\|
=
\frac{L_{gxy}}{\mu}\,\|\nabla_y f(x,y^*)\|\,\|\Delta\|.
\]

Combining these bounds yields
\[
\|\nabla_x \mathcal L_\lambda^{*}(x)-\nabla F(x)\|
\le
\Biggl(
L_{fx}
+\frac{M_{gxy}L_{fy}}{\mu}
+\Bigl(\frac{L_{gxy}}{\mu}+\frac{M_{gxy}L_{gyy}}{\mu^2}\Bigr)\|\nabla_y f(x,y^*)\|
\Biggr)\,\|\Delta\|.
\]
Finally, applying \eqref{eq:Delta-bound} gives
\[
\|\nabla_x \mathcal L_\lambda^{*}(x)-\nabla F(x)\|
\le
\frac{1}{\lambda}\left[
\frac{2}{\mu}\Bigl(L_{fx}+\frac{M_{gxy}L_{fy}}{\mu}\Bigr)\|\nabla_y f(x,y^*)\|
+\frac{2}{\mu^2}\Bigl(L_{gxy}+\frac{M_{gxy}L_{gyy}}{\mu}\Bigr)\|\nabla_y f(x,y^*)\|^2
\right].
\]
\end{proof} 

\subsection{Proof of Corollary~\ref{cor:bilevel-eps}}\label{cor:bilevel-epsss}
\begin{proof}
By Lemma~\ref{lem:bridge} and \eqref{eq:Gstar-assump}, for all $x$ in the region of interest we have
$\|\nabla_x \mathcal L_\lambda^{*}(x)-\nabla F(x)\|\le C_{\mathrm{br}}/\lambda$, hence
\[
\|\nabla F(x)\|\ \le\ \|\nabla_x \mathcal L_\lambda^{*}(x)\|+\frac{C_{\mathrm{br}}}{\lambda}.
\]
Using $(a+b)^2\le 2a^2+2b^2$ yields
\begin{equation}\label{eq:bridge-square}
\|\nabla F(x)\|^2
\ \le\
2\|\nabla_x \mathcal L_\lambda^{*}(x)\|^2
+2\Bigl(\frac{C_{\mathrm{br}}}{\lambda}\Bigr)^2.
\end{equation}
Taking expectation over $t\sim\mathrm{Unif}[0,T]$ gives
\begin{equation}\label{eq:bridge-square-erg}
\mathbb E_t\bigl[\|\nabla F(x(t))\|^2\bigr]
\ \le\
2\,\mathbb E_t\bigl[\|\nabla_x \mathcal L_\lambda^{*}(x(t))\|^2\bigr]
+2\Bigl(\frac{C_{\mathrm{br}}}{\lambda}\Bigr)^2.
\end{equation}

By the choice of $\lambda$, we have $\lambda\ge 2C_{\mathrm{br}}/\sqrt{\varepsilon}$, hence
\[
2\Bigl(\frac{C_{\mathrm{br}}}{\lambda}\Bigr)^2\ \le\ \frac{\varepsilon}{2}.
\]
Moreover, by Theorem~\ref{thm:fixed-lambda-simple},
\[
\mathbb E_t\bigl[\|\nabla_x \mathcal L_\lambda^{*}(x(t))\|^2\bigr]
\ \le\ \frac{W(x(0),y(0),z(0))-\mathcal L_{\lambda,\inf}^{*}}{c_1T}.
\]
Under the condition
\[
T\ \ge\ \frac{4\bigl(W(x(0),y(0),z(0))-\mathcal L_{\lambda,\inf}^{*}\bigr)}{c_1\,\varepsilon},
\]
we obtain
\[
2\,\mathbb E_t\bigl[\|\nabla_x \mathcal L_\lambda^{*}(x(t))\|^2\bigr]\ \le\ \frac{\varepsilon}{2}.
\]
Substituting the last two displays into \eqref{eq:bridge-square-erg} yields
$\mathbb E_t\bigl[\|\nabla F(x(t))\|^2\bigr]\le \varepsilon$.
\end{proof}

\subsection{Deriving Corollary~\ref{cor:bilevel-tracking} from an Intermediate Lemma~\ref{lem:bilevel-residuals}}
\begin{lemma}\label{lem:bilevel-residuals}
Assume the conditions of Theorem~\ref{thm:fixed-lambda-simple}.
Recall the residuals along the trajectory $(x(t),y(t),z(t))$ of \eqref{eq:flow-bilevel-unified}:
\[
r_y(t):=\nabla_y \mathcal L_\lambda(x(t),y(t))
=\nabla_y f(x(t),y(t))+\lambda\nabla_y g(x(t),y(t)),
\qquad
r_z(t):=\nabla_y g(x(t),z(t)).
\]
Let $\delta_0(\lambda)$ and $\eta_0$ be defined in \eqref{eq:eta0-bilevel}.
Suppose that
\[
\delta>\delta_0(\lambda),
\qquad
\eta>\eta_0,
\]
and define the positive margins
\[
c_y(\delta,\lambda):=\alpha\bigl(\delta-\delta_0(\lambda)\bigr),
\qquad
c_z(\eta,\lambda):=\beta\lambda^2\bigl(\eta-\eta_0\bigr).
\]
Then for $t\sim\mathrm{Unif}[0,T]$,
\begin{align}
\mathbb E_t\bigl[\|r_y(t)\|^2\bigr]
&\le
\frac{W(x(0),y(0),z(0))-\mathcal L_{\lambda,\inf}^{*}}{c_y(\delta,\lambda)\,T},
\label{eq:ergodic-F-lemma}\\
\mathbb E_t\bigl[\|r_z(t)\|^2\bigr]
&\le
\frac{W(x(0),y(0),z(0))-\mathcal L_{\lambda,\inf}^{*}}{c_z(\eta,\lambda)\,T}.
\label{eq:ergodic-Gz-lemma}
\end{align}
\end{lemma}

\begin{proof} 
The proof of Theorem~\ref{thm:fixed-lambda-simple} establishes the inequality: there exist constants $c_1>0$, $c_2>0$, and $c_3>0$ such that for a.e.\ $t\ge 0$,
\begin{equation}\label{eq:lyap-refined}
\frac{d}{dt}W(x(t),y(t),z(t))
\le
-c_1\|\nabla_x \mathcal L_\lambda^*(x(t))\|^2
-c_2\|r_y(t)\|^2
-c_3\|r_z(t)\|^2.
\end{equation}
Moreover, with the explicit definitions in the proof,
\[
c_2=\alpha\delta-K_yC_y(\lambda)^2,
\qquad
c_3=\beta\eta\lambda^2-K_zC_z^2,
\]
and the thresholds $\delta_0(\lambda)=K_yC_y(\lambda)^2/\alpha$ and $\eta_0=K_zC_z^2/(\beta\lambda^2)$
are chosen precisely so that $c_2>0$ and $c_3>0$. Under the strict inequalities $\delta>\delta_0(\lambda)$ and $\eta>\eta_0$, we can rewrite
\[
c_2=\alpha\delta-\alpha\delta_0(\lambda)=\alpha(\delta-\delta_0(\lambda))=:c_y(\delta,\lambda)>0,
\]
\[
c_3=\beta\eta\lambda^2-\beta\eta_0\lambda^2=\beta\lambda^2(\eta-\eta_0)=:c_z(\eta,\lambda)>0.
\]
Substituting these into \eqref{eq:lyap-refined} yields
\begin{equation}\label{eq:lyap-refined-margins}
\frac{d}{dt}W(x(t),y(t),z(t))
\le
-c_1\|\nabla_x \mathcal L_\lambda^*(x(t))\|^2
-c_y(\delta,\lambda)\|r_y(t)\|^2
-c_z(\eta,\lambda)\|r_z(t)\|^2.
\end{equation}

Integrating \eqref{eq:lyap-refined-margins} over $[0,T]$ gives
\begin{align*}
W(x(T),y(T),z(T)) - W(x(0),y(0),z(0))
&\le
- c_1\int_0^T \|\nabla_x \mathcal L_\lambda^*(x(t))\|^2\,dt \\
&\quad
- c_y(\delta,\lambda)\int_0^T \|r_y(t)\|^2\,dt
- c_z(\eta,\lambda)\int_0^T \|r_z(t)\|^2\,dt.
\end{align*}
Dropping the nonpositive term involving $\|\nabla_x \mathcal L_\lambda^*(x(t))\|^2$ yields the simpler inequality
\begin{equation}\label{eq:integral-two-residuals}
c_y(\delta,\lambda)\int_0^T \|r_y(t)\|^2\,dt
+
c_z(\eta,\lambda)\int_0^T \|r_z(t)\|^2\,dt
\le
W(x(0),y(0),z(0)) - W(x(T),y(T),z(T)).
\end{equation}
By definition,
\[
W(x,y,z)
=
\mathcal L_\lambda^*(x)+\alpha\widetilde H_\lambda(x,y)+\beta\lambda H(x,z),
\]
and $\widetilde H_\lambda\ge 0$, $H\ge 0$. Hence $W(x,y,z)\ge \mathcal L_\lambda^*(x)$ for all $(x,y,z)$.
Since $\mathcal L_{\lambda,\inf}^*:=\inf_x \mathcal L_\lambda^*(x)>-\infty$ by Assumption~\ref{ass:2}.4, we have
\[
W(x(T),y(T),z(T))\ge \mathcal L_\lambda^*(x(T))\ge \mathcal L_{\lambda,\inf}^*.
\]
Combining with \eqref{eq:integral-two-residuals} yields the two separate bounds
\[
c_y(\delta,\lambda)\int_0^T \|r_y(t)\|^2\,dt
\le
W(x(0),y(0),z(0))-\mathcal L_{\lambda,\inf}^*,
\qquad
c_z(\eta,\lambda)\int_0^T \|r_z(t)\|^2\,dt
\le
W(x(0),y(0),z(0))-\mathcal L_{\lambda,\inf}^*.
\]
If $t\sim\mathrm{Unif}[0,T]$, then
\[
\mathbb E_t[\|r_y(t)\|^2]=\frac{1}{T}\int_0^T \|r_y(s)\|^2\,ds,
\qquad
\mathbb E_t[\|r_z(t)\|^2]=\frac{1}{T}\int_0^T \|r_z(s)\|^2\,ds.
\]
Dividing the last two inequalities by $T$ gives \eqref{eq:ergodic-F-lemma}--\eqref{eq:ergodic-Gz-lemma}.
\end{proof}

\subsubsection{Proof of Corollary~\ref{cor:bilevel-tracking}}

\begin{proof}
Fix any $t\in[0,T]$ and write $x=x(t)$, $y=y(t)$, and $y_\lambda^*=y_\lambda^*(x)$.
Since $y_\lambda^*$ minimizes $\mathcal L_\lambda(x,\cdot)$, we have the first-order optimality condition
\[
\nabla_y \mathcal L_\lambda(x,y_\lambda^*)=\mathbf 0.
\]
By assumption, $y\mapsto \phi_\lambda(x,y)=f(x,y)+\lambda g(x,y)$ is $\mu_\lambda$-strongly convex, and
$\nabla_y\mathcal L_\lambda(x,y)=\nabla_y\phi_\lambda(x,y)$ ($-\lambda g^*(x)$ does not depend on $y$).
Therefore $\nabla_y\mathcal L_\lambda(x,\cdot)$ is $\mu_\lambda$-strongly monotone, i.e.,
for all $y$,
\[
\big\langle \nabla_y\mathcal L_\lambda(x,y)-\nabla_y\mathcal L_\lambda(x,y_\lambda^*),\; y-y_\lambda^*\big\rangle
\ge
\mu_\lambda \|y-y_\lambda^*\|^2.
\]
Using $\nabla_y\mathcal L_\lambda(x,y_\lambda^*)=\mathbf 0$ and Cauchy--Schwarz,
\[
\mu_\lambda \|y-y_\lambda^*\|^2
\le
\big\langle \nabla_y\mathcal L_\lambda(x,y),\; y-y_\lambda^*\big\rangle
\le
\|\nabla_y\mathcal L_\lambda(x,y)\|\,\|y-y_\lambda^*\|.
\]
If $y\neq y_\lambda^*$, divide both sides by $\|y-y_\lambda^*\|$ to obtain
\[
\|y-y_\lambda^*\|
\le
\frac{1}{\mu_\lambda}\,\|\nabla_y\mathcal L_\lambda(x,y)\|.
\]
If $y=y_\lambda^*$, the inequality holds trivially. Hence, for all $t$,
\begin{equation}\label{eq:pointwise-ytrack}
\|y(t)-y_\lambda^*(x(t))\|
\le
\frac{1}{\mu_\lambda}\,\|\nabla_y\mathcal L_\lambda(x(t),y(t))\|
=
\frac{1}{\mu_\lambda}\,\|r_y(t)\|.
\end{equation}
Squaring \eqref{eq:pointwise-ytrack} yields
$\|y(t)-y_\lambda^*(x(t))\|^2 \le \mu_\lambda^{-2}\|r_y(t)\|^2$ pointwise.
Taking $\mathbb E_t[\cdot]$ on both sides and using \eqref{eq:ergodic-F-lemma} gives
\[
\mathbb E_t\bigl[\|y(t)-y_\lambda^*(x(t))\|^2\bigr]
\le
\frac{1}{\mu_\lambda^2}\,\mathbb E_t[\|r_y(t)\|^2]
\le
\frac{W(x(0),y(0),z(0))-\mathcal L_{\lambda,\inf}^{*}}{\mu_\lambda^2\,c_y(\delta,\lambda)\,T}.
\]
Fix any $t$ and write $x=x(t)$, $z=z(t)$, and $y^*=y^*(x)$, where $y^*(x)\in\arg\min_y g(x,y)$.
By $\mu$-strong convexity of $g(x,\cdot)$, we have $\nabla_y g(x,\cdot)$ is $\mu$-strongly monotone, hence
\[
\big\langle \nabla_y g(x,z)-\nabla_y g(x,y^*),\; z-y^*\big\rangle
\ge
\mu\|z-y^*\|^2.
\]
Using $\nabla_y g(x,y^*)=\mathbf 0$ and Cauchy--Schwarz as above yields the pointwise bound
\begin{equation}\label{eq:pointwise-ztrack}
\|z(t)-y^*(x(t))\|\le \frac{1}{\mu}\,\|\nabla_y g(x(t),z(t))\|=\frac{1}{\mu}\,\|r_z(t)\|.
\end{equation}
Squaring gives $\|z(t)-y^*(x(t))\|^2\le \mu^{-2}\|r_z(t)\|^2$.
Taking $\mathbb E_t[\cdot]$ and using \eqref{eq:ergodic-Gz-lemma} yields
\[
\mathbb E_t\bigl[\|z(t)-y^*(x(t))\|^2\bigr]
\le
\frac{1}{\mu^2}\,\mathbb E_t[\|r_z(t)\|^2]
\le
\frac{W(x(0),y(0),z(0))-\mathcal L_{\lambda,\inf}^{*}}{\mu^2\,c_z(\eta,\lambda)\,T}.
\]
\end{proof}

\begin{corollary}\label{cor:bilevel-track-ytrue}
Assume the conditions of Corollary~\ref{cor:bilevel-tracking}.
In addition, assume the regularity conditions used in Lemma~\ref{lem:bridge} and
$\lambda\ge \bar\lambda:=2L_{fy}/\mu$. Suppose there exists $G_*>0$ such that
\[
\|\nabla_y f(x,y^*(x))\|\le G_*
\quad\text{in the region of interest.}
\]
Then for $t\sim\mathrm{Unif}[0,T]$,
\begin{equation}\label{eq:y-to-ytrue}
\mathbb E_t\bigl[\|y(t)-y^*(x(t))\|^2\bigr]
\le
\frac{2\bigl(W(x(0),y(0),z(0))-\mathcal L_{\lambda,\inf}^{*}\bigr)}{\mu_\lambda^2\,c_y(\delta,\lambda)\,T}
+\frac{8G_*^2}{\mu^2\lambda^2}.
\end{equation}
\end{corollary}

\begin{proof}
Fix $t$ and abbreviate $x=x(t)$, $y=y(t)$, $y^*=y^*(x)$, and $y_\lambda^*=y_\lambda^*(x)$.
By $(a+b)^2\le 2a^2+2b^2$,
\[
\|y-y^*\|^2
\le
2\|y-y_\lambda^*\|^2 + 2\|y_\lambda^*-y^*\|^2.
\]
Taking $\mathbb E_t[\cdot]$ and applying Corollary~\ref{cor:bilevel-tracking} gives
\[
\mathbb E_t[\|y(t)-y^*(x(t))\|^2]
\le
\frac{2\bigl(W(x(0),y(0),z(0))-\mathcal L_{\lambda,\inf}^{*}\bigr)}{\mu_\lambda^2\,c_y(\delta,\lambda)\,T}
+2\,\mathbb E_t[\|y_\lambda^*(x(t))-y^*(x(t))\|^2].
\]

It remains to bound the penalty bias $\|y_\lambda^*(x)-y^*(x)\|$.
From the derivation in Lemma~\ref{lem:bridge} (see \eqref{eq:Delta-bound}), for $\lambda\ge \bar\lambda$ we have
\[
\|y_\lambda^*(x)-y^*(x)\|
\le
\frac{2}{\mu\lambda}\,\|\nabla_y f(x,y^*(x))\|
\le
\frac{2G_*}{\mu\lambda}.
\]
Therefore $2\,\mathbb E_t[\|y_\lambda^*(x(t))-y^*(x(t))\|^2]\le 8G_*^2/(\mu^2\lambda^2)$.
Substituting this completes the proof.
\end{proof}

\section{Bilevel deviation bounds from Lemma~\ref{lem:deviation}}\label{app:deviation-bilevel}

This appendix derives the deviation-to-residual conversions used in the bilevel analysis, namely
\eqref{eq:DyDz-conv-main}. The key point is that both $x$--gradient mismatches can be controlled by
first-order inner residuals under cross-Lipschitz continuity and (weakly) strong convexity in the inner variables.

\begin{proposition}[Bilevel deviation bounds]\label{prop:bilevel-deviation}
Fix $\lambda>0$ and define
\[
\mathcal L_\lambda(x,y):=f(x,y)+\lambda\bigl(g(x,y)-g^*(x)\bigr),
\qquad
g^*(x):=\min_{u} g(x,u),
\]
and let $y_\lambda^*(x)\in\arg\min_y \mathcal L_\lambda(x,y)$.
Define the residual and deviation terms
\[
r_y:=\nabla_y \mathcal L_\lambda(x,y),
\qquad
D_y:=\nabla_x \mathcal L_\lambda(x,y)-\nabla_x \mathcal L_\lambda\bigl(x,y_\lambda^*(x)\bigr),
\]
and, for any $z$,
\[
r_z:=\nabla_y g(x,z),
\qquad
D_z:=\nabla_x g(x,z)-\nabla g^*(x).
\]
Assume Assumption~\ref{ass:2} holds. 

\smallskip
\noindent\textbf{(i) Bound for $D_z$.}
Then for all $(x,z)$,
\begin{equation}\label{eq:Dz-prop}
\|D_z\|
\ \le\
\frac{L_{gx}}{\mu}\,\|r_z\|
\ =:\ C_z\,\|r_z\|.
\end{equation}

\smallskip
\noindent\textbf{(ii) Bound for $D_y$.}
Let $\phi_\lambda(x,y):=f(x,y)+\lambda g(x,y)$ and assume that for every $x$,
the mapping $y\mapsto \phi_\lambda(x,y)$ is $\mu_\lambda$--strongly convex for some $\mu_\lambda>0$.
Then for all $(x,y)$,
\begin{equation}\label{eq:Dy-prop}
\|D_y\|
\ \le\
\frac{L_{fx}+\lambda L_{gx}}{\mu_\lambda}\,\|r_y\|
\ =:\ C_y(\lambda)\,\|r_y\|.
\end{equation}
In particular, if $y\mapsto f(x,y)$ is $\rho$--weakly convex and $y\mapsto g(x,y)$ is $\mu$--strongly convex,
then $\phi_\lambda(x,\cdot)$ is $(\lambda\mu-\rho)$--strongly convex, so \eqref{eq:Dy-prop} holds with
$\mu_\lambda=\lambda\mu-\rho>0$ whenever $\lambda>\rho/\mu$.
\end{proposition}

\begin{proof}
We first prove \eqref{eq:Dy-prop}. Define $\phi_\lambda(x,y):=f(x,y)+\lambda g(x,y)$ so that
\[
\mathcal L_\lambda(x,y)=\phi_\lambda(x,y)-\lambda g^*(x).
\]
Since $g^*(x)$ is independent of $y$, we have
\[
\nabla_y \mathcal L_\lambda(x,y)=\nabla_y \phi_\lambda(x,y),
\qquad
\nabla_x \mathcal L_\lambda(x,y)-\nabla_x \mathcal L_\lambda(x,y')
=\nabla_x \phi_\lambda(x,y)-\nabla_x \phi_\lambda(x,y').
\]
By Assumptions~\ref{ass:2},
\[
\|\nabla_x \phi_\lambda(x,y)-\nabla_x \phi_\lambda(x,y')\|
\le
\|\nabla_x f(x,y)-\nabla_x f(x,y')\|+\lambda\|\nabla_x g(x,y)-\nabla_x g(x,y')\|
\le (L_{fx}+\lambda L_{gx})\|y-y'\|.
\]
Applying Lemma~\ref{lem:deviation} to $\psi=\phi_\lambda$ yields
\[
\bigl\|\nabla_x \phi_\lambda(x,y)-\nabla_x \phi_\lambda\bigl(x,y_\lambda^*(x)\bigr)\bigr\|
\le \frac{L_{fx}+\lambda L_{gx}}{\mu_\lambda}\,\bigl\|\nabla_y \phi_\lambda(x,y)\bigr\|.
\]
Using the identities above, this is exactly \eqref{eq:Dy-prop}.

We now prove \eqref{eq:Dz-prop}. Since $g(x,\cdot)$ is $\mu$--strongly convex, its minimizer
$y^*(x)\in\arg\min_u g(x,u)$ is unique and satisfies $\nabla_y g(x,y^*(x))=0$.
By Danskin's theorem, $\nabla g^*(x)=\nabla_x g(x,y^*(x))$.
Moreover, by Assumption~\ref{ass:2}.3, $\nabla_x g(x,\cdot)$ is $L_{gx}$--Lipschitz in $y$.
Applying Lemma~\ref{lem:deviation} to $\psi=g$ (with $u=z$ and $u^*=y^*(x)$) gives
\[
\|\nabla_x g(x,z)-\nabla_x g(x,y^*(x))\|
\le \frac{L_{gx}}{\mu}\,\|\nabla_y g(x,z)\|.
\]
Substituting $D_z=\nabla_x g(x,z)-\nabla g^*(x)=\nabla_x g(x,z)-\nabla_x g(x,y^*(x))$
and $r_z=\nabla_y g(x,z)$ yields \eqref{eq:Dz-prop}.
\end{proof}


\section{Relaxing inner strong convexity/concavity: error-bound instantiations for \ref{eq:P1}--\ref{eq:P31}}
\label{app:convex-inner}
This appendix complements Section~\ref{sec:mere} by instantiating the
error bound~(Assumption~\ref{assump:inner-EB}) in place of strong convexity/concavity repair for \ref{eq:P1}--\ref{eq:P31}.
Throughout, the unified Lyapunov construction $W=\Phi+\alpha H_y+\beta H_z$ is unchanged.
Only the deviation module is replaced: instead of strong convexity/concavity, we assume a gradient error bound (EB)
that yields a deviation-to-residual conversion of the form $\|D\|\le C\|r\|$ on a region of interest
(typically a Lyapunov sublevel set).

\paragraph{Why we focus on Assumption~\ref{assump:inner-EB}~(EB).}
The Lyapunov absorption step (Lemma~\ref{lem:timescale}) only requires the conversion
$\|D\|\le C\|r\|$.
Assumption~\ref{assump:inner-EB} is exactly an interface condition that delivers this conversion.
Conditions such as PL/gradient-dominance and quadratic growth (QG) are widely used sufficient mechanisms
to guarantee an EB on the relevant sublevel set.
Hence it is enough to develop the theory under EB, and treat PL/QG as standard instantiations that imply EB
under the usual smooth convex regularity.

\subsection{Notation: PL and QG, and their implication to EB}\label{app:pl-qg-eb}

Let $h:\mathbb R^{d}\to\mathbb R$ be differentiable and let
$h^*:=\inf_{u\in\mathbb R^d} h(u)$ and $U^*:=\arg\min_u h(u)$ (assumed nonempty).
We write $\operatorname{dist}(u,U^*):=\inf_{v\in U^*}\|u-v\|$.

\paragraph{Polyak--\L{}ojasiewicz (PL) inequality.}
We say that $h$ satisfies the PL inequality on a set $\mathcal U\subseteq\mathbb R^d$ with parameter $\mu_{\mathrm{PL}}>0$ if
\begin{equation}\label{eq:PL-def}
\frac{1}{2}\|\nabla h(u)\|^2 \ \ge\ \mu_{\mathrm{PL}}\bigl(h(u)-h^*\bigr)
\qquad(\forall\,u\in\mathcal U).
\end{equation}
Equivalently, $h(u)-h^* \le \frac{1}{2\mu_{\mathrm{PL}}}\|\nabla h(u)\|^2$ on $\mathcal U$.

\paragraph{Quadratic growth (QG).}
We say that $h$ satisfies QG on $\mathcal U$ with parameter $\mu_{\mathrm{QG}}>0$ if
\begin{equation}\label{eq:QG-def}
h(u)-h^* \ \ge\ \frac{\mu_{\mathrm{QG}}}{2}\,\operatorname{dist}(u,U^*)^2
\qquad(\forall\,u\in\mathcal U).
\end{equation}

\paragraph{EB as the interface condition.}
Assumption~\ref{assump:inner-EB} specializes here to the gradient error bound
\begin{equation}\label{eq:EB-def-app}
\operatorname{dist}(u,U^*) \ \le\ \kappa\,\|\nabla h(u)\|
\qquad(\forall\,u\in\mathcal U),
\end{equation}
for some $\kappa>0$.

\paragraph{QG + convexity $\Rightarrow$ EB.}
If $h$ is convex and differentiable, then for any $u\in\mathbb R^d$ and any $u^*\in U^*$,
convexity gives
$h(u)-h^* \le \langle \nabla h(u),u-u^*\rangle \le \|\nabla h(u)\|\operatorname{dist}(u,U^*)$.
Combining this with QG~\eqref{eq:QG-def} on $\mathcal U$ yields, for all $u\in\mathcal U$,
\[
\frac{\mu_{\mathrm{QG}}}{2}\operatorname{dist}(u,U^*)^2
\ \le\ h(u)-h^*
\ \le\ \|\nabla h(u)\|\operatorname{dist}(u,U^*),
\]
and therefore
\begin{equation}\label{eq:QG-implies-EB}
\operatorname{dist}(u,U^*)\ \le\ \frac{2}{\mu_{\mathrm{QG}}}\,\|\nabla h(u)\|
\qquad(\forall\,u\in\mathcal U).
\end{equation}
Thus QG (together with convexity) implies EB with $\kappa=2/\mu_{\mathrm{QG}}$.

\paragraph{PL + QG $\Rightarrow$ EB.}
Even without invoking convexity, if both PL~\eqref{eq:PL-def} and QG~\eqref{eq:QG-def} hold on $\mathcal U$, then
\[
\operatorname{dist}(u,U^*)^2 \ \le\ \frac{2}{\mu_{\mathrm{QG}}}\bigl(h(u)-h^*\bigr)
\ \le\ \frac{1}{\mu_{\mathrm{QG}}\mu_{\mathrm{PL}}}\|\nabla h(u)\|^2,
\]
hence EB holds with $\kappa = (\mu_{\mathrm{QG}}\mu_{\mathrm{PL}})^{-1/2}$.
In many smooth convex settings, PL implies QG on sublevel sets under standard regularity;
in such cases PL alone can be used as a convenient sufficient condition for EB via the intermediate QG step.
We emphasize, however, that for our Lyapunov template EB is the only required condition:
any route that establishes EB on the relevant region can be plugged in without altering the Lyapunov function. 

\subsection{A generic deviation-to-residual conversion under an EB}\label{app:eb-interface}

We return to the parametric setting $\psi(x,u)$, where $x\in\mathbb R^{d_x}$ and $u\in\mathbb R^{d_u}$.
Let $\mathcal R\subseteq\mathbb R^{d_x}\times\mathbb R^{d_u}$ be a region of interest.
Assume $\psi$ is differentiable and convex in $u$ and define
\[
U^*(x)\ :=\ \arg\min_u \psi(x,u),
\qquad
r_u(x,u)\ :=\ \nabla_u \psi(x,u).
\]
The $x$--gradient mismatch is the deviation
\[
D_u(x,u;u^*)\ :=\ \nabla_x\psi(x,u)-\nabla_x\psi(x,u^*),
\qquad u^*\in U^*(x).
\]

Assume that $\nabla_x\psi(x,\cdot)$ is $L_{xu}$--Lipschitz on $\mathcal R$, i.e.,
\[
\|\nabla_x\psi(x,u_1)-\nabla_x\psi(x,u_2)\| \ \le\ L_{xu}\|u_1-u_2\|
\qquad(\forall\, (x,u_1),(x,u_2)\in\mathcal R).
\]
If $\psi$ satisfies the EB condition (Assumption~\ref{assump:inner-EB}) on $\mathcal R$, then Lemma~\ref{lem:deviation-EB} gives
\begin{equation}\label{eq:C-from-EB}
\inf_{u^*\in U^*(x)}
\|D_u(x,u;u^*)\|
\ \le\ L_{xu}\kappa\,\|r_u(x,u)\|
\qquad(\forall\, (x,u)\in\mathcal R),
\end{equation}
so the deviation-to-residual conversion holds with the constant
\[
C\ =\ L_{xu}\kappa.
\]
Plugging \eqref{eq:C-from-EB} into Lemma~\ref{lem:timescale} preserves the algebraic form of the time-scale thresholds;
only the conversion constant changes, and the thresholds inherit the universal dependence on $C^2$. 

\subsection{Minimax without strong convexity/concavity}\label{app:eb-p1}

In \ref{eq:P1}, the inner maximization in $y$ corresponds to minimizing $\psi(x,y):=-f(x,y)$ in $y$.
The mismatch/residual pair in Table~\ref{tab:normal-form-dictionary_long} is
\[
D_y=\nabla_x f(x,y)-\nabla_x f\bigl(x,y^*(x)\bigr),
\qquad
r_y=\nabla_y f(x,y),
\]
with effective time scale $\tau_y=\gamma$.

Assume $f(x,\cdot)$ is concave and differentiable for each $x$, and that on the Lyapunov sublevel set
$\mathcal R_1:=\{(x,y): W(x,y)\le W(x(0),y(0))\}$ the following EB holds for the maximization problem:
there exists $\kappa_f>0$ such that for all $(x,y)\in\mathcal R_1$,
\begin{equation}\label{eq:EB-P1}
\operatorname{dist}\bigl(y, Y^*(x)\bigr)\ \le\ \kappa_f\,\|\nabla_y f(x,y)\|,
\qquad
Y^*(x):=\arg\max_y f(x,y).
\end{equation}
Assume in addition that $\nabla_x f(x,\cdot)$ is $L_{xy}$--Lipschitz in $y$ on $\mathcal R_1$.

Then, combining Lipschitzness with \eqref{eq:EB-P1} yields the deviation-to-residual conversion
\[
\inf_{y^*\in Y^*(x)}
\bigl\|\nabla_x f(x,y)-\nabla_x f(x,y^*)\bigr\|
\ \le\ L_{xy}\kappa_f\,\|\nabla_y f(x,y)\|
\qquad(\forall\,(x,y)\in\mathcal R_1),
\]
so in Lemma~\ref{lem:timescale} one may take $C_y=L_{xy}\kappa_f$.
Accordingly, the time-scale absorption condition retains the same form and becomes
\[
\gamma\ \ge\ \frac{K_y}{\alpha}\,C_y^2
\ =\ \frac{K_y}{\alpha}\,(L_{xy}\kappa_f)^2.
\]
Under strong concavity with parameter $\mu$, one can take $\kappa_f=1/\mu$ and recover the strongly concave threshold.
When the EB constant is larger (weaker geometry), the required $\gamma$ increases through the same $C^2$ dependence.

\subsection{Bilevel without strong convexity/concavity}\label{app:eb-p2}

In \ref{eq:problem2}, the analysis requires two conversion inequalities (Table~\ref{tab:normal-form-dictionary_long}):
\[
\|D_y\|\le C_y(\lambda)\|r_y\|,
\qquad
\|D_z\|\le C_z\|r_z\|,
\]
with effective time scales $(\tau_y,\tau_z)=(\delta,\eta\lambda^2)$.

Let $\phi_\lambda(x,y):=f(x,y)+\lambda g(x,y)$.
Assume that on the Lyapunov sublevel set $\mathcal R_2:=\{(x,y,z): W(x,y,z)\le W(x(0),y(0),z(0))\}$:
(i) $y\mapsto \phi_\lambda(x,y)$ is convex and differentiable and admits minimizers
$Y_\lambda^*(x):=\arg\min_y \phi_\lambda(x,y)\neq\emptyset$,
(ii) $y\mapsto g(x,y)$ is convex and differentiable and admits minimizers
$Z^*(x):=\arg\min_y g(x,y)\neq\emptyset$,
and (iii) both subproblems satisfy gradient error bounds:
there exist $\kappa_{\phi,\lambda},\kappa_g>0$ such that for all $(x,y,z)\in\mathcal R_2$,
\begin{equation}\label{eq:EB-P2-phi}
\operatorname{dist}\bigl(y, Y_\lambda^*(x)\bigr)\ \le\ \kappa_{\phi,\lambda}\,\|\nabla_y\phi_\lambda(x,y)\|,
\end{equation}
and
\begin{equation}\label{eq:EB-P2-g}
\operatorname{dist}\bigl(z, Z^*(x)\bigr)\ \le\ \kappa_{g}\,\|\nabla_y g(x,z)\|.
\end{equation}
Assume moreover that $\nabla_x f(x,\cdot)$ and $\nabla_x g(x,\cdot)$ are Lipschitz in $y$ on $\mathcal R_2$ with constants
$L_{fx}$ and $L_{gx}$, respectively.

Then the generic EB-to-deviation interface \eqref{eq:C-from-EB} yields the bilevel conversion constants
\[
C_y(\lambda)\ =\ (L_{fx}+\lambda L_{gx})\,\kappa_{\phi,\lambda},
\qquad
C_z\ =\ L_{gx}\,\kappa_g,
\]
so that \eqref{eq:DyDz-conv-main} holds on $\mathcal R_2$.
Consequently, the time-scale thresholds retain the same algebraic form as in the strongly convex analysis,
with the universal substitution $1/\mu\leadsto \kappa$ inside $C_y(\lambda),C_z$.
In particular, any threshold depending on $C_y(\lambda)^2$ or $C_z^2$ deteriorates proportionally to
$\kappa_{\phi,\lambda}^2$ and $\kappa_g^2$.

\subsection{Min--min--max without strong convexity/concavity}\label{app:eb-p3}

In \ref{eq:P31}, one similarly requires EB modules for both inner problems.
Let $\mathcal R_3:=\{(x,y,z): W(x,y,z)\le W(x(0),y(0),z(0))\}$.
Assume that on $\mathcal R_3$:
(i) $f(x,\cdot)$ and $g(x,\cdot)$ are convex and differentiable with nonempty minimizer sets
$Y^*(x):=\arg\min_y f(x,y)$ and $Z^*(x):=\arg\min_z g(x,z)$,
(ii) the gradient error bounds hold with constants $\kappa_f,\kappa_g>0$:
\[
\operatorname{dist}\bigl(y,Y^*(x)\bigr)\ \le\ \kappa_f\,\|\nabla_y f(x,y)\|,
\qquad
\operatorname{dist}\bigl(z,Z^*(x)\bigr)\ \le\ \kappa_g\,\|\nabla_z g(x,z)\|,
\]
and (iii) $\nabla_x f(x,\cdot)$ is $L_{fx}$--Lipschitz in $y$ and $\nabla_x g(x,\cdot)$ is $L_{gx}$--Lipschitz in $z$
on $\mathcal R_3$.

Then the conversion constants are
\[
C_y\ =\ L_{fx}\kappa_f,
\qquad
C_z\ =\ L_{gx}\kappa_g,
\]
and Lemma~\ref{lem:timescale} yields the same absorption conditions with problem-specific effective time scales
(e.g., $\tau_y=\delta$ and $\tau_z=\eta$ in \ref{eq:P31}). 
Consequently, the time-scale thresholds in \ref{eq:P31} take the explicit form
\[
\delta_0 \;=\; \frac{K_y}{\alpha}\,(L_{fx}\kappa_f)^2,
\qquad
\eta_0 \;=\; \frac{K_z}{\beta}\,(L_{gx}\kappa_g)^2.
\]
In the strongly convex case, one may take $\kappa_f=1/\mu_f$ and $\kappa_g=1/\mu_g$, recovering
$\delta_0=\frac{K_y}{\alpha}(L_{fx}/\mu_f)^2$ and $\eta_0=\frac{K_z}{\beta}(L_{gx}/\mu_g)^2$.
Thus weaker inner geometry (larger $\kappa_f,\kappa_g$) increases the required time-scale separation through the same $C^2$ dependence.

\section{Nonunique inner solutions via penalty envelopes}\label{app:nonunique-penalty}

\subsection{What fails beyond \ref{step:S1}--\ref{step:S2} under nonuniqueness}\label{app:nonunique-why}
In the strongly convex/concave setting, the optimizer maps are single-valued and the envelope identity
$\nabla \Phi(x)=\nabla_x\psi(x,u^*(x))$ is immediate, which pins down a unique reference gradient at the inner optimum.
When the inner problems are merely convex/concave, optimizer sets are typically multi-valued, and two additional issues arise:

\begin{itemize}
\item \textbf{(Envelope nonsmoothness / selection dependence).}
Even if $\psi(x,\cdot)$ is convex, the value function $\Phi(x)=\min_u \psi(x,u)$ or $\Phi(x)=\max_u\psi(x,u)$ can be nonsmooth.
Then \ref{step:S1} cannot be justified by uniqueness, and the descent term must be formulated either in a nonsmooth sense
(e.g., $\operatorname{dist}(0,\partial \Phi(x))$) or via a smooth surrogate.

\item \textbf{(Deviation ambiguity).}
If $U^*(x)$ is set-valued, an expression of the form
$D(x,u;u^*)=\nabla_x\psi(x,u)-\nabla_x\psi(x,u^*)$ depends on the choice of $u^*\in U^*(x)$.
This is not a separate difficulty if $\Phi$ is differentiable: differentiability forces the set of partial gradients at optimizers to be a singleton
(see Lemma~\ref{lem:danskin-singleton} below), hence the deviation becomes unambiguous.
\end{itemize}

Our goal in this appendix is to keep the unified Lyapunov construction unchanged and only swap the two steps
\ref{step:S1}--\ref{step:S2}. To do so in the nonunique regime, we work with a smooth penalty envelope surrogate
and invoke a standard local sensitivity condition (solution-set Lipschitz continuity + active-set stability) that implies
differentiability of the surrogate value function even when inner solutions are not unique.

\subsection{Penalty-envelope reformulation for \ref{eq:problem2}}\label{app:nonunique-setup}
We specialize to \ref{eq:problem2} and rewrite the penalized outer value in a form that matches the penalty envelope studied in~\cite{kwon2023penalty}.
Let $Y$ be compact and define, for $\sigma>0$,
\begin{equation}\label{eq:penalty-hsigma}
h_\sigma(x,y)\ :=\ g(x,y)+\sigma f(x,y),\qquad
\ell(x,\sigma)\ :=\ \min_{y\in Y} h_\sigma(x,y),\qquad
T(x,\sigma)\ :=\ \arg\min_{y\in Y} h_\sigma(x,y).
\end{equation}
Define the penalty-envelope value function
\begin{equation}\label{eq:psi-sigma-def}
\psi_\sigma(x)\ :=\ \frac{\ell(x,\sigma)-\ell(x,0)}{\sigma},
\qquad\text{where}\qquad \ell(x,0)=\min_{y\in Y}g(x,y)=:g^*(x).
\end{equation}

\begin{lemma}\label{lem:psi-equals-Llambda}
Let $\lambda:=1/\sigma$. Then
\begin{equation}\label{eq:psi-sigma-equals-penalty}
\psi_{1/\lambda}(x)
=\min_{y\in Y}\Bigl(f(x,y)+\lambda g(x,y)\Bigr)-\lambda g^*(x)
=\min_{y\in Y}\Bigl(f(x,y)+\lambda\bigl(g(x,y)-g^*(x)\bigr)\Bigr)
=:\mathcal L_\lambda^*(x),
\end{equation}
which is precisely the penalized outer value used in \ref{eq:problem2}.
\end{lemma}

\begin{proof}
By $\sigma=1/\lambda$, we have $h_{1/\lambda}(x,y)=g(x,y)+\tfrac{1}{\lambda}f(x,y)=\tfrac{1}{\lambda}\bigl(f(x,y)+\lambda g(x,y)\bigr)$.
Thus $\ell(x,1/\lambda)=\tfrac{1}{\lambda}\min_{y\in Y}\bigl(f(x,y)+\lambda g(x,y)\bigr)$ and $\ell(x,0)=g^*(x)$.
Plugging into~\eqref{eq:psi-sigma-def} gives~\eqref{eq:psi-sigma-equals-penalty}.
\end{proof}

\paragraph{Scope of the appendix.}
All Lyapunov/ergodic statements below are for a fixed $\sigma>0$ (equivalently, a fixed penalty $\lambda=1/\sigma$),
i.e., stationarity for $\Phi=\psi_\sigma=\mathcal L_\lambda^*$.
Relating this surrogate stationarity to the limiting bilevel hyper-objective
$\psi(x):=\partial_\sigma \ell(x,\sigma)\vert_{\sigma=0+}$ requires additional compatibility at $\sigma=0+$ and is treated separately.

\subsection{A local sensitivity assumption}\label{app:nonunique-assumptions}
We now state assumptions in~\cite{kwon2023penalty} that imply differentiability of the penalty envelope
without assuming uniqueness of $T(x,\sigma)$.

\begin{assumption}\label{assump:SSC-nonunique}
Fix a region of interest $\mathcal X$ (e.g., the $x$-projection of a Lyapunov sublevel set).
Assume that for each $\sigma\in\{0,\bar\sigma\}$ (where $\bar\sigma:=1/\lambda$ is the fixed surrogate parameter),
the solution mapping $x\mapsto T(x,\sigma)$ is locally Lipschitz continuous on $\mathcal X$:
for every $x\in\mathcal X$ there exist $\delta>0$ and $L_T<\infty$ such that
\[
\operatorname{dist}\bigl(T(x_1,\sigma),T(x_2,\sigma)\bigr)\ \le\ L_T\,\|x_1-x_2\|
\qquad(\forall\,x_1,x_2\in B(x,\delta)\cap\mathcal X).
\]
Here $\operatorname{dist}(\cdot,\cdot)$ denotes a set-distance (e.g., Hausdorff distance), and $\operatorname{dist}(y,S):=\inf_{s\in S}\|y-s\|$ for point-to-set distance
and $B(x,\delta):=\{x'\in\mathbb R^{d_x}:\|x'-x\|<\delta\}$ denotes the (open) Euclidean ball.

\end{assumption}


Assumption~\ref{assump:SSC-nonunique} is local and does not require global uniqueness of $T(x,\sigma)$.
It imposes a Lipschitz stability of the solution set mapping with respect to $x$, which rules out abrupt set-valued
jumps of minimizers on the region of interest and, in turn, supports a smooth-envelope characterization of the surrogate
value function at fixed $\sigma$.

\subsection{Repairing \ref{step:S1}: smooth envelope without unique minimizers}\label{app:nonunique-S1}
We first provide a standard consequence of Danskin's theorem: differentiability implies a singleton set of optimizer-gradients.

\begin{lemma}\label{lem:danskin-singleton}
Let $U$ be compact and let $\psi:\mathbb{R}^{d_x}\times U\to\mathbb{R}$ be continuously differentiable in $x$ with $\nabla_x\psi$ continuous.
Define $\Phi(x):=\max_{u\in U}\psi(x,u)$ and $U^*(x):=\arg\max_{u\in U}\psi(x,u)$.
If $\Phi$ is differentiable at $x$, then
\[
\nabla \Phi(x)=\nabla_x\psi(x,u)\qquad(\forall\,u\in U^*(x)),
\]
and consequently the set $\{\nabla_x\psi(x,u):u\in U^*(x)\}$ is a singleton.
\end{lemma}

\begin{proof}
Danskin's theorem gives $\partial \Phi(x)=\mathrm{conv}\{\nabla_x\psi(x,u):u\in U^*(x)\}$.
If $\Phi$ is differentiable, then $\partial \Phi(x)=\{\nabla\Phi(x)\}$ is a singleton.
The convex hull of a set of vectors can be a singleton only if the set itself is contained in that singleton,
hence $\nabla_x\psi(x,u)$ is identical for all $u\in U^*(x)$ and equals $\nabla\Phi(x)$.
\end{proof}
We now apply this to the penalty envelope $\psi_\sigma$. 

Fix $\bar\sigma>0$ and recall
\[
\ell(x,\sigma):=\min_{y\in Y} h_\sigma(x,y),
\qquad
h_\sigma(x,y):=g(x,y)+\sigma f(x,y),
\qquad
\psi_\sigma(x):=\frac{\ell(x,\sigma)-\ell(x,0)}{\sigma}.
\]
Under Assumption~\ref{assump:SSC-nonunique}, Lemma~A.2 of~\cite{kwon2023penalty} implies that the value functions
$x\mapsto \ell(x,\bar\sigma)$ and $x\mapsto \ell(x,0)$ are differentiable on $\mathcal X$.
Moreover, for each $\sigma\in\{0,\bar\sigma\}$, differentiability of $\ell(\cdot,\sigma)$ together with
Lemma~\ref{lem:danskin-singleton} yields the smooth envelope identity
\[
\nabla_x \ell(x,\sigma)=\nabla_x h_\sigma(x,y^*)
\qquad(\forall\,y^*\in T(x,\sigma)).
\]
Consequently, $\psi_{\bar\sigma}$ is differentiable on $\mathcal X$ and satisfies
\begin{equation}\label{eq:grad-psi-sigma-envelope}
\nabla \psi_{\bar\sigma}(x)
=\frac{1}{\bar\sigma}\Bigl(\nabla_x h_{\bar\sigma}(x,y_{\bar\sigma}^*)-\nabla_x h_{0}(x,y_{0}^*)\Bigr)
=\nabla_x f(x,y_{\bar\sigma}^*)+\frac{1}{\bar\sigma}\Bigl(\nabla_x g(x,y_{\bar\sigma}^*)-\nabla_x g(x,y_{0}^*)\Bigr),
\end{equation}
where $y_{\bar\sigma}^*\in T(x,\bar\sigma)$ and $y_0^*\in T(x,0)$ are arbitrary minimizers.
By Lemma~\ref{lem:danskin-singleton}, the right-hand side of \eqref{eq:grad-psi-sigma-envelope} is independent of these choices.

\paragraph{Interpretation for \ref{eq:problem2}.}
With $\bar\sigma=1/\lambda$, Lemma~\ref{lem:psi-equals-Llambda} and~\eqref{eq:grad-psi-sigma-envelope} recover exactly the smooth-envelope module \ref{step:S1}
for $\Phi(x)=\mathcal L_\lambda^*(x)$, using the same first-order objects tracked by our flow
($y$ tracks $T(x,1/\lambda)$ and $z$ tracks $T(x,0)$).

\subsection{Repairing \ref{step:S2}: deviation-to-residual under nonunique minimizers}\label{app:nonunique-S2}
Once \ref{step:S1} is restored, the remaining obstruction is \ref{step:S2}: we must convert $x$-gradient deviations
(induced by imperfect inner tracking) into residuals dissipated by the inner dynamics.

For the $y$-subproblem at $\sigma=\bar\sigma$, define the $x$-gradient deviation as
\[
D_y(x,y)\ :=\ \nabla_x h_{\bar\sigma}(x,y)-\nabla_x h_{\bar\sigma}(x,y^*),
\qquad y^*\in T(x,\bar\sigma).
\]
By the smooth-envelope property established in Appendix~\ref{app:nonunique-S1}, the choice of $y^*\in T(x,\bar\sigma)$ does not matter, since $\nabla_x h_{\bar\sigma}(x,y^*)$ is identical for all minimizers in $T(x,\bar\sigma)$.

Similarly for the $z$-subproblem at $\sigma=0$,
\[
D_z(x,z)\ :=\ \nabla_x g(x,z)-\nabla_x g(x,z^*),\qquad z^*\in T(x,0).
\]

To recover the conversion $\|D_\cdot\|\le C_\cdot\|r_\cdot\|$ without strong convexity, we use the EB module
(Assumption~\ref{assump:inner-EB} in the main text), which is already formulated in terms of distance to a \emph{set} of minimizers
and hence is compatible with nonuniqueness.

\begin{proposition}\label{prop:EB-to-dev-nonunique}
Let $\psi(x,u)$ be differentiable and convex in $u$ on a region $\mathcal R$.
Assume $\nabla_x\psi(x,\cdot)$ is $L_{xu}$-Lipschitz in $u$ on $\mathcal R$ and the EB condition holds on $\mathcal R$:
$\operatorname{dist}(u,U^*(x))\le \kappa \|\nabla_u\psi(x,u)\|$, where $U^*(x)=\arg\min_u \psi(x,u)$.
Then for all $(x,u)\in\mathcal R$,
\[
\inf_{u^*\in U^*(x)}\|\nabla_x\psi(x,u)-\nabla_x\psi(x,u^*)\|
\ \le\ L_{xu}\,\kappa\,\|\nabla_u\psi(x,u)\|.
\]
\end{proposition}

\begin{proof}
Fix $(x,u)\in\mathcal R$. By Lipschitzness,
\[
\inf_{u^*\in U^*(x)}\|\nabla_x\psi(x,u)-\nabla_x\psi(x,u^*)\|
\ \le\ L_{xu}\,\inf_{u^*\in U^*(x)}\|u-u^*\|
\ =\ L_{xu}\,\operatorname{dist}(u,U^*(x)).
\]
Applying EB yields the claim.
\end{proof}

Applying Proposition~\ref{prop:EB-to-dev-nonunique} to $\psi=h_{\bar\sigma}$ and $\psi=g$ gives conversion constants
\[
\|D_y\|\le C_y\|r_y\|,\qquad C_y=L_{x y}^{(h_{\bar\sigma})}\kappa_{h_{\bar\sigma}},
\qquad\text{and}\qquad
\|D_z\|\le C_z\|r_z\|,\qquad C_z=L_{xz}^{(g)}\kappa_g,
\]
where $r_y=\nabla_y h_{\bar\sigma}(x,y)$ and $r_z=\nabla_y g(x,z)$ are the inner residuals.

\subsection{Ergodic stationarity for \ref{eq:problem2} under nonunique inner solutions}\label{app:nonunique-ergodic}
We show that the unified Lyapunov proof for \ref{eq:problem2} extends to nonunique lower-level solutions
without changing the Lyapunov construction.
The only modifications are the two modules \ref{step:S1}--\ref{step:S2}.

Fix $\lambda>0$ and let $\bar\sigma:=1/\lambda$.
By Lemma~\ref{lem:psi-equals-Llambda}, the penalized outer value in \ref{eq:problem2} admits the penalty-envelope form
$\Phi(x)=\mathcal L_\lambda^*(x)=\psi_{\bar\sigma}(x)$.
Under the local sensitivity assumption in Appendix~\ref{app:nonunique-S1}
(solution-set Lipschitz continuity and the existence of a regular stable optimizer at $\sigma\in\{0,\bar\sigma\}$),
the value functions $x\mapsto \ell(x,\bar\sigma)$ and $x\mapsto \ell(x,0)$ are differentiable on the trajectory region,
and hence $\Phi$ is differentiable and satisfies the smooth-envelope identity required by \ref{step:S1}.
Moreover, Lemma~\ref{lem:danskin-singleton} implies that the optimizer-gradient is unique at each $\sigma\in\{0,\bar\sigma\}$,
so the choice of minimizers $y^*\in T(x,\bar\sigma)$ and $z^*\in T(x,0)$ does not affect the envelope gradients.

To execute the absorption step, we control the $x$-gradient mismatches by residuals dissipated by the auxiliary dynamics.
Assume that on the Lyapunov sublevel set: (i) the subproblems $y\mapsto f(x,y)+\lambda g(x,y)$ and $z\mapsto g(x,z)$
are convex and differentiable with nonempty minimizer sets, (ii) the gradient error-bound (EB) conditions hold for both subproblems,
and (iii) $\nabla_x f$ and $\nabla_x g$ are Lipschitz in the inner variable.
Then Lemma~\ref{lem:deviation-EB} yields the deviation-to-residual conversion \ref{step:S2}:
\[
\|D_y\|\le C_y\|r_y\|,
\qquad
\|D_z\|\le C_z\|r_z\|,
\]
with constants determined by the cross-Lipschitz moduli and EB constants on the same sublevel set.
With \ref{step:S1}--\ref{step:S2} in place, the Lyapunov derivative admits the normal form
\eqref{eq:timescale-premise}, and Lemma~\ref{lem:timescale} absorbs the cross terms once the time scales exceed the
corresponding thresholds. Integrating $\dot W(t)\le -c_0\|\nabla \Phi(x(t))\|^2$ over $[0,T]$ yields an ergodic $O(1/T)$ stationarity bound of the same form as in the unique case, with constants depending on the EB/Lipschitz moduli through $C_y$ and $C_z$.

\paragraph{Extension to \ref{eq:P1} and \ref{eq:P31}.}
The same modular argument applies to any nested value function of the form
$\Phi(x)=\operatorname{opt}_{u\in U}\psi(x,u)$ (min or max), with a coupled flow that tracks the corresponding inner solutions.
In \ref{eq:P1}, $\Phi(x)=F(x)=\max_{y\in Y} f(x,y)$ and the coupled flow tracks the inner maximizers.
Nonuniqueness affects only \ref{step:S1}: we require a local sensitivity assumption ensuring differentiability of $F$
(and hence an envelope gradient that is independent of the maximizer selection), while \ref{step:S2} is provided by EB
for the inner maximization written as minimization of $-f$.
In \ref{eq:P31}, $\Phi(x)=f^*(x)-g^*(x)$ with $f^*(x)=\min_y f(x,y)$ and $g^*(x)=\min_z g(x,z)$,
and the coupled flow tracks minimizers for both inner problems.
Again, differentiability of $f^*$ and $g^*$ on the trajectory region supplies \ref{step:S1}, and EB for each inner
subproblem supplies \ref{step:S2} separately for $(D_y,r_y)$ and $(D_z,r_z)$.
Once these two interfaces are verified, the remainder of the proof (normal-form Lyapunov bound, time-scale absorption,
and ergodic integration) is identical across \ref{eq:P1}--\ref{eq:P31}.

\paragraph{Remark.}\label{rem:sigma-zero-link}
Our ergodic Lyapunov analysis is carried out for a fixed surrogate parameter $\sigma>0$
(equivalently, a fixed $\lambda=1/\sigma$), and therefore requires only the two modules \ref{step:S1}--\ref{step:S2}.
Linking surrogate stationarity to the limiting hyper-objective at $\sigma=0+$ is a separate sensitivity question.
As shown in~\cite{kwon2023penalty}, under additional regularity assumptions for the parametric lower-level program
(e.g., local stability of solution sets and active constraints) one can invoke a Theorem~3.1-type result to justify
the existence/compatibility of mixed derivatives at $\sigma=0+$ and obtain quantitative bias control as $\sigma\downarrow 0$.
Importantly, this $\sigma\downarrow 0$ discussion is modular: it affects only the surrogate-to-limit translation and
does not modify the Lyapunov differentiation or the time-scale absorption steps.

\section{Further Discussions on Trajectories}
\label{sec:track-bilevel0}
Leveraging our continuous-time Lyapunov analysis, we quantify the finite-time tracking error between the trajectory $x(t)$ generated by our coupled system \eqref{eq:flow-bilevel-unified} and the ideal hyper-gradient flow trajectory $w(t)$ governed by $\dot w(t) = -\nabla F(w(t))$ as in \eqref{eq:bi-gf}. The proofs of the results in this section are deferred to Appendix~\ref{sec:track-bilevel}.


\begin{theorem}\label{thm:finite-time-tracking}
Suppose Assumptions~\ref{ass:2}, \ref{ass:3}, and \ref{ass:4} hold, and let
$\lambda\ge \bar\lambda$. Let $w(t)$ and $(x(t),y(t),z(t))$ be solutions of
\eqref{eq:bi-gf} and \eqref{eq:flow-bilevel-unified}, respectively, with
$x(0)=w(0)$.
    For $\delta>\delta_0(\lambda)$ and $\eta>\eta_0$ and for any $T > 0$, we have
    \begin{align}
    \begin{aligned}
&\sup_{t\in[0,T]}\|x(t)-w(t)\|
        \;\le\;
        e^{L_F T} \Biggl[
        \underbrace{\frac{C_{\mathrm{br}} T}{\lambda}}_{\text{Bias}}
        \;+\;\underbrace{\sqrt{T\,V_0} \left( \frac{C_y(\lambda)}{\sqrt{\alpha(\delta - \delta_0(\lambda))}} + \frac{C_z}{\sqrt{\beta(\eta - \eta_0)}} \right)}_{\text{Tracking Error}}
        \Biggr],
    \end{aligned}    
    \end{align}
    where $V_0 := W(x(0),y(0),z(0))-\mathcal L_{\lambda,\inf}^{*}$ is the initial Lyapunov energy, $C_{\mathrm{br}}$ is given in Corollary~\ref{cor:bilevel-eps}, and $C_y(\lambda)$ and $C_z$ are the deviation constants from \eqref{eq:DyDz-conv-main}.
\end{theorem}
The following corollary characterizes the complexity required to achieve $\varepsilon$-accuracy. 
While increasing $\lambda$ reduces approximation bias, tracking precision is governed by the positive margins
$\delta-\delta_0(\lambda)$ and $\eta-\eta_0$ above the stability thresholds.
With the sharpened thresholds above, these stability thresholds do not exhibit the previous linear-in-$\lambda$ growth under $C_y(\lambda)=O(1)$.



\begin{corollary}\label{cor:eps-complexity}
    Fix $\varepsilon,  T>0$. To ensure that the trajectory approximation error satisfies $\sup_{t\in[0,T]}\|x(t)-w(t)\| \le \varepsilon$, it suffices to choose the penalty parameter $\lambda = \Theta\left(\frac{1}{\varepsilon}\right)$ and time-scales $\delta = O\left(\frac{1}{\varepsilon^2}\right), \eta = O\left(\frac{1}{\varepsilon^2}\right)$.
\end{corollary}

In contrast to the general smooth case, where the error bound grows exponentially with $T$, the error remains uniformly bounded over time if the hyper-objective satisfies a strong convexity assumption.


\begin{theorem}
\label{thm:uniform-tracking-sc}
    Under the same setting as in Theorem~\ref{thm:finite-time-tracking}, additionally assume that the hyper-objective $F(x)$ is $\mu_F$-strongly convex. Then, 
    \begin{align}
    \begin{aligned}
        &\|x(t)-w(t)\|
        \;\le\;
        \frac{C_{\mathrm{br}}}{\mu_F \lambda}
        \;+\;
      \sqrt{\frac{V_0}{\mu_F}} \left( \frac{C_y(\lambda)}{\sqrt{\alpha(\delta - \delta_0(\lambda))}} + \frac{C_z}{\sqrt{\beta(\eta - \eta_0)}} \right).
    \end{aligned}
    \end{align}
\end{theorem}

As a separate sanity check for the trajectory-tracking bound, Appendix~\ref{app:exp2} gives a one-dimensional quadratic example illustrating the bias--tracking-margin tradeoff in a setting with a closed-form ideal hyper-gradient flow.

\section{Proof of Theorems in Section~\ref{sec:track-bilevel0}}\label{sec:track-bilevel}

\begin{assumption}\label{ass:7}
Suppose that
\[
\|\nabla F(x)-\nabla F(x')\|
\le
L_F\|x-x'\|,
\qquad
\forall x,x'\in\mathcal X.
\]
\end{assumption}

\subsection{Proof of Theorem~\ref{thm:finite-time-tracking}}

\begin{proof}
    Let $e(t) := x(t) - w(t)$ be the trajectory error. We analyze the growth of $\|e(t)\|$ by differentiating its squared norm:
    \[
    \frac{1}{2}\frac{d}{dt}\|e(t)\|^2
    \;=\; \langle x(t)-w(t),\, \dot x(t) - \dot w(t) \rangle.
    \]
    Recall the dynamics of the ideal flow is $\dot w = -\nabla F(w)$. For the coupled system, as derived in \eqref{eq:xdot}, the dynamics of $x(t)$ satisfy:
    \[
    \dot x \;=\; -\nabla_x \mathcal L_\lambda^*(x) - D_y + \lambda D_z,
    \]
    where $D_y$ and $D_z$ are the deviation terms defined in Proposition~\ref{prop:bilevel-deviation}.
    Substituting these dynamics into the derivative yields:
    \begin{align*}
        \frac{1}{2}\frac{d}{dt}\|e\|^2
        \;&=\; \langle x-w, \, -\nabla_x \mathcal L_\lambda^*(x) - D_y + \lambda D_z + \nabla F(w) \rangle \\
        \;&=\; \langle x-w, \, \nabla F(w) - \nabla_x \mathcal L_\lambda^*(x) - D_y + \lambda D_z \rangle.
    \end{align*}

    We decompose the gradient difference term by adding and subtracting $\nabla F(x)$:
    \begin{align*}
        \nabla F(w) - \nabla_x \mathcal L_\lambda^*(x)
        \;=\; \underbrace{\bigl(\nabla F(w) - \nabla F(x)\bigr)}_{\text{Smoothness}}
        \;+\; \underbrace{\bigl(\nabla F(x) - \nabla_x \mathcal L_\lambda^*(x)\bigr)}_{\text{Approximation Bias}}.
    \end{align*}
    Using the Cauchy-Schwarz inequality, we bound each component:
    \begin{enumerate}[label=(\roman*)]
        \item \textbf{Smoothness:} Since $F$ is $L_F$-smooth for some $L_F>0$ (Assumption~\ref{ass:7}), 
        \[
        \langle x-w, \nabla F(w) - \nabla F(x) \rangle \;\le\; \|e\| \, \|\nabla F(w) - \nabla F(x)\| \;\le\; L_F \|e\|^2.
        \]
        \item \textbf{Bias:} By the Bridge Lemma (Lemma~\ref{lem:bridge}), $\|\nabla F(x) - \nabla_x \mathcal L_\lambda^*(x)\| \le C_{\mathrm{br}}/\lambda$. Thus,
        \[
        \langle x-w, \nabla F(x) - \nabla_x \mathcal L_\lambda^*(x) \rangle \;\le\; \|e\| \frac{C_{\mathrm{br}}}{\lambda}.
        \]
        \item \textbf{Tracking Errors:}
        \[
        \langle x-w, -D_y + \lambda D_z \rangle \;\le\; \|e\| \bigl( \|D_y\| + \lambda \|D_z\| \bigr).
        \]
    \end{enumerate}
    Combining these bounds results in a differential inequality for $\|e(t)\|$ (assuming $\|e(t)\|\neq 0$):
    \[
    \frac{d}{dt}\|e(t)\| \;\le\; L_F \|e(t)\| \;+\; \left( \frac{C_{\mathrm{br}}}{\lambda} \;+\; \|D_y(t)\| \;+\; \lambda \|D_z(t)\| \right).
    \]
    Applying Gr\"onwall's inequality with $e(0)=0$:
    \begin{equation}\label{eq:gronwall-step}
        \|e(t)\| \;\le\; \int_0^t e^{L_F(t-s)} \left( \frac{C_{\mathrm{br}}}{\lambda} + \|D_y(s)\| + \lambda \|D_z(s)\| \right) ds.
    \end{equation}
    Taking the supremum over $t \in [0, T]$:
    \begin{align}\label{eq:sup-bound-raw}
        \sup_{t\in[0,T]} \|e(t)\|
        \;\le\; e^{L_F T} \left( \frac{C_{\mathrm{br}} T}{\lambda} + \int_0^T \|D_y(s)\| ds + \lambda \int_0^T \|D_z(s)\| ds \right).
    \end{align}
    
    We bound the integrals using the Lyapunov dissipation. 
    Combining Lemma~\ref{lem:bilevel-residuals} with the deviation bounds $\|D_y\| \le \ C_y(\lambda)\,\|r_y\|$ and $\|D_z\| \le C_z \|r_z\|$, and applying Cauchy-Schwarz to the integral:
    \begin{align*}
        \int_0^T \|D_y(s)\| ds
        \;&\le\; \sqrt{T} \, C_y(\lambda) \left( \int_0^T \|r_y(s)\|^2 ds \right)^{1/2} \\
        \;&\le\; \sqrt{T} \, C_y(\lambda) \left( \frac{W(0) - W(T)}{c_y(\delta,\lambda)} \right)^{1/2}
        \;\le\; \sqrt{T V_0} \frac{C_y(\lambda)}{\sqrt{c_y(\delta,\lambda)}}.
    \end{align*}
    Similarly for the $z$-term:
    \begin{align*}
        \lambda \int_0^T \|D_z(s)\| ds
        \;&\le\; \lambda \sqrt{T} \, C_z \left( \int_0^T \|r_z(s)\|^2 ds \right)^{1/2} \leq  \lambda \sqrt{T V_0} \frac{C_z}{\sqrt{c_z(\eta,\lambda)}}.
    \end{align*}
    Plugging these integral bounds back into \eqref{eq:sup-bound-raw} yields the stated theorem.
\end{proof}

\subsection{Proof of Theorem~\ref{thm:uniform-tracking-sc}}

\begin{proof}
    Let $e(t) := x(t) - w(t)$. Differentiating the squared norm $\frac{1}{2}\|e(t)\|^2$:
    \[
    \frac{1}{2}\frac{d}{dt}\|e\|^2 \;=\; \langle x-w, \dot x - \dot w \rangle.
    \]
    Substituting the dynamics $\dot x = -\nabla_x \mathcal L_\lambda^*(x) - D_y + \lambda D_z$ and $\dot w = -\nabla F(w)$:
    \begin{align*}
        \frac{1}{2}\frac{d}{dt}\|e\|^2 
        \;&=\; \langle x-w, -\nabla F(w) - (-\nabla F(x)) \rangle + \langle x-w, -\nabla F(x) - (-\nabla_x \mathcal L_\lambda^*(x)) \rangle \\
        &\quad + \langle x-w, -D_y + \lambda D_z \rangle.
    \end{align*}
    We bound each term using the assumptions:
    \begin{enumerate}[label=(\roman*)]
        \item \textbf{Strong Convexity:} Since $F$ is $\mu_F$-strongly convex,
        \[
        \langle x-w, -(\nabla F(x) - \nabla F(w)) \rangle \;\le\; -\mu_F \|x-w\|^2 \;=\; -\mu_F \|e\|^2.
        \]
        \item \textbf{Bias:} Using the Bridge Lemma, $\|\nabla F(x) - \nabla_x \mathcal L_\lambda^*(x)\| \le \frac{C_{\mathrm{br}}}{\lambda}$, so:
        \[
        \langle x-w, \nabla F(w) - \nabla_x \mathcal L_\lambda^*(x) \rangle \;\le\; \|e\| \frac{C_{\mathrm{br}}}{\lambda}.
        \]
        \item \textbf{Tracking Errors:}
        \[
        \langle x-w, -D_y + \lambda D_z \rangle \;\le\; \|e\| (\|D_y\| + \lambda \|D_z\|).
        \]
    \end{enumerate}
    Combining these yields the differential inequality for the norm $\|e(t)\|$:
    \[
    \frac{d}{dt}\|e(t)\| \;\le\; -\mu_F \|e(t)\| \;+\; \underbrace{\left( \frac{C_{\mathrm{br}}}{\lambda} + \|D_y(t)\| + \lambda \|D_z(t)\| \right)}_{u(t)}.
    \]
    Solving this ODE with $e(0)=0$:
    \[
    \|e(t)\| \;\le\; \int_0^t e^{-\mu_F(t-s)} u(s) \, ds.
    \]
    We now bound the integral of each component of $u(s)$.
    
    1. \textbf{Bias Term:}
    \[
    \int_0^t e^{-\mu_F(t-s)} \frac{C_{\mathrm{br}}}{\lambda} \, ds \;=\; \frac{C_{\mathrm{br}}}{\lambda} \frac{1 - e^{-\mu_F t}}{\mu_F} \;\le\; \frac{C_{\mathrm{br}}}{\mu_F \lambda}.
    \]
    
    2. \textbf{Tracking Terms:} We apply the Cauchy-Schwarz inequality to the convolution integral. For any non-negative function $\phi(s)$:
    \begin{align*}
        \left( \int_0^t e^{-\mu_F(t-s)} \phi(s) \, ds \right)^2 
        \;&\le\; \left( \int_0^t e^{-\mu_F(t-s)} \, ds \right) \left( \int_0^t e^{-\mu_F(t-s)} \phi(s)^2 \, ds \right) \\
        \;&\le\; \frac{1}{\mu_F} \int_0^t \phi(s)^2 \, ds.
    \end{align*}
    Therefore, $\int_0^t e^{-\mu_F(t-s)} \phi(s) \, ds \le \frac{1}{\sqrt{\mu_F}} \sqrt{\int_0^t \phi(s)^2 ds}$.
    
    Using the Lyapunov dissipation bounds $\int_0^\infty \|D_y\|^2 ds \le \frac{C_y^2 V_0}{c_y}$ and $\int_0^\infty \lambda^2 \|D_z\|^2 ds \le \frac{C_z^2 V_0}{c_z}$ (where simplified margin notation is used), we get:
    \begin{align*}
        \int_0^t e^{-\mu_F(t-s)} \|D_y(s)\| \, ds 
        \;&\le\; \frac{1}{\sqrt{\mu_F}} \left( \frac{C_y(\lambda)^2 V_0}{c_y(\delta,\lambda)} \right)^{1/2} 
        \;=\; \sqrt{\frac{V_0}{\mu_F}} \frac{C_y(\lambda)}{\sqrt{c_y(\delta,\lambda)}}, \\
        \int_0^t e^{-\mu_F(t-s)} \lambda \|D_z(s)\| \, ds 
        \;&\le\; \frac{1}{\sqrt{\mu_F}} \left( \lambda^2 \frac{C_z^2 V_0}{c_z(\eta,\lambda)} \right)^{1/2} 
        \;=\; \sqrt{\frac{V_0}{\mu_F}} \frac{C_z}{\sqrt{\beta(\eta - \eta_0)}}.
    \end{align*}
    Note that for the $z$-term, substituting $c_z = \beta \lambda^2 (\eta-\eta_0)$ cancels the $\lambda$ factor, resulting in the expression in the theorem. Summing these bounds completes the proof.
\end{proof}

\section{One-dimensional trajectory-tracking sanity check}\label{app:exp2}

\paragraph{Goal.}
This experiment is a sanity check for the finite-time trajectory-tracking bound in
Section~\ref{sec:track-bilevel0}. It is not intended as a discrete-time threshold
validation. Instead, it illustrates the bias--tracking-margin tradeoff in a
one-dimensional quadratic instance where the ideal hyper-gradient flow is
available in closed form. Increasing $\lambda$ reduces the surrogate bias, while
accurate tracking requires positive margins
$\delta-\delta_0(\lambda)$ and $\eta-\eta_0$ above the stability
thresholds.

\paragraph{Problem instance.}
We use the 1D quadratic bilevel problem
\[
g(x,y)=\tfrac12(y-ax)^2,\qquad f(x,y)=\tfrac12(x-by)^2,
\]
with $(a,b)=(0.5,1.0)$ and $\mu=1$.
Then the lower-level optimizer is $y^*(x)=ax$ and the true hyper-objective is
\[
F(x)=f\bigl(x,y^*(x)\bigr)=\tfrac12(1-ab)^2x^2.
\]
Hence the ideal hyper-gradient flow $\dot w=-\nabla F(w)$ admits the closed form
\[
w(t)=w(0)\exp\!\bigl(-(1-ab)^2 t\bigr),\qquad (1-ab)^2=0.25,
\]
which we use as the reference trajectory.

\begin{figure}[t!]
\centering 
\includegraphics[width=1\textwidth]{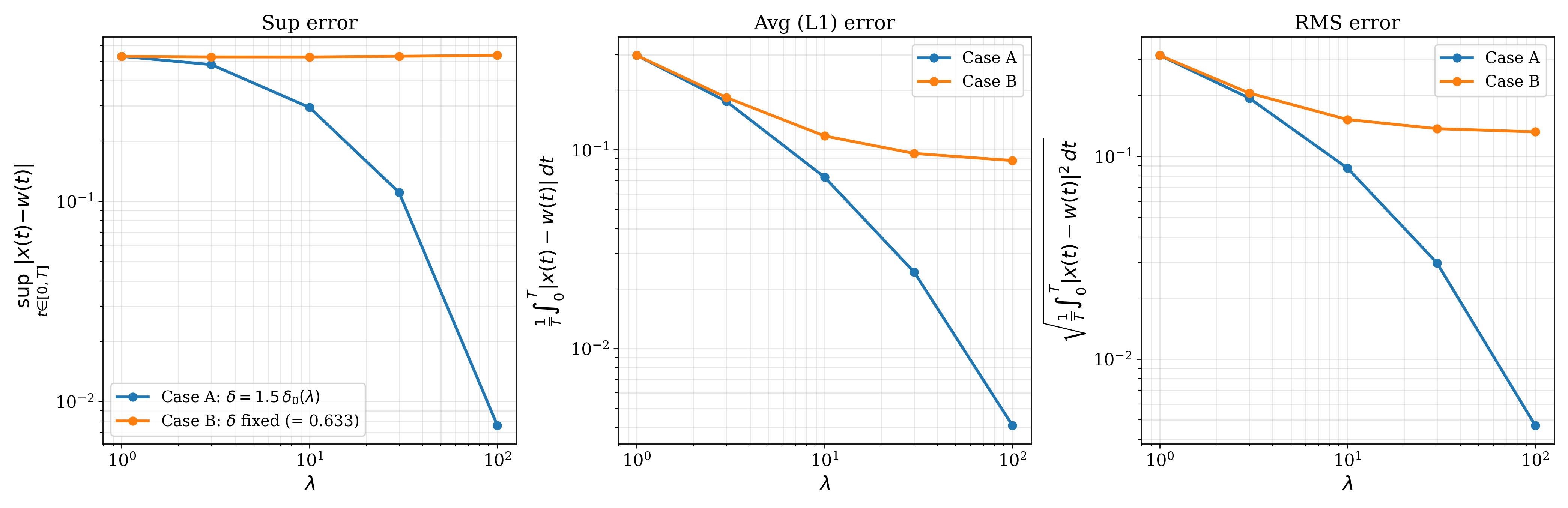}
\caption{Tracking error between the single-loop lifted-penalty trajectory $x(t)$ and the ideal hyper-gradient flow $w(t)$ as a function of $\lambda$.
Case~A uses margin-aware scaling $\delta=1.5\,\delta_0(\lambda)$ with $\eta=8\,\eta_0$; Case~B fixes $\delta$ (set at $\lambda=1$) while keeping the same $\eta$.}
\label{fig:exp2-metrics}
\end{figure}


\paragraph{Thresholds and cases.}
For each $\lambda$, we compute the explicit thresholds $\delta_0(\lambda)$ and $\eta_0$ using the closed-form constants of this instance (as in \eqref{eq:eta0-bilevel}, specialized to the quadratic setting).
We compare:
\begin{itemize}
\item \textbf{Case A (threshold-aware).} $\delta=1.5\,\delta_0(\lambda)$ and $\eta=8\,\eta_0$.
\item \textbf{Case B (under-separated baseline).} $\delta$ is fixed as $\delta=1.5\,\delta_0(1)$ for all $\lambda$, while using the same $\eta$ as in Case~A for the corresponding $\lambda$ (so that differences are attributable to $\delta$).
\end{itemize}
This diagnostic compares a margin-aware choice of the time-scale parameters with a fixed-time-scale baseline, illustrating that increasing $\lambda$ alone does not guarantee accurate tracking unless the positive margins above the stability thresholds are also controlled.

\paragraph{Metrics.}
We report three tracking error metrics between $x(t)$ and the ideal trajectory $w(t)$:
\[
\text{Sup:}\quad \sup_{t\in[0,T]}|x(t)-w(t)|,
\qquad
\text{Avg:}\quad \frac1T\int_0^T|x(t)-w(t)|\,dt,
\qquad
\text{RMS:}\quad \Big(\frac1T\int_0^T|x(t)-w(t)|^2dt\Big)^{1/2}.
\]
The averaged metrics align with the ergodic nature of our Lyapunov guarantees, while the sup metric provides a conservative worst-case diagnostic.

\paragraph{Results and interpretation.}
Figure~\ref{fig:exp2-metrics} shows a clear separation between the two regimes.
With margin-aware time scales, increasing $\lambda$ reduces the bias and improves
tracking. With a fixed $y$-tracking scale, the benefit of increasing $\lambda$ is
limited because the positive tracking margin is not maintained. This supports
the trajectory-tracking interpretation of Theorem~\ref{thm:finite-time-tracking}:
penalty bias and residual tracking error must be controlled together.

\section{Euler discretization analogue}\label{app:euler-discretization}

This appendix records a deterministic Euler analogue of the continuous-time
Lyapunov normal form. The purpose is not to give a complete discrete-time
complexity theorem for every possible discretization, but to make explicit how
the continuous-time S1--S3 structure transfers to a one-step Lyapunov analysis.

\begin{proposition}\label{prop:euler-one-step}
Let $u:=(x,y,z)$ and let $G(u):=(\dot x,\dot y,\dot z)$ denote the right-hand
side of a coupled flow. Let $\mathcal R$ be a region containing the relevant
Euler segments
\[
\{u_k+s hG(u_k):\; s\in[0,1],\ k\ge 0\}.
\]
Assume that $\nabla W$ is $L_W$-Lipschitz on $\mathcal R$.

Suppose that, on $\mathcal R$, the continuous-time Lyapunov normal form satisfies
\begin{equation}\label{eq:euler-ct-normal-form}
\dot W(u)
\le
-c_1\|\nabla\Phi(x)\|^2
-c_2\|r_y\|^2
-c_3\|r_z\|^2,
\end{equation}
for some $c_1,c_2,c_3>0$. Suppose further that the vector field is controlled by
the same quantities: there exist constants $M_1,M_2,M_3\ge 0$ such that
\begin{equation}\label{eq:euler-vector-field-control}
\|G(u)\|^2
\le
M_1\|\nabla\Phi(x)\|^2
+
M_2\|r_y\|^2
+
M_3\|r_z\|^2
\qquad (u\in\mathcal R).
\end{equation}
Consider the explicit Euler discretization
\[
u_{k+1}=u_k+hG(u_k),
\]
or equivalently
\[
x_{k+1}=x_k+h\dot x_k,\qquad
y_{k+1}=y_k+h\dot y_k,\qquad
z_{k+1}=z_k+h\dot z_k.
\]
If
\begin{equation}\label{eq:euler-h-condition}
0<h\le
\min_{i:\,M_i>0}
\frac{c_i}{L_WM_i},
\end{equation}
where $(c_1,c_2,c_3)$ and $(M_1,M_2,M_3)$ are matched componentwise, then
\begin{equation}\label{eq:euler-one-step-decrease}
W(u_{k+1})-W(u_k)
\le
-\frac{h c_1}{2}\|\nabla\Phi(x_k)\|^2
-\frac{h c_2}{2}\|r_{y,k}\|^2
-\frac{h c_3}{2}\|r_{z,k}\|^2.
\end{equation}
Consequently, if $W_{\inf}:=\inf_{u\in\mathcal R}W(u)>-\infty$, then
\begin{equation}\label{eq:euler-ergodic-bound}
\frac1N\sum_{k=0}^{N-1}\|\nabla\Phi(x_k)\|^2
\le
\frac{2(W(u_0)-W_{\inf})}{c_1hN},
\end{equation}
and analogous averaged bounds hold for the residual terms.
\end{proposition}

\begin{proof}
By the $L_W$-Lipschitz continuity of $\nabla W$ on $\mathcal R$, the standard
one-step descent estimate gives
\[
W(u_{k+1})-W(u_k)
\le
\langle \nabla W(u_k),u_{k+1}-u_k\rangle
+
\frac{L_W}{2}\|u_{k+1}-u_k\|^2.
\]
Using $u_{k+1}=u_k+hG(u_k)$, we obtain
\[
W(u_{k+1})-W(u_k)
\le
h\langle \nabla W(u_k),G(u_k)\rangle
+
\frac{L_W}{2}h^2\|G(u_k)\|^2.
\]
The first-order term is exactly the continuous-time derivative of $W$ along the
flow:
\[
\langle \nabla W(u_k),G(u_k)\rangle=\dot W(u_k).
\]
Therefore, by \eqref{eq:euler-ct-normal-form} and
\eqref{eq:euler-vector-field-control},
\[
\begin{aligned}
W(u_{k+1})-W(u_k)
\le\;&
-hc_1\|\nabla\Phi(x_k)\|^2
-hc_2\|r_{y,k}\|^2
-hc_3\|r_{z,k}\|^2\\
&+
\frac{L_W}{2}h^2
\Bigl(
M_1\|\nabla\Phi(x_k)\|^2
+
M_2\|r_{y,k}\|^2
+
M_3\|r_{z,k}\|^2
\Bigr).
\end{aligned}
\]
Equivalently,
\[
\begin{aligned}
W(u_{k+1})-W(u_k)
\le\;&
-h\Bigl(c_1-\frac{L_WhM_1}{2}\Bigr)\|\nabla\Phi(x_k)\|^2\\
&-
h\Bigl(c_2-\frac{L_WhM_2}{2}\Bigr)\|r_{y,k}\|^2\\
&-
h\Bigl(c_3-\frac{L_WhM_3}{2}\Bigr)\|r_{z,k}\|^2.
\end{aligned}
\]
The step-size condition \eqref{eq:euler-h-condition} implies
\[
c_i-\frac{L_WhM_i}{2}\ge \frac{c_i}{2},
\qquad i=1,2,3,
\]
with the convention that no restriction is needed when $M_i=0$. This proves
\eqref{eq:euler-one-step-decrease}.

Summing \eqref{eq:euler-one-step-decrease} over $k=0,\ldots,N-1$ gives
\[
\frac{hc_1}{2}
\sum_{k=0}^{N-1}\|\nabla\Phi(x_k)\|^2
\le
W(u_0)-W(u_N)
\le
W(u_0)-W_{\inf}.
\]
Dividing by $hc_1N/2$ yields \eqref{eq:euler-ergodic-bound}. The averaged
residual bounds follow identically by retaining the corresponding residual
terms in \eqref{eq:euler-one-step-decrease}.
\end{proof}

\begin{remark}
Proposition~\ref{prop:euler-one-step} separates the discretization issue from
the geometry issue. The Euler step only adds the Taylor remainder
$O(h^2\|G(u_k)\|^2)$, while the first-order term is the same Lyapunov
dissipation obtained in continuous time. Consequently, the geometry-dependent
part of the argument is still the deviation-to-residual interface:
\[
\|D_y\|\le C_y\|r_y\|,
\qquad
\|D_z\|\le C_z\|r_z\|.
\]
Under strong convexity, these constants come from Lemma~\ref{lem:deviation}.
Under weaker geometry, they are replaced by the EB/sensitivity constants in
Section~\ref{sec:mere} and Appendix~\ref{app:convex-inner}. Once this interface
is available, both the continuous-time absorption thresholds and the Euler
one-step step-size restriction follow from the same normal-form quantities
$(c_i,M_i)$. Thus the unified template identifies which parts of a discrete-time
analysis are structural and reusable, and which parts change as the inner
geometry weakens.
\end{remark} 

\subsection{Forward-Euler hypercleaning diagnostics}
\label{app:euler-hypercleaning}

\paragraph{Setup for Figure~\ref{fig:euler-heatmap}.}
Figure~\ref{fig:euler-heatmap} reports forward-Euler diagnostics for the
lifted-penalty dynamics on Wine and Digits. We use a standard hypercleaning
setup with one outer variable $w_i$ per training example and sample weight
$s_i(w_i)=(1+\exp(-w_i))^{-1}$. The inner model is a multiclass logistic
classifier with parameter $\theta$. With $30\%$ uniformly corrupted training
labels, the lower-level and validation objectives are
\[
g(w,\theta)
=
\frac{1}{n_{\mathrm{tr}}}
\sum_{i=1}^{n_{\mathrm{tr}}}
s_i(w_i)\ell(\theta;x_i,\widetilde y_i)
+
\frac{\rho_{\rm reg}}{2}\|\theta\|^2,
\qquad
f(\theta)
=
\frac{1}{n_{\mathrm{val}}}
\sum_{j=1}^{n_{\mathrm{val}}}
\ell(\theta;x_j^{\mathrm{val}},y_j^{\mathrm{val}}).
\]
We use the full Wine dataset and a subset of $80$ Digits examples, both with a
$60/40$ train/validation split.

\paragraph{Euler dynamics and thresholds.}
We instantiate the lifted penalty
\[
E_\lambda(w,\theta,\zeta)
=
f(\theta)+\lambda\bigl(g(w,\theta)-g(w,\zeta)\bigr),
\]
where $\theta$ tracks the penalized inner minimizer and $\zeta$ tracks the
lower-level minimizer associated with $g^*(w)$. We discretize the corresponding
flow by forward Euler and scan $(\delta,\eta)$ on logarithmic grids at fixed
$\lambda=10$. For the symmetric Lyapunov choice $\alpha=\beta=\tfrac12$, the
thresholds reduce to
\[
\delta_0=\frac12 C^2,
\qquad
\eta_0=2C^2.
\]
Since $f$ does not depend directly on $w$, we use the empirical conversion scale
\[
\widehat C=\frac{\widehat L_{w\theta}}{\widehat\mu_g},
\]
where $\widehat L_{w\theta}$ is estimated by finite-difference power iteration
and $\widehat\mu_g$ is the empirical curvature of the regularized lower-level
objective. Thus
\[
\widehat\delta_0=\frac12\widehat C^2,
\qquad
\widehat\eta_0=2\widehat C^2.
\]

\paragraph{Metric and interpretation.}
For each pair $(\delta,\eta)$, we report
\[
\frac1N\sum_{k=0}^{N-1}
\left\|
\nabla_w E_\lambda(w_k,\theta_k,\zeta_k)
\right\|^2,
\]
displayed on a base-$10$ logarithmic scale. This is the forward-Euler analogue
of the outer stationarity quantity controlled by the continuous-time Lyapunov
normal form. In both datasets and in both geometry regimes, the region
$\delta\ge\widehat\delta_0,\eta\ge\widehat\eta_0$ gives the smallest average
stationarity. Moreover, reducing $\rho_{\rm reg}$ from $0.05$ to $0.02$ weakens
the empirical inner geometry by roughly a factor of $2.5$. Since the theory
predicts
\[
\widehat\delta_0,\widehat\eta_0 \propto \widehat C^2
=
\left(\frac{\widehat L_{w\theta}}{\widehat\mu_g}\right)^2,
\]
the thresholds should increase by about $2.5^2=6.25$. The observed
deterioration factors are approximately $5.8$--$6.0$ across Wine and Digits,
which is consistent with the predicted $C^2$ scaling up to estimation noise.
Thus the experiment supports the modular S1--S3 interpretation: weakening the
inner geometry changes the deviation-to-residual conversion constant and shifts
the threshold scale, while the Lyapunov construction and absorption rule remain
unchanged. Forward Euler adds a one-step discretization remainder and hence
requires a stable absolute step size, but the relative time-scale boundary
predicted by the continuous-time normal form remains visible in the discretized
dynamics.

\section{min-min-max}\label{sec:minminmax}
\noindent\textbf{Why include min--min--max in the unified framework.}
\noindent
Although the main text focuses on the bilevel case, the Lyapunov/time-scale template in
Section~\ref{sec:unified} applies verbatim to other hierarchical games as well.
The min--min--max problem \eqref{eq:minminmax} is a clean illustrative instance:
the inner problems in $y$ and $z$ decouple into two smooth envelopes $f^*(x)$ and $g^*(x)$,
so the outer objective is the difference $\mathcal L(x)=f^*(x)-g^*(x)$.
Accordingly, the proof follows the same three steps as before:
\textbf{[S1]} smooth envelope gradients,
\textbf{[S2]} deviation-to-residual conversion for $(D_y,D_z)$,
and \textbf{[S3]} absorption of cross terms via sufficiently large time-scale parameters $(\delta,\eta)$.

\begin{assumption}\label{ass:6} Suppose that:
\begin{enumerate}
\item There exist constants $\mu_f,\mu_g>0$ such that for every $x\in\mathbb R^{d_x}$,
the mappings $y\mapsto f(x,y)$ and $z\mapsto g(x,z)$ are $\mu_f$-- and $\mu_g$--strongly convex,
respectively.

\item There exist constants $L_{fx},L_{gx}\ge 0$ such that for all $x$ and all $y,y',z,z'$,
\[
\|\nabla_x f(x,y)-\nabla_x f(x,y')\|\le L_{fx}\|y-y'\|,
\qquad
\|\nabla_x g(x,z)-\nabla_x g(x,z')\|\le L_{gx}\|z-z'\|.
\]

\item Let $f^*(x):=\min_y f(x,y)$, $g^*(x):=\min_z g(x,z)$, and define
\[
\mathcal L(x):=f^*(x)-g^*(x).
\]
Assume $\mathcal L_{\inf}:=\inf_x \mathcal L(x)>-\infty$.
\end{enumerate}

\end{assumption}

\paragraph{Problem Setting.}

We consider
\begin{equation}\label{eq:minminmax}
\min_{x\in\mathbb R^{d_x}}\ \min_{y\in\mathbb R^{d_y}}\ \max_{z\in\mathbb R^{d_z}}
E(x,y,z),
\qquad
E(x,y,z):=f(x,y)-g(x,z).
\end{equation}
For each fixed $x$, 
\[
\min_y\max_z E(x,y,z)
=\min_y\bigl(f(x,y)-\min_z g(x,z)\bigr)
=f^*(x)-g^*(x)
=\mathcal L(x).
\]

\noindent We study the gradient flow associated with \eqref{eq:minminmax}:
\begin{equation}\label{eq:flow}
\begin{aligned}
\dot x(t) &= -\nabla_x E\bigl(x(t),y(t),z(t)\bigr),\\
\dot y(t) &= -\delta\,\nabla_y E\bigl(x(t),y(t),z(t)\bigr),\\
\dot z(t) &= \eta\,\nabla_z E\bigl(x(t),y(t),z(t)\bigr),
\end{aligned}
\end{equation}
with $\delta>0$ and $\eta>0$.
Since $\nabla_y E=\nabla_y f$ and $\nabla_z E=-\nabla_z g$, \eqref{eq:flow} becomes
\begin{equation}\label{eq:flow-expanded}
\begin{aligned}
\dot x &= -\nabla_x f(x,y) + \nabla_x g(x,z),\\
\dot y &= -\delta\,\nabla_y f(x,y),\\
\dot z &= -\eta\,\nabla_z g(x,z).
\end{aligned}
\end{equation}

\noindent Define the nonnegative gaps
\[
H_f(x,y):=f(x,y)-f^*(x)\ (\ge 0),
\qquad
H_g(x,z):=g(x,z)-g^*(x)\ (\ge 0),
\]
and for $\alpha,\beta\in(0,1)$ define
\begin{equation}\label{eq:W-def}
W(x,y,z):=\mathcal L(x)+\alpha\,H_f(x,y)+\beta\,H_g(x,z).
\end{equation}

\begin{theorem}\label{thm:ergodic}
Assume Assumption~\ref{ass:6} and fix $\alpha,\beta\in(0,1)$.
Then there exist constants $c_1>0$ and thresholds $\delta_0>0$, $\eta_0>0$
such that for any $\delta\ge \delta_0$ and $\eta\ge \eta_0$, along any solution
$(x(t),y(t),z(t))$ of \eqref{eq:flow-expanded} we have for a.e.\ $t\ge 0$,
\begin{equation}\label{eq:Wdot-final1}
\frac{d}{dt}W(x(t),y(t),z(t))
\ \le\
-\,c_1\,\|\nabla \mathcal L(x(t))\|^2.
\end{equation}
Consequently, for every $T>0$, if $t\sim\mathrm{Unif}[0,T]$, then
\begin{equation}\label{eq:ergodic-rate}
\mathbb E_t\bigl[\|\nabla \mathcal L(x(t))\|^2\bigr]
\ \le\
\frac{W(x(0),y(0),z(0))-\mathcal L_{\inf}}{c_1\,T}.
\end{equation}
\end{theorem}

\begin{proof}
\textbf{[S1: Smooth envelopes.]}
\noindent By \textnormal{(A1)}, for each $x$ the minimizers
\[
y^*(x)\in\arg\min_{y} f(x,y),\qquad z^*(x)\in\arg\min_{z} g(x,z)
\]
are unique and satisfy 
\[
\nabla_y f\bigl(x,y^*(x)\bigr)=\mathbf 0,
\qquad
\nabla_z g\bigl(x,z^*(x)\bigr)=\mathbf 0.
\]
Define
\[
\begin{aligned}
f^*(x) &:= \min_{y} f(x,y) \;=\; f\bigl(x,y^*(x)\bigr),\\
g^*(x) &:= \min_{z} g(x,z) \;=\; g\bigl(x,z^*(x)\bigr),\\
\mathcal L(x) &:= f^*(x) - g^*(x).
\end{aligned}
\]
We compute
\[
\nabla f^*(x)
=\nabla_x f\bigl(x,y^*(x)\bigr)
+\bigl(Dy^*(x)\bigr)^{\top}\nabla_y f\bigl(x,y^*(x)\bigr)
=\nabla_x f\bigl(x,y^*(x)\bigr),
\]
and similarly,
\[
\nabla g^*(x)
=\nabla_x g\bigl(x,z^*(x)\bigr)
+\bigl(Dz^*(x)\bigr)^{\top}\nabla_z g\bigl(x,z^*(x)\bigr)
=\nabla_x g\bigl(x,z^*(x)\bigr).
\]
Therefore,
\begin{equation}\label{eq:L-grad}
\nabla \mathcal L(x)=\nabla_x f\bigl(x,y^*(x)\bigr)-\nabla_x g\bigl(x,z^*(x)\bigr).
\end{equation}

\smallskip
\noindent Set
\[
A:=\nabla \mathcal L(x)\in\mathbb R^{d_x},
\qquad
r_y:=\nabla_y f(x,y)\in\mathbb R^{d_y},
\qquad
r_z:=\nabla_z g(x,z)\in\mathbb R^{d_y},
\]
and define the \emph{$x$-gradient deviations} (mismatches)
\[
D_y:=\nabla_x f(x,y)-\nabla_x f\bigl(x,y^*(x)\bigr),
\qquad
D_z:=\nabla_x g(x,z)-\nabla_x g\bigl(x,z^*(x)\bigr).
\]

\medskip
\textbf{[S2: Deviation-to-residual conversion.]}
\noindent Fix $x$. Since $y\mapsto f(x,y)$ is $\mu_f$--strongly convex, for all $y$,
\[
\big\langle \nabla_y f(x,y)-\nabla_y f\bigl(x,y^*(x)\bigr),\,y-y^*(x)\big\rangle
\ \ge\ \mu_f\|y-y^*(x)\|^2.
\]
Using $\nabla_y f(x,y^*(x))=\mathbf 0$, we obtain
\[
\mu_f\|y-y^*(x)\|^2\le \langle r_y,\,y-y^*(x)\rangle
\le \|r_y\|\,\|y-y^*(x)\|,
\]
hence
\begin{equation}\label{eq:track-y}
\|y-y^*(x)\|\ \le\ \frac{1}{\mu_f}\,\|r_y\|.
\end{equation}
Similarly, since $z\mapsto g(x,z)$ is $\mu_g$--strongly convex,
\begin{equation}\label{eq:track-z}
\|z-z^*(x)\|\ \le\ \frac{1}{\mu_g}\,\|r_z\|.
\end{equation}

\smallskip
\noindent By Assumption~\ref{ass:6},
\[
\|D_y\|
=\bigl\|\nabla_x f(x,y)-\nabla_x f(x,y^*(x))\bigr\|
\le L_{fx}\|y-y^*(x)\|
\le \frac{L_{fx}}{\mu_f}\|r_y\|
=:C_y\|r_y\|,
\]
and
\[
\|D_z\|
=\bigl\|\nabla_x g(x,z)-\nabla_x g(x,z^*(x))\bigr\|
\le L_{gx}\|z-z^*(x)\|
\le \frac{L_{gx}}{\mu_g}\|r_z\|
=:C_z\|r_z\|.
\]
Thus,
\begin{equation}\label{eq:DyDz-bound}
\|D_y\|\le C_y\|r_y\|,
\qquad
\|D_z\|\le C_z\|r_z\|.
\end{equation}

\medskip
\textbf{[S3: Lyapunov derivative and absorption via time-scale separation.]}
\noindent Recall $E(x,y,z)=f(x,y)-g(x,z)$. Then
\[
\nabla_x E=\nabla_x f(x,y)-\nabla_x g(x,z),
\qquad
\nabla_y E=\nabla_y f(x,y)=r_y,
\qquad
\nabla_z E=-\nabla_z g(x,z)=-r_z.
\]
Hence the scaled gradient flow \eqref{eq:flow-expanded} can be written as
\[
\dot y=-\delta r_y,
\qquad
\dot z=-\eta r_z.
\]
Moreover, by \eqref{eq:L-grad} and the definitions of $D_y,D_z$,
\[
\nabla_x f(x,y)=\nabla_x f\bigl(x,y^*(x)\bigr)+D_y,
\qquad
\nabla_x g(x,z)=\nabla_x g\bigl(x,z^*(x)\bigr)+D_z,
\]
so
\[
\dot x
=-\nabla_x f(x,y)+\nabla_x g(x,z)
=-(\nabla_x f(x,y^*(x))+D_y)+(\nabla_x g(x,z^*(x))+D_z)
=-A-D_y+D_z.
\]
Thus
\begin{equation}\label{eq:xdot1}
\dot x=-A-D_y+D_z.
\end{equation}

\medskip
\noindent Since $\mathcal L$ depends only on $x$,
\begin{align}
\dot{\mathcal L}(x(t))
&=\big\langle \nabla\mathcal L(x),\,\dot x\big\rangle
=\langle A,-A-D_y+D_z\rangle
\nonumber\\
&=-\|A\|^2-\langle A,D_y\rangle+\langle A,D_z\rangle.
\label{eq:Ldot}
\end{align}

\smallskip
\noindent Define the nonnegative gaps (unified notation)
\[
H_y(x,y):=f(x,y)-f^*(x),\qquad
H_z(x,z):=g(x,z)-g^*(x).
\]
Using $\nabla f^*(x)=\nabla_x f(x,y^*(x))$ and $\nabla g^*(x)=\nabla_x g(x,z^*(x))$, we have
\[
\nabla_x H_y(x,y)=D_y,\quad \nabla_y H_y(x,y)=r_y,
\qquad
\nabla_x H_z(x,z)=D_z,\quad \nabla_z H_z(x,z)=r_z.
\]
Therefore, along the flow,
\begin{align}
\dot H_y
&=\langle D_y,\dot x\rangle+\langle r_y,\dot y\rangle
=\langle D_y,-A-D_y+D_z\rangle-\delta\|r_y\|^2
\nonumber\\
&=-\langle D_y,A\rangle-\|D_y\|^2+\langle D_y,D_z\rangle-\delta\|r_y\|^2,
\label{eq:Hy-dot}
\end{align}
and
\begin{align}
\dot H_z
&=\langle D_z,\dot x\rangle+\langle r_z,\dot z\rangle
=\langle D_z,-A-D_y+D_z\rangle-\eta\|r_z\|^2
\nonumber\\
&=-\langle D_z,A\rangle-\langle D_z,D_y\rangle+\|D_z\|^2-\eta\|r_z\|^2.
\label{eq:Hz-dot}
\end{align}

\medskip
\noindent Recall $W=\mathcal L+\alpha H_y+\beta H_z$. Combining \eqref{eq:Ldot}, \eqref{eq:Hy-dot}, and \eqref{eq:Hz-dot},
\begin{align}
\dot W
&=
-\|A\|^2
-(1+\alpha)\langle A,D_y\rangle
+(1-\beta)\langle A,D_z\rangle
-\alpha\|D_y\|^2
+(\alpha-\beta)\langle D_y,D_z\rangle
+\beta\|D_z\|^2
-\alpha\delta\|r_y\|^2
-\beta\eta\|r_z\|^2.
\label{eq:Wdot-raw}
\end{align}
For an upper bound, we may drop the nonpositive term $-\alpha\|D_y\|^2\le 0$.

\medskip
\noindent Fix $\varepsilon_1,\varepsilon_2>0$. By Cauchy--Schwarz and Young's inequality,
\[
-(1+\alpha)\langle A,D_y\rangle
\le (1+\alpha)\Bigl(\frac{\varepsilon_1}{2}\|A\|^2+\frac{1}{2\varepsilon_1}\|D_y\|^2\Bigr),
\]
\[
(1-\beta)\langle A,D_z\rangle
\le (1-\beta)\Bigl(\frac{\varepsilon_2}{2}\|A\|^2+\frac{1}{2\varepsilon_2}\|D_z\|^2\Bigr),
\]
\[
(\alpha-\beta)\langle D_y,D_z\rangle
\le \frac{|\alpha-\beta|}{2}\bigl(\|D_y\|^2+\|D_z\|^2\bigr).
\]
Substituting into \eqref{eq:Wdot-raw} yields
\begin{equation}\label{eq:Wdot-mid}
\dot W
\le
-\,c_1\|A\|^2
+K_y\|D_y\|^2
+K_z\|D_z\|^2
-\alpha\delta\|r_y\|^2
-\beta\eta\|r_z\|^2,
\end{equation}
where
\[
c_1:=1-\frac{(1+\alpha)\varepsilon_1}{2}-\frac{(1-\beta)\varepsilon_2}{2},
\qquad
K_y:=\frac{1+\alpha}{2\varepsilon_1}+\frac{|\alpha-\beta|}{2},
\qquad
K_z:=\frac{1-\beta}{2\varepsilon_2}+\frac{|\alpha-\beta|}{2}+\beta.
\]
Choose $\varepsilon_1,\varepsilon_2>0$ sufficiently small so that $c_1>0$.

\medskip
\noindent Using \eqref{eq:DyDz-bound} in \eqref{eq:Wdot-mid},
\[
K_y\|D_y\|^2-\alpha\delta\|r_y\|^2
\le (K_yC_y^2-\alpha\delta)\|r_y\|^2,
\qquad
K_z\|D_z\|^2-\beta\eta\|r_z\|^2
\le (K_zC_z^2-\beta\eta)\|r_z\|^2.
\]
Hence
\[
\dot W
\le -c_1\|A\|^2 + (K_yC_y^2-\alpha\delta)\|r_y\|^2 + (K_zC_z^2-\beta\eta)\|r_z\|^2.
\]
Define thresholds
\[
\delta_0:=\frac{K_yC_y^2}{\alpha},
\qquad
\eta_0:=\frac{K_zC_z^2}{\beta}.
\]
If $\delta\ge \delta_0$ and $\eta\ge \eta_0$, then the last two coefficients are nonpositive, and we obtain
\begin{equation}\label{eq:Wdot-final2}
\dot W \ \le\ -c_1\|A\|^2 \ =\ -c_1\|\nabla\mathcal L(x)\|^2,
\end{equation}
as desired.

\medskip
\noindent Since $H_y\ge 0$ and $H_z\ge 0$, we have $W(x,y,z)\ge \mathcal L(x)\ge \mathcal L_{\inf}$.
Integrating \eqref{eq:Wdot-final2} over $[0,T]$ yields
\[
c_1\int_0^T \|\nabla \mathcal L(x(t))\|^2\,dt
\le W(x(0),y(0),z(0)) - W(x(T),y(T),z(T))
\le W(x(0),y(0),z(0))-\mathcal L_{\inf}.
\]
If $t\sim\mathrm{Unif}[0,T]$, dividing by $T$ gives \eqref{eq:ergodic-rate}.
\end{proof}

\end{document}